\pdfoutput=1
\documentclass[microtype]{gtpart}
\usepackage{graphicx}
\usepackage[mathscr]{eucal}
\usepackage{amssymb}
\usepackage[dvipsnames]{xcolor}
\usepackage{enumerate, cite}
\usepackage[margin=1.5in]{geometry}
\usepackage{hyperref}

\newcommand{\br}{\mathbb{R}}
\newcommand{\bc}{\mathbb C}
\newcommand{\bz}{\mathbb Z}
\newcommand{\bn}{\mathbb N}
\newcommand{\bq}{\mathbb Q}

\newcommand{\bS}{\mathbb S}

\newcommand{\cE}{\mathcal E}

\newcommand{\ep}{\epsilon}
\newcommand{\vp}{\varphi}

\newcommand{\ssm}{\smallsetminus}

\newcommand{\sE}{\mathscr E}
\newcommand{\Enp}{\sE_{\mathrm{np}}}
\newcommand{\Eno}{\sE_{\mathrm{no}}}

\newcommand{\sU}{\mathscr U}
\newcommand{\sV}{\mathscr V}
\newcommand{\sW}{\mathscr W}

\DeclareMathOperator{\mcg}{MCG}

\DeclareMathOperator{\Ends}{\cE}

\newcommand{\Homeo}{H}

\DeclareMathOperator{\supp}{supp}

\providecommand{\co}{\colon\thinspace}

\newtheorem{Thm}{Theorem}[section]
\newtheorem{Thm*}{Theorem}
\newtheorem{Prop}[Thm]{Proposition}
\newtheorem{Lem}[Thm]{Lemma}
\newtheorem{Cor}[Thm]{Corollary}
\newtheorem{Cor*}[Thm*]{Corollary}
\newtheorem{Conj}[Thm]{Conjecture}

\newtheorem*{MainThm1}{Theorem~\ref{thm:main1}}
\newtheorem*{MainThm2}{Theorem~\ref{thm:main2}}
\newtheorem*{MainThm3}{Theorem~\ref{thm:commutator}}
\newtheorem*{MainThm4}{Theorem~\ref{thm:cb}}
\newtheorem*{MainThm5}{Theorem~\ref{thm:rokhlin}}
\newtheorem*{MainThm6}{Theorem~\ref{thm:distortion}}

\theoremstyle{definition}
\newtheorem{Def}[Thm]{Definition}

\newtheorem*{Ex*}{Examples}

\newtheorem*{Rem*}{Remark}
\newtheorem{Problem}[Thm]{Problem}

\numberwithin{equation}{section}

\title{Homeomorphism groups of self-similar 2-manifolds}

\author{Nicholas~G.~Vlamis}
\address{Department of Mathematics \\ CUNY Graduate Center \\ New York, NY 10016, and \newline Department of Mathematics \\ CUNY Queens College \\ Flushing, NY 11367}
\email{nvlamis@gc.cuny.edu}

\begin{document}  

\begin{abstract}
The class of self-similar 2-manifolds consists of manifolds exhibiting a type of homogeneity akin to the 2-sphere and the Cantor set, and includes both the 2-sphere and the 2-sphere with a Cantor set removed.
This chapter aims to provide a narrative thread between recent results on the structure of homeomorphism groups/mapping class groups of self-similar 2-manifolds, and also connections to classical structural results on the homeomorphism group of the 2-sphere and the Cantor set. 
In order to do this, we provide a survey of recent results, an exposition on classical results about homeomorphism groups, provide a treatment of the structure of stable sets, and prove extensions/strengthenings  of the recent results surveyed.

Of particular note, we establish the following theorems:
(1) A characterization of homeomorphisms of (orientable) perfectly self-similar 2-manifolds that normally generate the group of (orientation-preserving) homeomorphisms---a strengthening of a result of Malestein--Tao.
(2) The homeomorphism group of a perfectly self-similar 2-manifold is strongly distorted---an extension of a result of Calegari--Freedman for spheres. 
(3) The homeomorphism group of a perfectly tame 2-manifold is Steinhaus, and hence has the automatic continuity property---an extension of a result of Mann in dimension two---providing the first examples of homeomorphism groups of infinite-genus 2-manifolds with the Steinhaus property.
\end{abstract}

\maketitle

\vspace{-0.5in}


\tableofcontents
\clearpage

\section{Introduction}

Over the last decade, there has been growing interest in the study of mapping class groups of infinite-type 2-manifolds, often referred to as \emph{big mapping class groups}\index{mapping class group!big}.  
Recently, Mann and Rafi's work \cite{MannLarge} has provided a framework for partitioning infinite-type 2-manifolds into several natural classes, which has led to a flurry of activity. 
We focus here on a particular subclass of 2-manifolds that exhibit a high degree of homogeneity at infinity, which leads to their homeomorphism groups and mapping class groups having ``low complexity'' from the viewpoint of their geometry and subgroup structure (as compared to other 2-manifolds). 
The recent progress in this subclass connects back to the work of Anderson, which has been adapted to this setting.

In 1958, Anderson \cite{AndersonAlgebraic} introduced a simple yet ingenious technique to establish the algebraic simplicity of a class of transformation groups, including the group of orientation-preserving self-homeomorphisms of the 2-sphere and the group of self-homeomorphisms of the Cantor set.
In 1947, Ulam and von Neumann announced this same result for the 2-sphere, but it was never published. 
It is worth noting that Anderson's method fails in the context of diffeomorphism groups, and the analogous theorem in this setting is due to Thurston \cite{ThurstonSimple}.
Thurston did not publish his proof, but the details are presented in \cite{Banyaga}.
A short proof has been given by Mann (see \cite{MannShort} and the set of notes \cite[Section~2]{MannMinicourse}).

\subsection*{The main object of study: self-similar 2-manifolds}

The subclass of 2-manifolds we restrict ourselves to are known as self-similar 2-manifolds, and they exhibit a high-degree of homogeneity at infinity, which is meant to be captured in the following definition. 

\begin{Def}[Self-similar 2-manifold\index{2-manifold!self-similar}]
A subset \( K \) of a manifold \( M \) is \emph{displaceable}\index{displaceable} if there exists a homeomorphism \( f\co M \to M \) such that \( f(K) \cap K = \varnothing \). 
A 2-manifold \( M \) is \emph{self-similar} if every proper compact subset is displaceable and, for any separating simple closed curve \( c \) in \( M \), there is a component of \( M \ssm c \) that is homeomorphic to \( M \) with a point removed. 
We further partition the class of self-similar 2-manifolds into two subclasses: 
A self-similar 2-manifold \( M \) is \emph{perfectly self-similar}\index{2-manifold!self-similar!perfectly} if \( M \#M \) is homeomorphic to \( M \); otherwise, it is \emph{uniquely self-similar}\index{2-manifold!self-similar!uniquely}. 
\end{Def}

It follows from the Jordan--Schoenflies theorem that the 2-sphere and the plane are self-similar; moreover, the 2-sphere is perfectly self-similar, but the plane is not. 
These two examples are the only finite-type self-similar 2-manifolds; all other cases must necessarily have infinitely many ends or infinite genus. 
The simplest infinite-type examples of self-similar 2-manifolds are the 2-sphere with a Cantor set removed, the plane with an infinite discrete set removed, and  the orientable one-ended infinite-genus 2-manifold. 
Among these examples, only the 2-sphere with the Cantor set removed is perfectly self-similar. 
We encourage the reader to keep these examples in mind throughout the chapter, as they capture the essence of self-similar 2-manifolds.

More generally, a self-similar 2-manifold can be obtained by removing a compact countable subset of Cantor--Bendixson degree one from the 2-sphere; this procedure results in an uncountably family of pairwise non-homeomorphic  self-similar 2-manifolds (all of which are uniquely self-similar).  
If one imagines blowing up the highest rank accumulation point in the previous examples to a Cantor set, an uncountable collection of perfectly self-similar 2-manifolds is obtained.

The original definition of a self-similar 2-manifold was given by Malestein--Tao in \cite{MalesteinSelf} and was framed in terms of the topological ends of the manifold, which in turn used the terminology introduced by Mann--Rafi in \cite{MannLarge}. 
The advantage of the definition above is that we do not rely on the notion of an end of a manifold, and hence the definition clearly includes the 2-sphere, which allows us to see how the recent developments in the theory of homeomorphism groups of (perfectly) self-similar 2-manifolds are a natural extension of results regarding the homeomorphism group of the 2-sphere.

\subsection*{Goals and outline}

Many of the recent articles investigating the structure of mapping class groups of self-similar 2-manifolds have been released in a brief period of time, and hence, this chapter benefits by cross-pollinating ideas between the articles. 
For instance, we are able to give a stronger version of a result of Malestein--Tao \cite{MalesteinSelf} using ideas from Calegari--Chen \cite{CalegariNormal} (see Theorem~\ref{thm:main2}). 
This is an example of our first goal: strengthen and/or extend the current results in the literature.

Our second goal is to generalize recent results about mapping class groups to the setting of homeomorphism groups (e.g., Theorem~\ref{thm:rokhlin}), and in the process, provide the reader with a background in the topology of homeomorphism groups.
This provides a narrative thread between recent progress in the study of big mapping class groups with the history of studying homeomorphism groups. 

The next goal, realized through various corollaries, is to exhibit a clear connection between the homeomorphism groups of 2-manifolds and the homeomorphism groups of second-countable Stone spaces, which factors through the classification of surfaces.  
In particular, many results about the homeomorphism group of the 2-sphere generalize naturally to homeomorphisms groups of perfectly self-similar 2-manifolds, including the 2-sphere with a Cantor set removed. 
We can therefore see how a single proof can simultaneously establish properties of the homeomorphisms group of the 2-sphere and the Cantor set. 

Our final goal is to collect, in one place, useful tools and results for understanding and working with end spaces of 2-manifolds and their stable sets.
These tools are used throughout the chapter in various arguments.

With these goals in mind, the chapter has been written in a nearly self-contained fashion relying in most places only on a minimal set of prerequisites, including point-set topology, algebraic topology (e.g., \cite[Chapters~0,  1, and 2]{Hatcher}), and the basics of topological group theory (e.g., \cite[Chapter 1]{Zippin}).
However, this is not to imply that the chapter makes for easy reading, as many ideas are introduced quickly and treated concisely. 
In the places where providing all the required details would send us too far afield, we provide appropriate references for the reader. 
We also point the reader to \cite{AramayonaVlamis}, a chapter in an earlier volume of this series written by Aramayona and the author giving an introduction to big mapping class groups.

The outline of the paper is as follows:
\begin{itemize}
\item Section~\ref{sec:overview} provides a survey of recent results on homeomorphism groups and mapping class groups of self-similar 2-manifolds.
\item Section~\ref{appendix:topology} gives an overview of the key topological properties of 2-manifolds, including their classification up to homeomorphism. 
\item Section~\ref{sec:stable-sets} introduces tools for studying end spaces of 2-manifolds, and in particular, develops the structure theory of stable sets.
\item Section~\ref{sec:Anderson} introduces the notion of Freudenthal subsurfaces and introduces a version of Anderson's method that is well suited to our purposes.
\item Section~\ref{sec:homeomorphism-groups} is an exposition on fundamental properties of homeomorphism groups of 2-manifolds.
\item Section~\ref{sec:self-similar} gives several equivalent notions of self-similarity. 
\item Sections~7--12 provide proofs of the various theorems surveyed in Section~\ref{sec:overview}.

\end{itemize}

\subsection*{Acknowledgements}

The proof of Theorem~\ref{thm:mcg_defined} is due to Mladen Bestvina, as communicated to the author by Jing Tao.
The author thanks George Domat for pointing out that automatic continuity implies the existence of a unique Polish topology; Kathryn Mann for many helpful conversations;  Diana Hubbard and Chaitanya Tappu for comments on an earlier draft; and the anonymous reader for their comments. 
The author is supported by NSF Grant DMS--2212922 and PSC-CUNY Award \#~64129-00 52.
The author was also partially supported by the Faculty Fellowship Publication Program and thanks his cohort for reading a portion of a draft of this chapter. 


\section{Overview of results}
\label{sec:overview}

Let us begin by setting basic notation.
In what follows, a \emph{2-manifold}\index{2-manifold} (resp.,~\emph{surface}\index{surface}) refers to a connected second-countable Hausdorff topological space in which every point admits an open neighborhood homeomorphic to the plane \( \br^2 \) (resp., a closed half-plane, that is,  \( \{ (x,y) \in \br^2 : y \geq 0\} \)).
In particular, a 2-manifold is a surface whose boundary is empty.
Note that every 2-manifold is metrizable. 
A 2-manifold is of \emph{finite type}\index{2-manifold!finite-type} if it can be realized as the interior of a compact surface; otherwise, it is of \emph{infinite type}\index{2-manifold!infinite-type}. 
A \emph{simple closed curve}\index{simple closed curve} on a surface \( S \) is the image of a topological embedding of the circle in \( S \); it is \emph{separating}\index{simple closed curve!separating} if its complement is disconnected. 

Given a 2-manifold \( M \), let \( \mathrm{Homeo}(M) \) be the \emph{homeomorphism group}\index{group}\index{group!homeomorphism} of \( M \), that is, the group of homeomorphisms \( M \to M \).
Equipped with the compact-open topology\index{compact-open topology}, \( \mathrm{Homeo}(M) \) is a \emph{Polish group}\index{group!Polish}, that is, a separable and completely metrizable topological group (see Section~\ref{sec:homeomorphism-groups} for more details).
If \( M \) is orientable,  \( \mathrm{Homeo}^+(M) \) denotes the subgroup of \( \mathrm{Homeo}(M) \) consisting of orientation-preserving homeomorphisms, which, as a closed subgroup, is also Polish. 

\textbf{Notation.}
Let \( M \) be a 2-manifold. 
If \( M \) is orientable, we set \( H(M) = \mathrm{Homeo}^+(M) \); otherwise, we set \( H(M) = \mathrm{Homeo}(M) \).

The \emph{mapping class group}\index{mapping class group} of a 2-manifold \( M \), denoted \( \mcg(M) \), is the quotient group \( H(M) / H_0(M) \), where \( H_0(M) \) is the connected component of the identity in \( H(M) \), or equivalently, \( \mcg(M) \) is the group of isotopy classes of elements in \( H(M) \); this equivalence is established in Section~\ref{sec:homeomorphism-groups}. 
Since \( H_0(M) \) is a closed subgroup of \( H(M) \), it follows that \( \mcg(M) \) is a Polish topological group when equipped with the quotient topology  (see \cite[Theorem~8.19]{KechrisClassical}). 

We write \( 2^\bn \) to represent the countable product \( \prod_{n\in\bn} \{0,1\} \) equipped with the product topology, where \( \{0,1\} \) is given the discrete topology. 
With this topology, \( 2^\bn \) is homeomorphic to the Cantor set. 

In what follows, we survey and motivate a number of recent results.
In the following sections, we provide proofs of all the numbered theorems and corollaries presented in this section.
For each of the numbered theorems, we are able to give some extension of the original result in the literature, such as extending to a larger classes of surfaces and including non-orientable surfaces, or generalizing from mapping class groups to homeomorphism groups. 
Despite these extensions, the proofs generally follow the existing proofs in the literature with appropriate adaptation to our setting.


\subsection{Normal generation and purity}
\label{subsec:normal}

As noted in the introduction, Anderson \cite{AndersonAlgebraic} proved that \( H(\bS^2) \) and \( \mathrm{Homeo}(2^\bn) \) are \emph{simple groups}\index{group!simple} (i.e., they only have two normal subgroups).  
The first theorem we introduce is a natural generalization of these results to the setting of perfectly self-similar 2-manifolds (and their end spaces) and strengthens a result of Malestein--Tao \cite[Theorem~A]{MalesteinSelf}. 
The theorem relies on what we call \emph{Anderson's method} (see Section~\ref{sec:Anderson} and, more specifically, Proposition~\ref{prop:main2}).
The version of Anderson's method presented is suited to our purposes and is a straightforward generalization of Calegari--Chen's \cite[Lemma~1]{CalegariNormal}, which itself is a variation of Anderson's original technique presented in \cite{AndersonAlgebraic}.

The \emph{commutator}\index{commutator} of two elements \( g \) and \( h \) in a group is the group element \( [g,h]=ghg^{-1}h^{-1} \). 
A group \( G \) is \emph{perfect}\index{group!perfect} if it is equal to its commutator subgroup  \( [G,G]  \)---the subgroup generated by the commutators in \( G \).
A group \( G \) is \emph{uniformly perfect}\index{group!perfect!uniformly} if there exits \( p \in \bn \) such that every element of \( G \)  can be written as a product of \( p \) commutators; the minimal such \( p \) is called the \emph{commutator width\index{commutator width} of \( G \)}. 
An element \( g \) of a group \emph{normally generates}\index{normal generator} the group if every element can be expressed as a product of conjugates of \( g \).

To state the theorem, we need the notion of a half-space: a closed subset \( D \) of a self-similar 2-manifold \( M \) is a \emph{half-space}\index{half-space} if \( D \) is a subsurface of \( M \) with connected compact boundary and the closure of \( M \ssm D \) is homeomorphic to \( D \) (see Definition~\ref{def:half-space}). 
For example, in the 2-sphere, every embedded closed disk is a half-space.

\begin{MainThm2} 
Let \( M \) be a perfectly self-similar 2-manifold, and let \( G  \) denote either \( H(M) \) or \( \mcg(M) \).
\begin{enumerate}[(1)]
\item \( G \) is uniformly perfect and its commutator width is at most two. 
\item
If \( g \in G \) displaces a half-space of \( M \), then every element of \( G \)  can be expressed as a product of  at most eight conjugates of \( g \) and \( g^{-1} \). In particular, \( G \) is normally generated by \( g \).
\item
If \( n \in \bn \) such that \( n \geq 2 \), then every element of \( G \) can be expressed as the product of at most eight elements of order \( n \).
\end{enumerate}
\end{MainThm2}

In the case of the 2-sphere minus a Cantor set, Theorem~\ref{thm:main2}(1) was established in a blog post of Calegari, and Theorem~\ref{thm:main2}(2) was established by Calegari--Chen \cite{CalegariNormal}. 
Under the additional assumption of orientability, Malestein--Tao \cite{MalesteinSelf} established Theorem~\ref{thm:main2}(1) (with commutator width three instead of two) and Theorem~\ref{thm:main2}(3) restricted to \( n = 2 \). 
As already noted, the 2-sphere is perfectly self-similar, and moreover, every nontrivial homeomorphism must displace a disk (i.e.,~a half-space).
Therefore, by Theorem~\ref{thm:main2}, every nontrivial element of \( H(\bS^2) \) normally generates the group, yielding:

\begin{Cor}[Anderson \cite{AndersonAlgebraic}]
\label{cor:simple}
\( H(\bS^2) \) is a simple group.
\qed
\end{Cor}

A \emph{quasimorphism}\index{quasimorphism} of a group \( G \) is a function \( q \co G \to \br \) such that there exists \( d \in \br_+ \) satisfying \( |q(gh)-q(g)-q(h)| < d \) for all \( g,h \in G \) (see \cite[Chapter~2]{scl} for an introduction to quasimorphisms and their properties, and \cite{KotschikWhat} for an informal introduction to quasimorphisms and additional references).
By Theorem~\ref{thm:main2},  if \( M \) is a perfectly self-similar 2-manifold, then every element of \( H(M) \) is the product of at most two commutators; this implies that every quasimorphism of \( H(M) \) is bounded. 

\begin{Cor}
\label{cor:quasimorphism}
If \( M \) is a perfectly self-similar 2-manifold, then every quasimorphism of \( H(M) \) (resp., \( \mcg(M) \)) is bounded. \qed
\end{Cor}

Informally, Corollary~\ref{cor:quasimorphism} implies that \( H(M) \) admits no ``interesting'' geometric actions; we will not be more explicit here, as we will prove a stronger geometric restriction in the following subsection. 

Next, let us explain how Theorem~\ref{thm:main2} is also a generalization of Anderson's theorem stating that \( \mathrm{Homeo}(2^\bn) \) is simple.
An end of a perfectly self-similar 2-manifold is \emph{maximally stable}\index{stable!maximally} if it corresponds to a nested sequence of half-spaces (see Definition~\ref{def:half-space} for the general definition of a maximally stable end\index{end!maximally stable} and see Section~\ref{appendix:topology} for an introduction to ends). 
Recall that a 2-manifold is \emph{planar}\index{2-manifold!planar} if it is homeomorphic to an open subset of \( \br^2 \). 
Let \( E \) be the end space of a planar perfectly self-similar 2-manifold \( M \).  
From Richards's work on the classification of surfaces \cite{RichardsClassification}, we obtain the following two facts (see Section~\ref{appendix:topology} for more details): 
(1) The end space of a planar 2-manifold is a second-countable Stone space\footnote{A \emph{Stone space}\index{Stone space} is a compact zero-dimensional Hausdorff topological space.}, and conversely, every second-countable Stone space can be realized as the end space of a planar 2-manifold.
(2) The action of \( H(M) \) on \( E \) induces an epimorphism of \( H(M) \) onto \( \mathrm{Homeo}(E) \).

We say that a second-countable Stone space \( E \)  is \emph{(perfectly/uniquely) self-similar}\index{Stone space!self-similar} if it is realized as the end space of a planar (perfectly/uniquely) self-similar 2-manifold \( M \) and that \( e \in E \) is maximally stable if it is a maximally stable end of \( M \)  (intrinsic definitions can be given with the language introduced in Section~\ref{sec:stable-sets}). 
Together with Theorem~\ref{thm:main2}, the above facts yield the following corollary:

\begin{Cor}
Let \( E \) be a perfectly self-similar second-countable Stone space.
Every element of \( \mathrm{Homeo}(E) \) that permutes the maximally stable points nontrivially normally generates \( \mathrm{Homeo}(E) \).
In particular, \( \mathrm{Homeo}(2^\bn) \) is simple. 
\qed
\end{Cor}

As another corollary, we give an extension of the Purity Theorem of Calegari--Chen \cite{CalegariNormal}.
The Purity Theorem below completely characterizes the normal generators of \( H(M) \) for a non-compact perfectly self-similar 2-manifold, and therefore, is the natural generalization of Anderson's theorem, Corollary~\ref{cor:simple}, to the non-compact setting.

\begin{Cor}[Purity Theorem]
\label{cor:countable}
Let \( M \) be a non-compact perfectly self-similar 2-manifold, and let \( G \) denote either \( H(M) \) or \( \mcg(M) \).
Then, every proper normal subgroup of \( G \) is contained in the kernel of the action of \( G \) on the set of maximally stable ends of \( M \). 
In particular, an element of \( H(M) \) normally generates \( H(M) \) if and only if it induces a nontrivial permutation of the maximally stable ends of \( M \).
Moreover, \( G \) does not contain a proper normal subgroup of countable index. 
\end{Cor}

We give a proof of Corollary~\ref{cor:countable} in Section~\ref{sec:normal}.
Let \( K \subset \br^2 \subset \bS^2 \) be a Cantor space (that is, homeomorphic to \( 2^\bn \)). 
The Purity Theorem was first proved by Calegari--Chen \cite{CalegariNormal} in the special case of  \( \bS^2 \ssm K \) (they also give a statement for finite-type surfaces with a Cantor space removed). 
To foreshadow a subsection below, we note that in this special case, the author \cite{VlamisThree} had previously established the weaker result that neither  \( \Homeo(\bS^2\ssm K) \) nor \( \Homeo(\br^2\ssm K) \) contain a proper normal subgroup of countable index.  
The proof given in \cite{VlamisThree} is an example application of the automatic continuity of these homeomorphism groups (established by Mann \cite{MannAutomatic2}), which we explore in a later subsection. 

In a separable topological group, any open subgroup must have countable index.
Therefore, Corollary~\ref{cor:simple} (in the case of the 2-sphere) and Corollary~\ref{cor:countable} (for all other perfectly self-similar 2-manifolds) yield the following corollary, which we will later generalize to the setting of all self-similar 2-manifolds.

\begin{Cor}
\label{cor:open}
If \( M \) is a perfectly self-similar 2-manifold, then neither \( H(M) \) nor \( \mcg(M) \) contain a proper open normal subgroup.
\qed
\end{Cor}

As a last corollary to Theorem~\ref{thm:main2}, we give an extension of \cite[Lemma~2.5]{CalegariNormal}.

\begin{Cor}
\label{cor:torsion}
Let \( M \) be a planar perfectly self-similar 2-manifold. 
Every proper normal subgroup of \( H(M) \) and \( \mcg(M) \) is torsion free. 
In particular, every torsion element of \( H(M) \) (resp., \( \mcg(M) \)) is a normal generator. 
\end{Cor}

The proof of Corollary~\ref{cor:torsion} is given in Section~\ref{sec:normal}.
The planarity assumption in Corollary~\ref{cor:torsion} is necessary: by the work of Aougab, Patel, and the author (see \cite[Lemmas~3.5~\&~3.6]{AougabIsometry}), the pure mapping class group of an orientable infinite-genus  self-similar 2-manifold with no planar ends contains an isomorphic copy of every countable group; in particular, it contains torsion elements that do not normally generate.
Here, the \emph{pure mapping class group}\index{mapping class group!pure} refers to the kernel of the action of the mapping class group on the end space.


\subsection{Strong distortion}
\label{subsec:sd}

In the previous section, we showed the extent to which the algebraic simplicity of \( H(\mathbb S^2) \) can be generalized to the setting of perfectly self-similar 2-manifolds. 
Here, we will give a strong notion of geometric simplicity that is exhibited by \( H(\mathbb S^2) \),  which can  also be seen in homeomorphism groups of perfectly self-similar 2-manifolds.

A group \( G \) is \emph{strongly distorted}\index{group!strongly distorted} if there exists \( m \in \bn \) and a sequence \( \{w_n\}_{n\in\bn} \subset \bn \)  such that, given any sequence \( \{g_n\}_{n\in\bn} \) in \( G \), there exists a set \( S \subset G \) of cardinality at most \( m \) such that \( g_n \in S^{w_n} \) for every \( n \in \bn \). 
One immediate consequence is that every non-identity element of \( G \) is distorted in the standard sense, that is, if \( g \in G \) is not the identity, then there exists a finitely generated subgroup \( \Gamma \) of \( G \) such that the word length of \( g^n \) in \( \Gamma \) grows sublinearly in \( n \) (equivalently, the homomorphism from \( \bz \) to \( \Gamma \) given by \( n \mapsto g^n \) fails to be a quasi-isometric embedding). 
Quasi-morphisms are often used to detect undistorted group elements, and so Corollary~\ref{cor:quasimorphism} suggests that the elements of \( H(M) \) are distorted, and indeed:

\begin{MainThm6}
If \( M \) is a perfectly self-similar 2-manifold, then \( H(M) \) and \( \mcg(M) \) are strongly distorted. 
\end{MainThm6}

The proof of Theorem~\ref{thm:distortion} is an adaptation of Calegari--Freedman's proof of \cite[Theorem~C]{CalegariDistortion}, and setting \( M = \mathbb S^2 \), we recover Calegari--Freedman's result in dimension two.

\begin{Cor}[Calegari--Freedman \cite{CalegariDistortion}]
\( H(\mathbb S^2) \) is strongly distorted. 
\end{Cor}

Strong distortion is inherited by quotients, and so arguing as in the preceding subsection, we deduce the following corollary for second-countable Stone spaces.

\begin{Cor}
The homeomorphism group of a perfectly self-similar second-countable Stone space is strongly distorted; in particular, \( \mathrm{Homeo}(2^\bn) \) is strongly distorted.  
\qed
\end{Cor} 

To the best of the author's knowledge, this result does not appear in the literature; however, it is known that \( \mathrm{Homeo}(2^\bn) \) is strongly bounded \cite{Kwiatkowska}. 

There are several other consequences of being strongly distorted; for instance, if \( G \) is strongly distorted, then:
\begin{itemize}
\item \( G \) is \emph{strongly bounded}\index{group!strongly bounded} (i.e.,~every isometric action of \( G \) on a metric space has bounded orbits, and in particular, every left-invariant metric on \( G \) is bounded), 
\item \( G \) has \emph{uncountable cofinality}\index{uncountable cofinality} (i.e.,~\( G \) is not the union of a countable strictly increasing sequence of subgroups), and
\item \( G \) has the \emph{Schreier property}\index{Schreier property} (i.e.,~every countable subset of \( G \) is contained in a finitely generated subgroup of \( G \)). 
\end{itemize}
A more in-depth discussion of these various properties, and how they are related, can be found in the introduction of \cite{LeRouxStrong}.


\subsection{Coarse Boundedness}
\label{subsec:cb}

Thus far, we have only discussed perfectly self-similar 2-manifolds, and we have motivated the results by comparing to the structure of \( H(\mathbb S^2) \). 
We will now broaden our discussion to include  uniquely self-similar 2-manifolds.
Ideally, we would like to compare the structure of homeomorphism groups of uniquely self-similar 2-manifolds with that of \( H(\br^2) \).
Let us quickly summarize the analogs of Theorems~\ref{thm:main2}~and~\ref{thm:distortion} for \( H(\br^2) \): every element of \( H(\br^2) \) has commutator length at most three (see Cor~\ref{cor:unif})  and \( H(\br^2) \) is strongly distorted \cite[Theorem~1.6]{LeRouxStrong}.

Unfortunately, we cannot generalize these statements to the entire class of uniquely self-similar 2-manifolds. 
Let \( L \) denote the orientable one-ended infinite-genus 2-manifold, often referred to as the \emph{Loch Ness monster surface}\index{surface!Loch Ness monster}\index{2-manifold!Loch Ness monster}. 
Domat--Dickmann \cite{DomatBig} showed that \( H(L) \) admits an epimorphism to \( \br \), which implies \( H(L) \) fails to be perfect, strongly distorted, strongly bounded, and fails to have uncountable cofinality and the Schreier property. 
Moreover, realizing \( \br \) as a direct sum of a continuum worth of copies of \( \bq \), we see that \( H(L) \) has many countable-index subgroups.

Despite this, we can obtain some positive comparisons to \( H(\br^2) \) if we take into account the topology of \( H(M) \). 
Rosendal has laid the foundation for extending the tools of geometric group theory to the setting of non-locally compact topological groups (see \cite{RosendalBook}), and in particular, understanding when such a group has a canonical metric (up to quasi-isometry). 
Mann--Rafi \cite{MannLarge} have classified the mapping class groups that admit such a metric.
Here, we will focus on the portion of Mann--Rafi's work dedicated to understanding which 2-manifolds' homeomorphism groups behave---geometrically---like finite (or compact) groups in Rosendal's theory.

A topological group \( G \) is \emph{coarsely bounded}\index{group!coarsely bounded} if every \( G \)-orbit is bounded whenever \( G \) acts continuously on a metric space by isometries\footnote{This also appears in the literature as the \emph{topological Bergman property} or \emph{property (OB)}.}.
In particular, every continuous left-invariant (pseudo-)metric on a coarsely bounded topological group has finite diameter, and in fact, this is equivalent to being coarsely bounded. 
In the setting of Polish groups, there are numerous equivalent conditions (see \cite[Proposition~2.7]{RosendalBook}). 
Note that if \( G \) is strongly bounded, then it is coarsely bounded. 

A historical remark: in the article introducing the notion of a coarsely bounded group \cite{RosendalTopological}, Rosendal proved that the homeomorphism group of a sphere is coarsely bounded. 
The preprint appeared within a day of the preprint of Calegari--Freedman \cite{CalegariDistortion}, establishing strong distortion for these homeomorphism groups.
In a later version of \cite{CalegariDistortion}, an appendix, written by Corulier, was added showing that strong distortion implies strongly bounded, and hence strengthening the result of Rosendal.

In \cite[Proposition~3.1]{MannLarge}, Mann--Rafi establish that the mapping class group of every self-similar 2-manifold is coarsely bounded.
We will see that their proof readily adapts to the case of homeomorphism groups and non-orientable 2-manifolds.

\begin{MainThm4} 
Let \( M \) be a self-similar 2-manifold. 
Then, \( \Homeo(M) \) and \( \mcg(M) \) are coarsely bounded. 
\end{MainThm4}

Being coarsely bounded is inherited by quotients by closed subgroups, and so we have the following.

\begin{Cor}
\label{cor:cb-end}
The homeomorphism group of a uniquely self-similar second-countable Stone space is coarsely bounded. 
\qed
\end{Cor}

There are many natural examples of uniquely self-similar second-countable Stone spaces, such as every compact countable Hausdorff space of Cantor--Bendixson degree one.
Using the classification of compact countable Hausdorff spaces \cite{CountableClassification}, every such space of Cantor--Bendixson degree one is homeomorphic to an ordinal space\footnote{Given an ordinal \( \beta \), the ordinal space associated to \( \beta \), also denoted \( \beta \), is the space \( \{ \eta : \eta < \beta\} \) equipped with the order topology.}  of the form \( \omega^\alpha + 1 \), where \( \omega \) is the first countable ordinal and \( \alpha \) is a countable ordinal.
When \( \alpha = 1 \), we have from Corollary~\ref{cor:cb-end} that \( \mathrm{Sym}(\bn) \), the symmetric group on \( \bn \), is coarsely bounded. 
This is a corollary of the stronger result of Bergman \cite{BergmanGenerating} stating that every word metric on \( \mathrm{Sym}(\bn) \) is bounded.
In fact, it can be deduced from the work of Galvin \cite{GalvinGenerating} that \( \mathrm{Sym}(\bn) \) is strongly distorted, and hence strongly bounded (which implies Bergman's result).

\subsection{Rokhlin property}

In Theorem~\ref{thm:cb}, by taking into account the topological structure of homeomorphism groups, we were able to establish a weaker version of Theorem~\ref{thm:distortion} in the case of uniquely self-similar 2-manifolds.  
The next result is motivated by asking to what extent can the topological group structure be used to recover a weaker version of Theorem~\ref{thm:main2} for uniquely self-similar 2-manifolds. 
We offer one answer, which is to show that Corollary~\ref{cor:open}, a corollary of Theorem~\ref{thm:main2}, can be extended to include uniquely self-similar 2-manifolds.

A topological group has the \emph{Rokhlin property}\index{Rokhlin property} if it contains a dense conjugacy class. 
Note that the existence of a dense conjugacy class precludes a topological group from having any proper open normal subgroups (since the complement of an open normal subgroup is an open conjugation-invariant subset). 
A classification of mapping class groups of orientable 2-manifolds with the Rokhlin property was given independently by Lanier and the author \cite{LanierMapping} and by Hern\'andez--Hru\v{s}\'ak--Morales--Randecker--Sedano--Valdez \cite{HernandezConjugacy}; this result was extended to non-orientable 2-manifolds by Estrada \cite{EstradaConjugacy}. 
We give the relevant part of this classification here, extended to homeomorphism groups:

\begin{MainThm5}
If \( M \) is a uniquely self-similar 2-manifold, then \( H(M) \) and \( \mcg(M) \) have the Rokhlin property. 
\end{MainThm5}

The Rokhlin property was introduced in \cite{GlasnerZero}, and it was shown in \cite{GlasnerTopological} that \( \Homeo(\mathbb S^{2d}) \) has the Rokhlin property, where \( d \in \bn \) and \( \mathbb S^{2d} \) is the \( (2d) \)-dimensional sphere (see Theorem~\ref{thm:2sphere}). 
It follows from this fact, Theorem~\ref{thm:rokhlin}, the results of \cite{LanierMapping,HernandezConjugacy}, and forthcoming work of Lanier and the author \cite{LanierVirtual} that if \( M \) is a  2-manifold, then \( H(M) \) has the Rokhlin property if and only if \( M \cong \mathbb S^2 \) or \( M \) is uniquely self-similar.

Given the discussion above, Theorem~\ref{thm:rokhlin} shows that \( H(M) \) has no proper open normal subgroup whenever \( M \) is uniquely self-similar, which allows us to extend Corollary~\ref{cor:open} to all self-similar 2-manifolds. 
Before stating the corollary, we also note that, in a Polish group \( G \), any closed countable-index subgroup \( N \) of \( G \) is open.
To see this, observe that the left cosets of \( N \) cover \( G \), and hence \( N \) must have nonempty interior, as a Polish space cannot be expressed as a countable union of nowhere dense subsets (a consequence of the Baire Category Theorem).
To finish, observe that any subgroup of a topological group with nonempty interior is open, since the subgroup is a union of the translates of any open subset contained in the subgroup.

\begin{Cor}
If \( M \) is a self-similar 2-manifold, then neither \( H(M) \) nor \( \mcg(M) \) contain a proper open normal subgroup. Moreover, any proper closed normal subgroup of \( H(M) \) or \( \mcg(M) \) has uncountable index.  
\qed
\end{Cor}

We are also able to establish Corollary~\ref{cor:quasimorphism} in the topological category for all self-similar 2-manifolds.
As quasimorphisms are bounded on conjugacy classes (see \cite[Section~2.2]{scl}), any continuous quasimorphism of a group with the Rokhlin property must be bounded.
Using this fact for uniquely self-similar 2-manifolds together with Theorem~\ref{thm:rokhlin} and Corollary~\ref{cor:quasimorphism} in the perfectly self-similar case, we have the following corollary. 

\begin{Cor}
If \( M \) is a self-similar 2-manifold, then every continuous quasimorphism of \( H(M) \) (resp., \( \mcg(M) \)) is bounded. 
\end{Cor}

There is a fascinating body of work establishing automatic continuity properties for homomorphisms between classes of groups.
When coupled with the Rokhlin property, the topological condition of having a group element with a dense conjugacy class has algebraic consequences; in particular, there are a plethora of countable groups that we know cannot appear as the quotient of a Polish group with a unique open normal subgroup.

A group \( H \) is \emph{cm-slender}\index{group!cm-slender} if the kernel of any abstract homomorphism from a completely metrizable topological group to \( H \) is open. 
We readily see that any homomorphism from a Polish group with the Rokhlin property to a cm-slender group is trivial. 
Examples of cm-slender groups are abundant, and include free abelian groups \cite{DudleyContinuity}; torsion-free word-hyperbolic groups, Baumslag--Solitar groups, and Thompson’s group \( F \) (see \cite{ConnerNote});  non-exceptional spherical Artin groups (e.g., braid groups) \cite{CorsonPreservation}; and any torsion-free subgroup of a mapping class group of a finite-type 2-manifold \cite{BogopolskiAbstract}. 
In fact, Conner has made the following conjecture:

\begin{Conj}[Conner's Conjecture]
A countable group is cm-slender if and only if it is torsion free and does not contain an isomorphic copy of \( \bq \). 
\end{Conj}

We capture this discussion in the following corollary: 

\begin{Cor}
Let \( M \) be a self-similar 2-manifold. 
Every homomorphism from either \( H(M) \) or \( \mcg(M) \) to a cm-slender group is trivial. 
\qed
\end{Cor}

The Rokhlin property is inherited by quotients by closed subgroups, yielding the following corollary.

\begin{Cor}
The homeomorphism group of a uniquely self-similar second-countable Stone space has the Rokhlin property.
\qed
\end{Cor}

This corollary tells us that \( \mathrm{Sym}(\bn) \) has the Rokhlin property.
A much stronger property is known for the symmetric group, namely \( \mathrm{Sym}(\bn) \) has ample generics (see \cite{KechrisTurbulence}).


\subsection{Automatic continuity}
\label{subsec:automatic}

A topological group \( G \) has the \emph{automatic continuity property}\index{automatic continuity property} if every abstract homomorphism from \( G \) to a separable topological group is continuous.  
This is a strong property indicating a deep connection between the algebra and topology of a group.
Automatic continuity has a relatively long history of being studied, and we refer the reader to Rosendal's excellent survey \cite{RosendalSurvey} for an introduction.

We need a few brief definitions to begin.
A subset \( W \) of a group \( G \) is \emph{countably syndetic}\index{countably syndetic subset} if \( G \) is the union of countably many left translates of \( W \).
A topological group \( G \) is \emph{Steinhaus}\index{group!Steinhaus} if there exists \( m \in \bn \), depending only on \( G \), such that \( W^m \) contains an open neighborhood of the identity for any symmetric countably syndetic set \( W \subset G \).

The notion of a Steinhaus topological group was introduced by Rosendal--Solecki in order to provide a general strategy for establishing automatic continuity; in particular, every Steinhaus topological group has the automatic continuity property \cite[Proposition~2]{RosendalAutomatic1}.

Rosendal \cite{RosendalAutomatic2} showed that \( H(M) \) is Steinhaus for every compact 2-manifold (and Mann \cite{MannAutomatic1} subsequently showed that \( H(M) \) is Steinhaus for every compact manifold).
It is therefore natural to ask: \emph{which non-compact manifolds have Steinhaus homeomorphisms groups?}
In \cite[Question~2.4]{MannAutomatic2}, Mann asks this question in the context of mapping class groups of 2-manifolds.

For starters, we know that there exist 2-manifolds whose mapping class groups (and hence homeomorphism groups) fail to have the automatic continuity property.
Mann gives such examples in \cite{MannAutomatic2}.
Additionally, for the Loch Ness monster surface \( L \), the epimorphism \( \mcg(L) \to \mathbb R \) of Domat--Dickmann \cite{DomatBig} mentioned previously fails to be continuous, and hence \( \mcg(L) \) does not have the automatic continuity property.
The discontinuity of this homomorphism can be deduced in several ways, for instance, it follows from any one of the following facts: (1)  \( \mcg(L) \) is coarsely bounded (Theorem~\ref{thm:cb}), (2) \( \mcg(L) \) has the Rokhlin property (Theorem~\ref{thm:rokhlin}), and (3) \( \mcg(L) \) has a dense subgroup that is perfect (a fact we have not discussed, see \cite{AramayonaFirst} for the relevant discussion). 
This example was bootstrapped by Malestein--Tao \cite{MalesteinSelf} to construct additional examples of surfaces whose mapping class group fails to have the automatic continuity property, including the plane with an infinite discrete set removed (i.e., the flute surface).

Now, let \( M \) be a 2-manifold obtained by removing the union of a totally disconnected perfect set and a finite set from a compact 2-manifold (e.g., removing a Cantor set from the 2-sphere). 
Then, Mann \cite{MannAutomatic2} showed that \( H(M) \) and \( \mcg(M) \) are Steinhaus.
The next theorem makes progress on \cite[Question~2.4]{MannAutomatic2} mentioned above by extending Mann's work to the setting of perfectly tame 2-manifolds. 

\begin{Def}
\label{def:perfectly_stable}
A 2-manifold \( M \) is \emph{perfectly tame}\index{2-manifold!perfectly tame}\footnote{The naming is motivated by the fact that a perfectly tame 2-manifold is tame in the sense of Mann--Rafi \cite{MannLarge}.} if it is homeomorphic to the connected sum  of a finite-type 2-manifold and finitely many perfectly self-similar 2-manifolds, each of whose space of ends can be written as \( \mathscr P \cup \mathscr D \), where \( \mathscr P \) is perfect\footnote{Recall that the empty set is a totally disconnected, perfect, and discrete subset of every  space.}, \( \mathscr D \) is a discrete set consisting of planar ends, and \( \partial \mathscr D =  \mathscr P \).
\end{Def}

The definition of a perfectly tame 2-manifold is engineered so that the arguments given by Mann in \cite{MannAutomatic2} can be adapted to establish:

\begin{MainThm1}
The homeomorphism group of a perfectly tame 2-manifold is Steinhaus. 
\end{MainThm1}

In the proof of Theorem~\ref{thm:main1}, we will see that the Steinhaus constant can be taken to be 60. 
In forthcoming work, a stronger version of Theorem~\ref{thm:main1} has been independently obtained by Bestvina--Domat--Rafi. 
If \( G \) is a Polish Steinhaus group and \( G' \) is a Polish group such that there exists an epimorphism \( G \to G' \), then \( G' \) is Steinhaus \cite[Corollary~3]{RosendalAutomatic1}, which yields the following corollary.

\begin{Cor}
The mapping class group of a perfectly tame 2-manifold is Steinhaus. \qed
\end{Cor}

A noteworthy example of a perfectly tame 2-manifold is the blooming Cantor tree surface\index{surface!blooming Cantor tree}\index{2-manifold!blooming Cantor tree} (i.e.,~the orientable 2-manifold whose end space is a Cantor set of non-planar ends), which is infinite genus; hence, Theorem~\ref{thm:main1} gives the first example of an infinite-genus 2-manifold with Steinhaus homeomorphism group. 
Dickmann \cite{DickmannAutomatic} recently classified the surfaces whose pure mapping class groups have automatic continuity, and also includes examples of an infinite-genus surfaces with non-empty non-compact boundary whose mapping class group has the Steinhaus property.

As already noted, the Steinhaus property implies automatic continuity, yielding:

\begin{Cor}
If \( M \) is a perfectly tame 2-manifold, then \( H(M) \) and \( \mcg(M) \) have the automatic continuity property. 
\qed
\end{Cor}

As \( 2^\bn \) can be realized as the end space of a planar perfectly tame 2-manifold, namely the 2-sphere with a Cantor set removed, we can appeal to the fact that quotients of Steinhaus groups by closed subgroups are Steinhaus to obtain the following corollary, which is a result of Rosendal--Solecki \cite{RosendalAutomatic1}:

\begin{Cor}
\( \mathrm{Homeo}(2^\bn) \) is Steinhaus and hence has the automatic continuity property. 
\qed
\end{Cor}

One application of automatic continuity is establishing the uniqueness of a Polish group structure.
Let \( G \) be a group, and let \( \tau, \tau' \) be two topologies on \( G \) such that \( G \) equipped with each of the topologies is a Polish group. 
Assume \( (G, \tau) \) has the automatic continuity property, so that the identity homomorphism \( (G, \tau) \to  (G,\tau') \) is continuous. 
Let \( V \in \tau \), and let us consider \( V \) as a subset of \( (G, \tau') \). 
Then \( V \), being the continuous image of an open set of a Polish space, is analytic in \( (G, \tau') \).
Similarly, \( G \ssm V \), being the continuous image of a closed subset of a Polish space, is analytic in \( (G, \tau') \).
A theorem of Suslin implies that \( V \) is Borel in \( (G, \tau') \) as both \( V \) and its complement are analytic in \( (G, \tau') \).
In other words, the identity homomorphism \( (G, \tau') \to ( G, \tau) \) is Baire measurable.
Now, every Baire measurable homomorphism between Polish groups is continuous (see \cite[Section~2]{RosendalSurvey}), and hence, the identity homomorphism is a topological isomorphism. 

\begin{Cor}
\label{cor:unique}
If \( M \) is a perfectly tame 2-manifold, then \( H(M) \) and \( \mcg(M) \) have unique Polish group topologies. 
\qed
\end{Cor}

Corollary~\ref{cor:unique} is most interesting in the case of mapping class groups. 
In the case of homeomorphism groups, Corollary~\ref{cor:unique} is much weaker than what is already known: Kallman \cite{KallmanUniqueness} has given very general criteria for a closed subgroup of a homeomorphism group to have a unique Polish group topology.  
In particular, the homeomorphism group of any manifold has a unique Polish group topology, namely the  compact-open topology.

\subsection{Commutator subgroups}
\label{subsec:commutator}

Given a group \( G \), the group \( G/[G,G] \) is called the \emph{abelianization}\index{abelianization} of \( G \) and is denoted \( G^{\textrm{ab}} \); it has the property that every homomorphism from \( G \) to an abelian group factors through \( G^{\textrm{ab}} \). 
In particular, a group is perfect if and only if its abelianization is trivial. 

Pure mapping class groups of finite-type 2-manifolds of genus at least three are perfect (see \cite[Section~5.1.2]{FarbPrimer} for a short proof and an overview of the history).
In contrast, combining the work of Domat--Dickmann \cite{DomatBig} and Aramayona, Patel, and the author \cite{AramayonaFirst}, the abelianization of the pure mapping class group of an orientable infinite-type 2-manifold is infinite.
On the other hand, we saw in  Theorem~\ref{thm:main2} that \( H(M) \) and  \( \mcg(M) \) are uniformly perfect whenever \( M \) is perfectly self-similar. 

Abelian groups always admit countable quotients, and hence the abelianization is useful in trying to understand the structure of countable-index subgroups.
Therefore, there is a natural interest in computing abelianizations:

\begin{Problem}
Given a 2-manifold \( M \), compute the abelianization of \( H(M) \) and \( \mcg(M) \). 
\end{Problem}

With this problem in mind, we give a slight generalization of a result of Field--Patel--Rasmussen \cite{FieldStable} to the setting of homeomorphism groups and to include non-orientable 2-manifolds. 

\begin{MainThm3}
Let \( M \) be a 2-manifold obtained by taking the connected sum of a finite-type 2-manifold and finitely many perfectly self-similar 2-manifolds.
Then, the commutator subgroup of \( H(M) \)  is open.
\end{MainThm3}

Given a subgroup \( H \) of a topological group \( G \), the canonical map \( G \to G/H \) is open when \( G/H \) is equipped with the quotient topology.  
This follows from the fact that the saturation of an open subset \( U \) of \( G \) is simply \( UH \).
Using this basic fact about topological groups, together with the fact that the image of the commutator subgroup in the domain of an epimorphism is the commutator subgroup in the codomain, yields the following two corollaries:

\begin{Cor}
Let \( M \) be a 2-manifold obtained by taking the connected sum of a finite-type 2-manifold and finitely many perfectly self-similar 2-manifolds.
Then, \( [\mcg(M),\mcg(M)] \) is an open subgroup of \( \mcg(M) \).
\qed
\end{Cor}

\begin{Cor}
Let \( M \) be a 2-manifold obtained by taking the connected sum of a finite-type 2-manifold and finitely many perfectly self-similar 2-manifolds.
The abelianization of \( H(M) \) and \( \mcg(M) \) are isomorphic.
Moreover, the abelianization of \( H(M) \) and \( \mcg(M) \) are countable.
\end{Cor}

\begin{proof}
Let \( q \co H(M) \to H(M)^\text{ab} \) be the canonical homomorphism. 
Equip \( H(M)^{\text{ab}} \) with the quotient topology, so that \( q \) is open and continuous.
Moreover, since \( [H(M), H(M)] \) is open and \( q \) is open, \( H(M)^{\text{ab}} \) is discrete. 
As the continuous image of a separable space,  \( H(M)^{\text{ab}} \) is separable, and hence countable (as it is separable and discrete).
Now since \( q \) is a continuous homomorphism to a discrete space, the connected component of the identity is in the kernel of \( q \); in particular, \( q \) factors through \( \mcg(M) \). 
\end{proof}


\section{Topology of surfaces}
\label{appendix:topology}

This section has been written with the goal of providing enough background so that the chapter is accessible to non-experts, especially  graduate students.
In particular, we provide a thorough overview of the necessary background in the topology of surfaces, while assuming the reader is familiar with the basics of algebraic topology, namely the concepts of connected sums, orientation, and Euler characteristic.

In the background of the following discussion is the fact that every 2-manifold admits a triangulation (see \cite{ThomassenJordan}), which allows one to appeal to the theory of piecewise-linear manifolds, and in particular the structure of regular neighborhoods (see \cite{BryantPiecewise} for a survey of the relevant theory). 
We will not explicitly appeal to these facts, but they are required if one wants to prove the statements given below.

Most readers are likely familiar with the classification of compact 2-manifolds, but there is a classification for all 2-manifold in terms of certain 0-dimensional data. 
The majority of this section is dedicated to stating and understanding this classification, which is an essential ingredient in the proof of every theorem highlighted in the previous section. 
We will often work with subsurfaces, and so we will focus on presenting the classification of surfaces with compact boundary. 
The standard reference is \cite{RichardsClassification}. 

The classification theory relies on the notion of a topological end. 
For the sake of variation, we give a slightly different (and more concise) definition than we did in an earlier volume of this series \cite{AramayonaVlamis}. 
There are many equivalent definitions one can give, all of which have (dis)advantages. 
We state the definition in terms of surfaces, but it can be adapted to a much broader class of spaces (see \cite{FernandezEnds} for a concise and detailed discussion of end spaces and the associated Freudenthal compactification). 

Given two compact subsets \( K \) and \( K' \) of a surface \( S \) such that \( K \subset K' \) there exists a projection \( f_{K,K'} \co \pi_0(S \ssm K') \to \pi_0(S \ssm K) \) mapping a component \( U \) of \( S \ssm K' \) to the component of \( S\ssm K \) containing \( U \). 

\begin{Def}[Space of ends\index{space of ends}\index{2-manifold!end space of}]
Let \( S \) be a surface, and let \( \mathcal K \) be the directed partially ordered set of compact subsets of \( S \) (ordered by inclusion). 
The \emph{space of ends}, or \emph{end space}, of \( S \), denoted \( \Ends(S) \), is the space whose underlying set is the inverse limit of the inverse system \( (\{\pi_0(S\ssm K)\}_{K \in \mathcal K}, \{f_{K,K'}\}_{K\subseteq K' \in \mathcal K}) \), and whose topology is the limit topology, that is, the coarsest topology for which the natural projection \( \Ends(S) \to \pi_0(S \ssm K) \) is continuous for each \( K \in \mathcal K \). 
An \emph{end}\index{end}\index{2-manifold!end of} of \( S \) is an element in \( \Ends(S) \).
\end{Def}

This definition is hard to digest (we will clarify shortly), but it is useful for establishing formal properties.
For instance, we readily see that the space of ends is a topological invariant, and in particular, every homeomorphism \( f \co S \to S \) induces a homeomorphism \( f \co \Ends(S) \to \Ends(S) \).
A \emph{subsurface}\index{subsurface} of \( S \) is a closed subset of \( S \) that is a surface when equipped with the subspace topology.
Every surface admits an exhaustion by compact subsurfaces, that is, if \( S \) is a surface with compact boundary, there exists a sequence \( \{ \Sigma_n \}_{n\in \bn} \) of compact subsurfaces such that \( \Sigma_n \) is contained in the interior of \( \Sigma_{n+1} \) and such that \( S = \bigcup_{n\in\bn} \Sigma_n \). 
As the sequence \( \{ \Sigma_n\}_{n\in\bn} \) is cofinal in the poset \( \mathcal K \), the inverse limit taken over this compact exhaustion---instead of all of \( \mathcal K \)---is still \( \Ends(S) \).
Appealing to a standard fact of inverse limits, this says that \( \Ends(S) \) embeds into \( \prod_{n\in\bn} \pi_0(S \ssm \Sigma_n) \) as a closed subset.
Moreover, as \( \prod_{n\in\bn} \pi_0(S \ssm K) \) is a countable product of finite discrete spaces, it is Hausdorff, second countable, zero-dimensional, and compact (by Tychonoff's theorem). 
All of these properties are inherited by closed subsets, establishing the following:

\begin{Prop}
The space of ends of a surface is Hausdorff, second countable, zero-dimensional, and compact\footnote{In other words, it is a second countable Stone space.}.
In particular, it is homeomorphic to a closed subset of the Cantor set \( 2^\bn \).
\qed
\end{Prop}

Let us consider some basic examples:
\begin{itemize}
\item The end space of a compact 2-manifold is empty.
\item The end space of \( \br^2 \) is a singleton.
\item  \( \mathbb S^1 \times \br \) has two ends. 
\item The end space of \( \mathbb C \ssm \bn \) is the one-point compactification of \( \bn \). 
\item The end space of the manifold obtained by removing a Cantor space from \( \mathbb S^2 \) is homeomorphic to the Cantor space \( 2^\bn \). 
\end{itemize}

Let us now be more concrete and give a description of the clopen subsets of the space of ends as we will represent them throughout the chapter. 
From the definition, given a surface \( S \) and any open subset \( \Omega \) of \( S \)  with compact boundary, there is a corresponding clopen subset \( \widehat \Omega \) of \( \Ends(S) \), which is given as follows: \( \Omega \) defines a point in \( \pi_0(S\ssm \partial \Omega) \), and \( \widehat \Omega \) is the preimage of this point under the natural projection \( \Ends(S) \to \pi_0(S \ssm \partial \Omega) \). 
In fact, every clopen subset of \( \Ends(S) \) is of the form \( \widehat \Omega \) for some open subset \( \Omega \) with compact boundary. 
To simplify language, we say that \( \Omega \) is a \emph{neighborhood}\index{end!neighborhood of} of an end \( e \) if \( e \in \widehat \Omega \).

\begin{Def}
\label{def:freudenthal}
A \emph{Freudenthal subsurface}\index{subsurface!Freudenthal} of a 2-manifold is a subsurface whose boundary is connected and compact. 
\end{Def}

If \( D \) is a Freudenthal subsurface with interior \( \Omega \), then we define \( \widehat D \) to be equal to \( \widehat \Omega \). 
A sequence of Freudenthal subsurfaces \( \{D_n\}_n\in\bn \) satisfying \( D_{n+1} \subset D_n \) and \( \bigcap D_n = \varnothing \) defines an end, and in fact, the space of ends can be defined as equivalence classes of such sequences (c.f., \cite{AramayonaVlamis, RichardsClassification}).
Given the discussion thus far, and using the fact that connected surfaces are path connected, we record the following lemma, which we will use throughout the chapter.

\begin{Lem}
\label{lem:basis}
Let \( S \) be a surface with compact boundary.  
Every clopen subset of \( \sE(S) \) is of the form \( \widehat D \) for some open Freudenthal subsurface \( D \). 
Moreover, given a collection \( \{ \sU_i\}_{i\in I}  \) of pairwise-disjoint clopen subsets of \( \sE(S) \), there exists a collection \( \{D_i\}_{i \in I} \) of pairwise-disjoint Freudenthal subsurfaces of \( S \) such that \( \widehat D_i = \sU_i \) for each \( i \in I \). 
\qed
\end{Lem}

An end \( e \) of \( S \) is \emph{planar}\index{end!planar} (resp., \emph{orientable}\index{end!orientable}) if there exists an open subset \( \Omega \) of \( S \) such that \( \partial \Omega \) is compact, \( e \in \widehat \Omega \), and \( \Omega \) is homeomorphic to an open subset of \( \br^2 \) (resp., is an orientable manifold); otherwise, it is \emph{non-planar} (resp., \emph{non-orientable}).
The set of non-orientable ends of \( S \) is denoted \( \Eno(S) \) and the set of non-planar ends is denoted \( \Enp(S) \).
We then have that \( \Enp(S) \) and \( \Eno(S) \) are closed subsets of \( \sE(S) \) and \( \Eno(S) \subset \Enp(S) \). 

\textbf{Notation.}
Given a surface \( S \) with compact boundary, the space of ends of \( S \)  will refer to the triple \( (\sE(S), \Enp(S), \Eno(S)) \). 
Moreover, a homeomorphism between the space of ends of two surfaces \( S \) and \( R \) will refer to a homeomorphism \( \sE(S) \to \sE(R) \) mapping \( \Enp(S) \) onto \( \Enp(R) \) and \( \Eno(S) \) onto \( \Eno(R) \).

Continuing towards the classification, we need the notion of genus. 
The \emph{genus}\index{genus} of a compact surface with \( n \) boundary components (possibly \( n = 0 \)) is defined to be the quantity \( g = (2-\chi-n)/2 \), where \( \chi \) is the Euler characteristic of the surface.  
(In the non-orientable case, this is sometimes referred to as the \emph{reduced genus}\index{genus!reduced}.)

A surface \( S \) is of \emph{finite genus}\index{genus!finite} if there exists a compact subsurface \( \Sigma \) such that every compact subsurface of \( S \) containing \( \Sigma \) has the same genus as \( \Sigma \); we then define the genus of \( S \) to be equal to the genus of \( \Sigma \). 
If \( S \) is not of finite genus, then we say it is of \emph{infinite genus}\index{genus!infinite}.
Observe that if \( S \) is of infinite genus, then \( \Enp \neq \varnothing \).
If \( S \) is not orientable, then it is of \emph{even} (resp., \emph{odd}) non-orientability if \( S \) contains a compact subsurface \( \Sigma \) such that each component of \( S \ssm \Sigma \) is orientable and \( \Sigma \) has integral (resp., half-integral) genus. 
If \( S \) is not orientable, but of neither even nor odd orientability, then we say it is \emph{infinitely non-orientable}; note that \( S \) is infinitely non-orientable if and only if \( \Eno\neq \varnothing \). 
With these definitions, we partition the class of surfaces into four orientability classes: the orientable surfaces, the infinitely non-orientable surfaces, the surfaces of even non-orientablility, and the surfaces of odd non-orientability.

\begin{Thm}[The classification of surfaces with compact boundary\index{2-manifold!classification of}\index{surface!classification of}]
\label{thm:classification}
Let \( S \) and \( R \) be two surfaces with compact boundary such that they have the same number of boundary components, they are of the same orientability class, and they are of the same genus.
Then, \( S \) and \( R \) are homeomorphic if and only if there is a homeomorphism \( \sE(S) \to \sE(R) \) mapping \( \Enp(S) \) onto \( \Enp(R) \) and \( \Eno(S) \) onto \( \Eno(R) \).
\qed 
\end{Thm}

There is a corresponding classification of surfaces that allows for non-compact boundary \cite{BrownClassification}, but we will not require it here. 
As a consequence of Richards's proof of the classification of surfaces, we obtain the following theorem:

\begin{Thm}
\label{thm:surjection}
Let \( S \) be a surface with compact boundary.
If \( \vp \co \sE(S) \to \sE(S) \) is a homeomorphism mapping \( \Enp(S) \) onto itself and mapping \( \Eno(S) \) onto itself, then there exists a homeomorphism \( f\co S \to S \) such that \( \hat f = \vp \). 
\qed
\end{Thm}

An important application of the classification of surfaces is the change of coordinates principle.  
Before, stating this corollary, we need to understand more about 1-submanifolds of 2-manifolds. 
The open annulus is the manifold \( \mathbb S^1 \times \br \) and the open M\"obius band is the manifold obtained from \( [-1,1] \times (0, 1) \) by identifying \( (-1,y) \) and \( (1,-y) \) for all \( y \in (0,1) \). 
In the former case, we call \( \mathbb S^1 \times \{0\} \) the core curve, and in the latter case, we call the image of \( [-1,1]\times\{0\} \) in the quotient space the core curve. 

The standard reference for the following discussion regarding closed curves is \cite{EpsteinCurves}.
A \emph{simple closed curve} in a 2-manifold is a compact 1-submanifold, or equivalently, the image of an embedding of the circle \( \mathbb S^1 \). 
Every simple closed curve \( a \) admits an open neighborhood \( A \) such that there exists a homeomorphism mapping \( A \) to either an open annulus  or an open M\"obius band and that sends \( a \) to the core curve in either case; in the former case, the simple closed curve is said to be \emph{two sided}\index{simple closed curve!two-sided} and in the latter is said to be \emph{one sided}\index{simple closed curve!one-sided}. 
An important consequence of the above fact is that the closure of any complementary component of a subsurface is a subsurface\footnote{This is one of the ways in which 2-manifolds are special.  The analogous statement in higher dimensions is false, as can be seen by the existence of the Alexander horned sphere in \( \br^3 \).}.

\begin{Cor}[Change of coordinates principle\index{change of coordinates principle}]
\label{cor:coordinates}
Let \( S \) be a surface with compact boundary, and let \( \Sigma \) be a subsurface of \( S \) with compact boundary. 
If \( \Sigma' \) is a subsurface of \( S \) homeomorphic to \( \Sigma \) such that the 
 closures of \( S \ssm \Sigma \) and \( S \ssm \Sigma' \) are homeomorphic, then there exists a homeomorphism \( S \to S \) mapping  \( \Sigma \) onto \( \Sigma' \).  
\qed
\end{Cor}

Let us finish this section with describing how one constructs a 2-manifold with prescribed orientability, genus, and end space (the construction is due to Richards \cite{RichardsClassification}).  
First, let us recall some basic spaces and the basics of compact surfaces.  
Let \( \mathbb S^2 \) denote the 2-sphere. 
The torus is the space \( \mathbb S^1 \times \mathbb S^1 \) and the projective plane is the space obtained by identifying antipodal points on the 2-sphere. 
Note that the genus of the 2-sphere is 0, the genus of the torus is one, and the genus of the projective plane is one-half. 

Every compact orientable (resp., non-orientable) surface is homeomorphic to a surface obtained by taking the connected sum of the 2-sphere with finitely many tori (resp., projective planes) and removing the interiors of finitely many pairwise-disjoint closed 2-disks. 
In particular, a compact surface not homeomorphic to \( \mathbb S^2 \) is of genus zero if and only if it is planar, that is, homeomorphic to a compact subsurface of \( \mathbb R^2 \).
Using the Jordan-Schoenflies theorem (see \cite{ThomassenJordan}) and the fact that every subsurface  admits an exhaustion by compact subsurfaces, we can deduce that for a surface \( S \) with compact boundary, \( \Enp(S) \) is empty if and only if \( S \) has finite genus.

Now, let \( \sE'' \subset \sE' \subset \sE \) be a nested triple of closed subsets of the Cantor set.
Embed \( \sE \) into \( \mathbb S^2 \), and identify \( \sE \) with this embedding. 
From the above discussion, the prescribed genus can be finite if and only if \( \sE' = \varnothing \). 
So, if the prescribed genus is finite, call it \( g \), then we let \( M \) be the 2-manifold obtained by taking the connected sum of \( \mathbb S^2 \ssm \sE \) with  \( g \)  tori in the orientable case or \( 2g \) projective planes otherwise.
Now suppose the prescribed genus is infinite.  
In this case, we have that \( \sE' \) is nonempty.
First suppose that \( \sE'' = \varnothing \).
Choose a sequence \( \{D'_n\}_{n\in \bn} \) of pairwise-disjoint closed disks in the complement of \( \sE \) such that the \( D_n' \) accumulate onto \( \sE' \).
Let \( M \) be the 2-manifold obtained as follows: for each \( n \in \bn \), perform a connected sum of \( \mathbb S^2 \ssm \sE \) with a torus using the disk \( D_n' \), and then connect sum with a finite number of projective planes to have the desired non-orientable genus. 
Finally, if \( \sE'' \) is not empty, choose a sequence \( \{D'_n\}_{n\in \bn} \) of pairwise-disjoint closed disks in the complement of \( \sE \) such that the \( D_n' \) accumulate onto \( \sE' \) and a sequence \( \{D''_n\}_{n\in\bn} \) of pairwise-disjoint closed disks contained in the complement of the union of \( \sE \) with the \( D_n' \) such that the \( D_n'' \) accumulate onto \( \sE'' \).
Let \( M \) be the 2-manifold obtained as follows: for each \( n \in \bn \), perform a connected sum of \( \mathbb S^2 \ssm \sE \) with a torus using the disk \( D_n' \), and with a projective plane using the disk \( D_n'' \).
Then, \( M \) is the desired 2-manifold. 

\section{Stable sets}
\label{sec:stable-sets}

The goal of this section is to introduce the notion of stable subsets of end spaces as defined by Mann--Rafi \cite[Definition~4.14]{MannLarge} and to establish a sequence of lemmas and propositions that have appeared in various forms in the literature.
The main idea behind a stable set is to capture key characteristics of the Cantor space \( 2^\bn \) in a more general setting.
Recall that Brouwer's theorem says that \( 2^\bn \) is the unique---up to homeomorphism---nonempty second-countable zero-dimensional perfect  compact Hausdorff topological space (see \cite[Theorem~7.4]{KechrisClassical} for a reference). 
As a consequence, once we have definitions in place, this will imply that every clopen subset of \( 2^\bn \) is stable. 

\begin{Rem*}
Though we work in the setting of 2-manifolds in this section, we will not use the topology of the 2-manifold itself, and hence the results in this section are results about  second-countable Stone spaces.
\end{Rem*}

\subsection{Definitions, notations, and conventions}

Given a 2-manifold \( M \), let \( \sE = \sE(M) \) denote its space of ends, let \( \Enp \) denote the subset of \( \sE \) consisting of non-planar ends, and let \( \Eno \) denote the subset of \( \sE \) consisting of non-orientable ends (see Section~\ref{appendix:topology} for more details and an introduction the space of ends).
Let \( \Homeo_M(\sE) \) denote the group consisting of homeomorphisms \( \sE \to \sE \) mapping \( \Enp \) onto \( \Enp \) and \( \Eno \) onto \( \Eno \) (see Theorem~\ref{thm:surjection}).

It follows from Richards's proof of the classification of surfaces \cite{RichardsClassification} that the canonical homomorphism \( \Homeo(M) \to \Homeo_M(\sE) \) is surjective. 

To not overburden the reader with notation, if we say two subset \( \sU \) and \( \sU' \) of \( \sE \) are homeomorphic, we mean there there is a homeomorphism \( \sU \to \sU' \) that sends  \( \sU\cap \Enp \) onto \( \sU' \cap \Enp \) and \( \sU\cap \Eno \) onto \( \sU' \cap \Eno \).

\begin{Def}[Stability]
\label{def:stable}
Let \( M \) be a 2-manifold.
An end \( e \) of \( M \) is \emph{stable}\index{end!stable} if it admits a neighborhood basis in \( \sE \) consisting of pairwise-homeomorphic clopen sets, and such a basis is called a \emph{stable basis}\index{stable basis} of \( e \). 
A clopen subset  \( \sU \) of \( \sE(M) \) is a \emph{stable neighborhood}\index{end!stable neighborhood of} of a stable end if the end admits a stable neighborhood basis containing \( \sU \).
A clopen subset of \( \sE(M) \) is \emph{stable}\index{stable subset} if there exists an end for which it is a stable neighborhood. 
Two stable ends are \emph{of the same type} if they admit homeomorphic stable neighborhoods. 
\end{Def}

Given a stable clopen subset \( \sU \) of \( \sE \), let \( \mathcal S(\sU) \) denote the elements of \( \sU \) for which \( \sU \) is stable neighborhood. 
Note that any two elements of \( \mathcal S(\sU) \) are necessarily of the same type. 

Now that we have the definition of a stable set, we reiterate that as an application of Brouwer's theorem, every clopen subset of \( 2^\bn \) is stable.
In particular, every point in \( 2^\bn \) is stable, and given any clopen subset \( \sU \subset 2^\bn \), \( \mathcal S(\sU) = \sU \). 
For another example, consider \( \bar \bn = \bn \cup \{\infty\} \), the one-point compactification of \( \bn \) and the end space of \( \bc\ssm \bn \).
Every clopen neighborhood of \( \infty \) is stable. 
More generally, given a compact Hausdorff space, the stable sets are exactly the clopen subsets of Cantor--Bendixson degree one.
This is an application of the classification of countable compact Hausdorff spaces \cite{CountableClassification}: the Cantor--Bendixson degree and rank of a such space determine the space up to homeomorphism.
More concretely, if \( X \) is compact countable Hausdorff space, then it is homeomorphic to the ordinal space of the form \( \omega^\alpha \cdot n + 1 \), where \( \omega \) is the first countable ordinal, \( \alpha+1 \) is the Cantor--Bendixson rank of \( X \), and \( n  \) is the Cantor--Bendixson degree of \( X \).

\subsection{Structure of stable sets}

We proceed to establish the basic topological structure of stable sets. 
As mentioned earlier, the following lemmas, or versions of them, have appeared in several places.
We will therefore avoid giving attributions, but we make the exception of observing that they all stem from the original introduction of stable sets by Mann--Rafi \cite[Section~4]{MannLarge} and their original contributions to their structure.  
We begin with a lemma that explains the extent to which stable ends of the same type are in fact the same. 
The proof is an example of a standard type of argument referred to as a back-and-forth argument. 

\begin{Prop}
\label{prop:well-defined}
Let \( M \) be a 2-manifold and let \( \mu_1 \) and \( \mu_2 \)  be a stable ends of \( M \) of the same type.
If, for \( i \in \{1,2\} \), \( \sU_i \) is a clopen neighborhood of \( \mu_i \) in \( \sE(M) \) contained in a stable neighborhood of \( \mu_i \), then there exists a homeomorphism \( \sU_1 \to \sU_2 \) sending \( \mu_1 \) to \( \mu_2 \). 
\end{Prop}

\begin{proof}
Let \( \sV \) and \( \sW \) be stable clopen neighborhoods of \( \mu_1 \) and \( \mu_2 \), respectively, such that \( \sU_1 \subset \sV \) and \( \sU_2 \subset \sW \). 
Let  \( \{  \sV_n \}_{n\in\bn} \) and \( \{  \sW_n \}_{n\in\bn} \) be neighborhood bases of \( \mu_1\) and \( \mu_2 \), respectively, satisfying \( \sV_1 = \sU_1 \),  \( \sW_1 = \sU_2  \),  and, for \( n\in\bn \),  \( \sV_n \subset \sV \) and \( \sW_n \subset \sW \). 
Since \( \mu_1 \) and \( \mu_2 \) are of the same type, every element of \( \mathcal S(\sW) \) admits a stable neighborhood basis consisting of sets homeomorphic to \( \sV \); in particular, there is a homeomorphic copy of \( \sV_m \) contained in  \( \sW_n \) for all \( n, m \in \bn \).
Similarly, every element of \( \mathcal S(\sV) \) admits a stable neighborhood basis consisting of sets homeomorphic to \( \sW \), and there is a homeomorphic copy of \( \sW_m \) contained in \( \sV_n \) for all \( n, m \in \bn \). 

We claim there exists a homeomorphism \( f_1 \) from \( \sV_1 \ssm \sV_2 \) onto an open subset of \( \sW_1 \ssm \{\mu_2\} \).
To see this, let \( \sW_1' \) be a clopen stable neighborhood of \( \mu_2 \) contained in \( \sW_1 \) and homeomorphic to \( \sW \). 
First suppose that \( \mathcal S(\sW_1') \) contains an element \( \mu_2' \) distinct from \( \mu_2 \).
Then, we can choose a clopen stable neighborhood \( \sW_1'' \) of \( \mu_2 \) homeomorphic to \( \sW \) and that does not contain  \( \mu_2 \). 
Then, \( \sW_1'' \) contains an open subset homeomorphic to \( \sV_1 \), and we can choose a homeomorphism \( f_1 \) mapping \( \sV_1 \) onto an open subset of \( \sW_1'' \).
Then, the restriction of \( f_1 \) to \( \sV_1\ssm \sV_2 \) yields the desired map. 
Now, if \( \mathcal S(\sW_1') = \{\mu_2\} \), then let \( f_1 \) be any homeomorphism of \( \sV_1 \) onto an open subset of \( \sW_1' \).
Then, we must have that \( f_1(\mu_1) = \mu_2 \), and hence restricting \( f_1 \) to \( \sV_1 \ssm \sV_2 \) yields the desired map.

Similarly, there exists a homeomorphism \( g_1 \) from \( (\sW_1\ssm \sW_2)\ssm \mathrm{image}(f_1) \) onto an open subset of \( \sV_2 \ssm \{\mu_1\} \). 
Let \( \sV_1' = (\sV_1\ssm \sV_2) \cup \mathrm{image}(g_1) \) and \( \sW_1' = ( \sW_1 \ssm \sW_2) \cup \mathrm{image}(f_1) \), then \( h_1 \co \sV_1' \to \sW_1' \) defined by \( h_1 = f_1 \sqcup g_1^{-1} \) (that is, \( h_1(x) = f_1(x) \) if \( x \in V_1 \) and \( h_1(x) = g_1^{-1}(x) \) if \( x \in \mathrm{image}(g_1) \)) is a homeomorphism. 

Following the same process as above, we recursively construct the homeomorphism \( f_n \) from \( ( \sV_n \ssm \sV_{n+1}) \ssm (\sV_1' \cup \cdots \cup \sV'_{n-1}) \) onto an open subset of \( ( \sW_n \ssm \sW_{n+1}) \ssm (\{\mu_2\} \cup  \sW_1' \cup \cdots \cup \sW_{n-1}') \) and then choose a homeomorphism \( g_n \) of \( (\sW_n \ssm \sW_{n+1}) \ssm ( \sW_1' \cup \cdots \cup \sW_{n-1}' \cup \mathrm{image}(f_n)) \) onto an open subset of 
\( \sV_{n+1} \ssm (\{\mu_1\} \cup \sV_1' \cup \cdots \cup \sV'_{n-1}) \).
Let 
\[
\sV_n' = (( \sV_n \ssm \sV_{n+1}) \ssm (\sV_1' \cup \cdots \cup \sV'_{n-1})) \cup \mathrm{image}(g_n)
\]
and
\[
\sW_n' = (( \sW_n \ssm \sW_{n+1}) \ssm (\sW_1' \cup \cdots \cup \sW'_{n-1})) \cup \mathrm{image}(f_n).
\]
Then, \( h_n \co \sV_n' \to \sW_n' \) defined by \( h_n = f_n \sqcup g_n^{-1} \) is a homeomorphism.  
Observe that \( \sU_1 \ssm \{\mu_1\} = \bigsqcup_{n\in\bn} \sV_n' \) and \( \sU_2\ssm\{\mu_2\} = \bigsqcup_{n\in\bn} \sW_n' \), which allows us to define \( h \co \sU_1 \to \sU_2 \) by \( h|_{\sV_n'} = h_n \) and \( h(\mu_1) = \mu_2 \).
It is readily checked that \( h \) is the desired homeomorphism. 
\end{proof}

Setting \( \mu_2 = \mu_1 \) in Proposition~\ref{prop:well-defined}, we obtain the following corollary.

\begin{Cor}
\label{cor:well-defined}
Let \( M \) be a 2-manifold.
Any two stable neighborhoods of a stable end are homeomorphic. 
\qed
\end{Cor}

An important property---established in the next proposition---of a stable set is that its corresponding subset of stable ends is either a singleton or perfect, and in particular, it is compact. 
This also tells us that \( 2^\bn \) and \( \bar \bn \), despite being special examples, nonetheless are the right examples for considering the structure of stable sets in general.  

\begin{Prop}
\label{prop:dichotomy}
Let \( M \) be a 2-manifold and let \( \sU \) be a stable clopen subset of \( \sE(M) \). 
Then, \( \mathcal S(\sU) \) is either a singleton or a perfect set. 
\end{Prop}

\begin{proof}
Suppose \( \mathcal S(\sU) \) contains at least two ends.
We first show that no element of \( \mathcal S(\sU) \) is isolated.
Let \( \mu \in \mathcal S(\sU) \). 
Fix a stable neighborhood basis \( \{\sU_n\}_{n\in\bn} \) of \( \mu \). 
Then, as \( \sU_n \) is homeomorphic to \( \sU \) by Corollary~\ref{cor:well-defined},  \( \mathcal S(\sU_n) \) contains an end distinct from from \( \mu \) and hence \( \mu \) is not isolated in \( \mathcal S(\sU) \).

We now argue that \( \mathcal S(\sU) \) is closed. 
Let \( e \) be in the closure of \( \mathcal S(\sU) \).
As \( \sU \) is clopen, it must be that \( \sU \) is a clopen neighborhood of \( e \).
Let \( \sV \) be any clopen neighborhood of \( e \) contained in \( \sU \). 
Then, there exists \( \mu \in \mathcal S(\sU) \cap \sV \), which implies, via Proposition~\ref{prop:well-defined}, that \( \sV \) is homeomorphic to \( \sU \).
Therefore, \( e \) has a stable neighborhood basis consisting of sets homeomorphic to and contained in \( \sU \); hence, \( e \in \mathcal S(\sU) \) and  \( \mathcal S(\sU) \) is closed. 
\end{proof}

The next sequence of lemmas and propositions is meant to capture properties of \( 2^\bn \) and \( \bar \bn \) that hold more generally for stable sets. 
We begin by assuming \( \mathcal S(\sU) \) is perfect, and so \( 2^\bn \) is our motivating example in this case. 

\begin{Lem}
\label{lem:decomposition}
Let \( M \) be a 2-manifold and let \( \sU \subset \sE(M) \) be a clopen stable set such that \( \mathcal S(\sU) \) is perfect.
Given any \( e \in \mathcal S(\sU) \), there exists a sequence \( \{\sU_n\}_{n\in\bn} \) of pairwise-disjoint open sets homeomorphic to \( \sU \) such that \( \sU \ssm \{e\} = \bigcup_{n\in\bn} \sU_n \). 
\end{Lem}

\begin{proof}
Fix a stable neighborhood basis \( \{\sV_n\}_{n\in\bn} \) of \( e \) contained in \( \sU \) such that \( \sV_{n+1} \subset \sV_n \) and such that \( \mathcal S(\sV_n) \ssm \sV_{n+1} \neq \varnothing \).
Then, \( \sV_n \ssm \sV_{n+1} \) is a clopen neighborhood of an element of \( \mathcal S(\sU) \) and is contained in \( \sU \); hence, \( \sV_n \ssm \sV_{n+1} \) is homeomorphic to \( \sU \) by Corollary~\ref{cor:well-defined}. 
Then, \( \sU_n = \sV_n \ssm \sV_{n+1} \) are the desired sets. 
\end{proof}

\begin{Prop}
\label{prop:disjoint-union}
Let \( M \) be a 2-manifold, let \( \sU \subset \sE(M) \) be a clopen stable set such that \( \mathcal S(\sU) \) is perfect, and let \( X \) be a countable discrete space.
If \( X \) is finite, then \( \sU \times X \) is homeomorphic to \( \sU \); otherwise, when \( X \) is infinite, the one-point compactification of \( \sU \times X \) is homeomorphic to \( \sU \). 
\end{Prop}

\begin{proof}
First assume \( X \) is finite.
Let \( n = |X| \) and let \( e_1, \ldots, e_n \in \mathcal S(\sU) \) be distinct points. 
Choose pairwise-disjoint clopen stable neighborhoods \( \sU_1, \ldots, \sU_n \subset \sU \) of \( e_1, \ldots, e_n \), respectively, such that each \( \sU_i \) is  homeomorphic to \( \sU \).
By Proposition~\ref{prop:well-defined},  \( \sU' = \sU_1 \cup \cdots \cup \sU_n \) is homeomorphic to \( \sU \).
The result follows by observing that \( \sU' \) is homeomorphic to \( \sU \times X \).

Now, suppose \( X \) is infinite.
Fix \( e \in \mathcal S(\sU) \).
By Lemma~\ref{lem:decomposition}, there exists a sequence  \( \{\sU_n\}_{n\in\bn} \)  of pairwise-disjoint open sets, each homeomorphic to \( \sU \), and such that \( \sU\ssm\{e\} = \bigcup_{n\in\bn} \sU_n \).
Enumerate the elements of \( X \), so \( X = \{x_n\}_{n\in\bn} \). 
Let \( \sU^* = (\sU \times X) \cup \{\infty\} \) denote the one-point compactification of \( \sU \times X \).
Choose \( \vp \co \sU^* \to \sU \) such that \( \vp(\infty) = e \) and such that \( \vp \) restricted to \( \sU\times \{x_n\} \) is a homeomorphism from \( \sU \times \{x_n\} \) to \( \sU_n \).
Then, \( \vp \) is a homeomorphism. 
\end{proof}

Let us now turn back to stable sets more generally.
The next lemma and proposition provide additional structure for stable sets that will be used to build homeomorphisms of 2-manifolds in the following sections. 

\begin{Lem}
\label{lem:copies}
Let \( M \) be a 2-manifold and let \( \sU \subset \sE(M) \) be a clopen stable set.
If \( \sV \) is a clopen subset of \( \sU \) such that \( \mathcal S(\sU) \ssm \sV \neq \varnothing \), then there exists \( e \in \mathcal S(\sU) \) and  a  stable neighborhood basis \( \{\sU_n\}_{n\in\bn} \) of \( e \) such that \( \sU_n \ssm \sU_{n+1} \) contains an open subset \( \sV_n \) homeomorphic to \( \sV \) for each \( n \in \bn \). 
\end{Lem}

\begin{proof}
Fix a metric \( d \) on \( \sE(M) \). 
Let \( e_1 \in \mathcal S(\sU) \ssm \sV \), let \( \sU_1 = \sU \), and let \( \sV_ 1 = \sV \).
As \( \sV_1 \) is closed, we can find an open stable neighborhood \( \sU_2 \) of \( e_1 \) homeomorphic to \( \sU \) that is contained in \( \sU_1 \) and disjoint from \( \sV_1 \). 
By possibly shrinking \( \sU_2 \), we may assume that \( \sV_1 \subset \sU_1 \ssm \sU_2 \) and that the diameter of \( \sU_2 \) is less than \( 1/2 \).
Now, \( \sU_2 \), being homeomorphic to \( \sU \), contains an open set \( \sV_2 \) homeomorphic to \( \sV \) such that there exists \( e_2 \in \mathcal S(\sU_2) \ssm \sV_2 \).
As before, we can now find an open stable neighborhood \( \sU_3 \) of \( e_2 \) of diameter less than \( 1/3 \) that is homeomorphic to \( \sU \) and such that \( \sV_2 \subset \sU_2 \ssm \sU_3 \). 

Continuing in this fashion, we build a sequence of pairwise-homeomorphic nested clopen stable sets \( \{\sU_n\}_{n\in\bn} \)  with diameters limiting to zero and  a sequence of pairwise-disjoint pairwise-homeomorphic open sets \( \{\sV_n\}_{n\in\bn} \) such that \( \sV_n \subset \sU_n \ssm \sU_{n+1} \) and \( \sV_n \) is homeomorphic to \( \sV \). 
Since \( \{\mathcal S(\sU_n)\}_{n\in\bn} \) is a nested sequence of closed subsets of a compact space, it must be that \( \bigcap \mathcal S(\sU_n) \neq \varnothing \); moreover, as the diameter of \( \sU_n \) tends to zero, there exists an end \( e \) such that \( \{e\} = \bigcap \mathcal S(\sU_n) \); hence, \( \{ \sU_n \}_{n\in\bn} \) is a stable neighborhood basis for \( e \). 
\end{proof}

\begin{Prop} 
\label{prop:same}
Let \( M \) be a 2-manifold and let \( \sU \subset \sE(M) \) be a clopen stable set.
If \( \sV \) is a clopen set of \( \sE(M) \) that is homeomorphic to an open subset of \( \sU \ssm \{e\} \) for some \( e \in \mathcal S(\sU) \), then \( \sU \cup \sV \) is homeomoprhic to \( \sU \). 
\end{Prop}

\begin{proof}
If \( \sV \subset \sU \), then the statement is trivial, so we assume that \( \sV \) is not contained in \( \sU \).
Moreover, since \( \sU \cup \sV = \sU \cup (\sV \ssm \sU) \), we may assume that \( \sU \cap \sV = \varnothing \). 
By Lemma~\ref{lem:copies}, there exists \( e' \in \mathcal S(\sU) \), a stable neighborhood basis \( \{ \sU_n\}_{n\in\bn} \) of \( e' \), and a sequence \( \{\sV_n\}_{n\in \bn} \) of pairwise-disjoint clopen subsets such that \( \sV_n \subset \sU_n \ssm \{e\} \) and \( \sV_n \) is homeomorphic to \(\sV \).

We can now construct a homeomorphism \( \psi \co \sU \cup \sV \to \sU \) as follows: fix a homeomorphism \( \psi_0 \co \sV \to \sV_1 \) and, for \( n \in \bn \), fix a homeomorphism \( \psi_n \co \sV_n \to \sV_{n+1} \).
Define \( \psi \co \sU\cup \sV \to \sU \) by \( \psi|_\sV = \psi_0 \), \( \psi|_{\sV_n} = \psi_n \), and \( \psi(x) = x \) for all \( x \) in the complement of \( \sV \cup \left(\bigcup \sV_n\right) \). 
It is readily verified that \( \psi \) is a homeomorphism. 
\end{proof}

We finish this section with providing an equivalent definition of stability, which is referred to as self-similarity by Mann--Rafi \cite{MannLarge}. 
A clopen subset \( \sU \) of \( \sE(M) \) is \emph{self-similar}\index{self-similar subset} if given open subsets \( \sU_1, \ldots, \sU_k \subset \sE(M) \) such that \( \sU = \sU_1  \cup \cdots \cup \sU_k \), then there exists \( i \in \{1, \ldots, k\} \) such that  \( \sU_i \) contains an open subset homeomorphic to \( \sU \).

\begin{Prop}
\label{prop:equiv_def}
Let \( M \) be a 2-manifold.
Suppose \( \sU \) is a clopen subset of \( \sE(M) \).
Then, the following are equivalent:
\begin{enumerate}[(1)]
\item \( \sU \) is stable.
\item \( \sU \) is self-similar.
\item If \( \sV, \sW \subset \sE(M) \) are open sets such that \( \sU = \sV \cup \sW \), then \( \sV \) or \( \sW \) contains an open subset homeomorphic to \( \sU \).
\end{enumerate}
\end{Prop}

\begin{proof}
Let us first show the equivalence of the latter two statements.
Clearly (2) implies (3) by setting \( k = 2 \). 
Conversely, suppose \( \sU = (\sU_1 \cup \cdots \sU_{k-1}) \cup \sU_k \) with \( \sU_1, \ldots, \sU_k \) open subsets of \( \sE(M) \). 
Then, by assumption, \( \sU_k \) or \( \sU_1 \cup \cdots \cup \sU_{k-1} \) must contain an open subset homeomorphic to \( \sU \).
If it is \( \sU_k \), then we are finished.
Otherwise, \( \sU_1 \cup \cdots \cup \sU_{k-1} \) contains an open set \( \sV \) homeomorphic to \( \sU \), and we repeat the above argument with \( \sV = (\sV_1 \cup \cdots \sV_{k-2}) \cup \sV_{k-1} \), where \( \sV_i = \sV \cap \sU_i \). 
Clearly, this process has finitely many steps and ends with an open subset homeomorphic to \( \sU \) and contained in one of the \( \sU_i \). 

We now establish the equivalence of the first two statements.
Clearly (1) implies (2).
Let us finish by proving the converse.
Fix a metric on \( \sE(M) \).
Cover \( \sE(M) \) with finitely many balls of radius 1/2 (of course this can be done as \( \sE(M) \) is compact).  
By assumption, taking the intersection of these balls with \( \sU \), one of these balls contains an open subset \( \sU_1 \) contained in and homeomorphic to \( \sU \); moreover, the diameter of \( \sU_1 \) is at most one. 
Now cover \( \sE(M) \) by balls of radius \( 1/4 \), and again, by taking intersections with \( \sU_1 \) and using the assumption, we obtain an open subset \( \sU_2 \) contained in \( \sU_1 \), homeomorphic to \( \sU \), and of diameter of at most 1/2. 
Continuing in this fashion, we build a sequence of nested open sets \( \{\sU_n\}_{n\in\bn} \) such that \( \sU_n \) is homeomorphic to \( \sU \) and of diameter at most \( 1/n \).
As \( \sU \) is compact, it follows that \( \bigcap_{n\in\bn} \sU_n \) is nonempty and, by construction, its diameter is 0; hence, there is a unique element \( \mu \) in the intersection, and \( \{\sU_n\}_{n\in\bn} \) is a stable neighborhood basis of \( \mu \).
Therefore, \( \sU \) is a stable neighborhood of \( \mu \), and hence stable. 
\end{proof}


\section{Freudenthal subsurfaces and Anderson's method}
\label{sec:Anderson}

The goal of this section is to adapt and generalize the technique used by Anderson in \cite{AndersonAlgebraic} to a general tool for 2-manifolds, which we call \emph{Anderson's Method} (see Proposition~\ref{prop:main2}).  
It is worth noting that Anderson's original presentation is itself in a general setting, which he then applied to several transformation groups (including \( \Homeo_0(\mathbb S^2) \) and \( \mathrm{Homeo}(2^\bn) \)). 
In the literature, there are several other instances of Anderson's ideas being generalized to various settings (for instance, see \cite{EpsteinSimplicity}).
That is to say, the notion of Anderson's method is broader than the presentation given here, which is suited to our purposes.

\subsection{Definitions, notations, conventions}

Given a subset \( X \) of a topological space, we let \( \overline X \) denote its closure, \( X^{\mathrm{o}} \) denote its interior, and \( \partial X = \overline X \ssm X^{\mathrm{o}} \) denote its boundary. 
A family of subsets \( \{ X_n \}_{n\in\bn} \) of a topological space \emph{converges to a point \( x \)} if for every open neighborhood \( U \) of \( x \) there exists \( N \in \bn \) such that \( X_n \subset U \) for all \( n > N \); it is \emph{locally finite} if given any compact subset \( K \) the set \( \{ n \in \bn: X_n \cap K \neq \varnothing \} \) is finite; it is \emph{convergent}\index{convergent family} if it is either locally finite\index{locally finite family} or converges to a point. 

The \emph{support}\index{support} of a homeomorphism \( f \co X \to X \), denoted \( \mathrm{supp}(f) \), is the closure of the set \( \{ x\in X : f(x) \neq x \} \). 
Given a subset \( A \) of  \( X \), we say a homeomorphism \( f \co X \to X \) is \emph{supported on \( A \)} if \( \supp(f) \subset A \).
If \( \{f_n\}_{n\in\bn} \) is a sequence of self-homeomorphisms of a topological space such that \( \{\supp(f_n)\}_{n\in\bn} \) either converges to a point or is locally finite, then the \emph{infinite product}\index{infinite product} \( f= \prod_{n\in \bn} f_n \) exists and is a homeomorphism.
As function composition reads right to left, we read the infinite product right to left as well, e.g., given a point \( x \) and \( k = \max\{ n \in \bn : x\in \supp(f_n)\} \), then \( f(x) = f_k \circ f_{k-1} \circ \cdots \circ f_1(x) \). 

Given a subset \( S \) of a group \( G \), the \emph{normal closure}\index{normal closure} of \( S \) is the subgroup of \( G \) generated by all the conjugates of the elements of \( S \).
An element \( g \) of \( G \) \emph{normally generates}\index{normal generator} \( G \) or is a \emph{normal generator} of \( G \) if the normal closure of \( \{g \} \) is \( G \). 

Given an open subset \( \Omega \) of a 2-manifold \( M \) with compact boundary, define \( \widehat \Omega \) to be the subset of \( \sE(M) \) consisting of ends for which \( \Omega \) is a neighborhood in \( M \) (see Section~\ref{appendix:topology} for more details).  
Given a closed subset \( \Sigma \)  of \( M \) with compact boundary, we set \( \widehat \Sigma = \widehat {\Sigma^{\mathrm{o}}} \). 
Recall that surfaces, and hence subsurfaces, are required to be connected.

\begin{Def}[Types of Freudenthal subsurfaces]
Let \( M \) be a 2-manifold. 
Recall (Definition~\ref{def:freudenthal}) that a \emph{Freudenthal subsurface}\index{subsurface!Freudenthal} of \( M \) is a subsurface with connected compact boundary. 
A Freudenthal subsurface \( \Delta \) is \emph{trim}\index{subsurface!Freudenthal!trim} if 
\begin{enumerate}[(i)]
\item \( \Delta \) is planar whenever \( \widehat \Delta \cap \Enp(M) = \varnothing \), and
\item \( \Delta \) is orientable whenever \( \widehat\Delta \cap \Eno(M) = \varnothing \).
\end{enumerate} 
A  Freudenthal subsurface \( \Delta \) is \emph{stable}\index{subsurface!Freudenthal!stable} if it is trim and either \( \Delta \) is compact or \( \widehat \Delta \) is stable (note that every compact stable Freudenthal subsurface is homeomorphic to the closed 2-disk).
For a stable Freudenthal subsurface \( \Delta \), we define
\[
\mathcal S(\Delta) = \left\{
\begin{tabular}{ll}
\( \Delta^\mathrm{o} \) & if \( \Delta \) is compact\\
\( \mathcal S(\widehat \Delta) \) & otherwise
\end{tabular} \right.
\]
A Freudenthal subsurface \( \Delta \) is \emph{dividing}\index{subsurface!Freudenthal!dividing} if  it stable, \( \mathcal S(\Delta) \) is infinite, and \( M \ssm \Delta \) contains a Freudenthal subsurface homeomorphic to \( \Delta \).
\end{Def}
 
We will often use the fact  that every clopen subset of \( \sE(M) \) is of the form \( \widehat \Delta \) for some trim Freudenthal subsurface \( \Delta \) (see Lemma~\ref{lem:basis}).
This is equivalent to the fact that the interiors of the Freudenthal subsurfaces form a basis for the topology of the Freudenthal compactification of \( M \), a viewpoint not discussed in this chapter, but motivates our naming convention. 

\begin{figure}[t]
\centering
\includegraphics[scale=1]{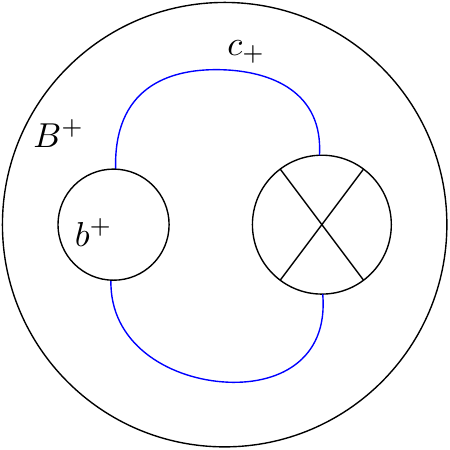}
\caption{The setup for one side of the double slide \( s_{b,c} \), and hence the setup for a boundary slide. The disk with the ``X'' represents a \emph{crosscap}\index{crosscap}, meaning the open disk is removed and the antipodal points of the boundary circle are identified.}
\label{fig:double-slide}
\end{figure}

On several occasions, we will need a locally supported homeomorphism of a non-orientable 2-manifold that allows us to reverse the orientation of a two-sided curve (in particular, the boundary of a Freudenthal subsurface); we introduce the required homeomorphism here.
Let \( B \) be a closed M\"obius band; more specifically, let \( B = [-1,1]^2/\sim \), where \( \sim \) is the equivalence relation generated by \( (1,y) \sim (-1,-y) \). 
Let \( D \subset B \) be the closed disk centered at the origin of radius 1/2.
Then sliding \( D \) along the central curve of \( B \) until it comes back to itself yields an isotopy \( H \co B \times [0,1] \to B \).
In fact, this isotopy can be taken relative to \( \partial B \), and hence the homeomorphism \( H_1 \) of \( B \) onto itself, given by \( H_1(x)=H(x,1) \),  restricts to a homeomorphism of \( B \ssm D^\mathrm o \) onto itself that reverses the orientation of \( \partial D \) while fixing \( \partial B \) pointwise. 
Moreover, the restriction of \( H_1 \) to \( \partial D \) is simply the reflection \( (x,y) \mapsto (x,-y) \). 
The homeomorphism \( H_1 \co B \ssm D^\mathrm o \to B \ssm D^\mathrm o \) is called a \emph{boundary slide}. 
The boundary slide is a standard tool in the study of mapping class groups of non-orientable surfaces (see \cite{ParisMapping}); it is based on the notion of a \emph{crosscap slide}, or \emph{Y-homeomorphism}, introduced by Lickorish \cite{LickorishHomeomorphisms}.

Let \( M \) be a non-orientable 2-manifold, and let \( b \) be a separating simple closed curve in \( M \) such that both components of \( M \ssm b \) are non-orientable and such that neither are homeomorphic to a M\"obius band.
Label the closures of the components of \( M \ssm b \) by \( M_b^+ \) and \( M_b^- \).
Viewing \( M_b^\pm \) as a surface, let \( b^\pm \) denote the boundary of \( M_b^\pm \). 
Let \( c \) be a two-sided simple closed curve in \( M \) such that \( |c \cap b| = 2 \) and such that  \( c_\pm  = c \cap M_b^\pm \)  has a neighborhood \( B_\pm \) in \( M_b^\pm \) admitting a homeomorphism \( p_\pm \co B\ssm D^\mathrm o \to B_\pm \) sending \( \partial D \) to \( b_\pm \) (see Figure~\ref{fig:double-slide}).
The boundary slide \( p_+ \circ H_1 \circ p_+^{-1} \) of \( b \) in \( M_b^+ \) along \( c_+ \) in \( B_+ \) does not extend to a homeomorphism of \( M \) since it switches orientation on one side of \( b \) but not the other; however, if we also perform the boundary slide \( p_- \circ H_1 \circ p_-^{-1} \) of \( b \) in \( M_b^- \) along \( c_- \) in \( P_- \), then we will have switched the orientation on both sides, which allows us to define a homeomorphism \( s_{b,c} \co M \to M \)  by setting \( s_{b,c}|_{B_\pm} = p_\pm \circ H_1 \circ p_\pm^{-1} \) and setting \( s_{b,c} \) to be the identity on the complement of \( B_+\cup B_- \).

\begin{Def}
The homeomorphism \( s_{b,c} \) is called a \emph{double slide}\index{double slide}.
\end{Def}

\subsection{Anderson's Method}

We begin with a baby version of (what we will refer to as) Anderson's Method. 
It is now a common technique used to express local transformations as commutators and can be used in a fairly general setting.

\begin{Def}
Let \( M \) be a 2-manifold and let \( \Sigma \subset M \) be a subsurface. 
A homeomorphism \( \vp\co M\to M \) is a \emph{\(\Sigma\)-translation}\index{translation} if \( \vp^n(\Sigma) \cap \vp^m(\Sigma) = \varnothing \) for any distinct integers \( n \) and \( m \). 
If, in addition, \( \{ \vp^n(\Sigma) \}_{n\in\bn} \) is convergent, then we say that \( \vp \) is a \emph{convergent} \( \Sigma \)-translation\index{translation!convergent}. 
\end{Def}

Recall that the \emph{commutator} of two elements \( g \) and \( h \) in a group is defined to be \( [g,h] = ghg^{-1}h^{-1} \).
The following proposition is motivated by Anderson's proof of \cite[Lemma~1]{AndersonAlgebraic}.  

\begin{Prop}
\label{prop:commutator}
Let \( M \) be a 2-manifold.
Suppose \( \Sigma \) and \( \Delta \) are closed subsets of \( M \) such that \( \Sigma \subset \Delta\) and such that there exists a convergent \( \Sigma \)-translation \( \vp \) supported in \( \Delta \).
If  \( h \in \Homeo(M) \) such that \( \supp(h) \subset \Sigma \), then \( h \) can be expressed as a commutator of two elements of \( \Homeo(M) \), each of which is supported in  \( \Delta \). 
\end{Prop}

\begin{proof}
By the convergence condition of \( \vp \), we can define the homeomorphism \[ \sigma = \prod_{n=0}^\infty \vp^n h \vp^{-n}.\]
Consider the commutator \( [\sigma,\vp] = \sigma(\vp\sigma^{-1}\vp^{-1}) \). 
Let us give an intuitive description of this commutator: \( \sigma \) is performing \( h \) on \( \vp^n(\Sigma) \) for \( n \geq 0 \),  and similarly, one views \( \vp \sigma^{-1} \vp^{-1} \) as performing \( h^{-1} \) on \( \vp^n(\Sigma) \) for \( n > 0 \), so that \( \vp \sigma^{-1} \vp^{-1} \) restricts to the  identity on \( \Sigma \). 
Leaving the details to the reader, we then see that \( \sigma(\vp\sigma^{-1}\vp^{-1}) \) restricts to the identity outside of \( \Sigma \) and restricts to \( h \) on \( \Sigma \); in particular, \( [\sigma,\vp] = h \), which is the desired result.
\end{proof}

Before getting to Anderson's Method, we will need two lemmas, both of which involve using ambient homeomorphisms to displace Freudenthal subsurfaces. 
It readily follows from the definition of stability and the classification of surfaces that any two non-compact stable Freudenthal subsurfaces  \( \Delta_1 \) and \( \Delta_2 \) of a 2-manifold are abstractly homeomorphic if and only if \( \mathcal S(\Delta_1) \) and \( \mathcal S(\Delta_2) \) contain a stable end of the same type. 
Lemma~\ref{lem:homogeneous} below provides a sufficient condition for when this abstract homeomorphism can be promoted to an ambient homeomorphism.

\begin{Lem}
\label{lem:homogeneous}
Let \( M \) be 2-manifold and let \( \Delta_1, \Delta_2 \subset M \) be abstractly homeomorphic  stable Freudenthal subsurfaces.
If \( \Sigma \) is a subsurface of \( M \) such that \( \partial \Sigma \) is compact,  \( \Delta_{1} \cup  \Delta_{2} \) is contained in the interior of \( \Sigma \) and, for each \( i \in \{1,2\} \), the interior of \( \Sigma \) contains a  Freudenthal subsurface \( \Delta'_i \) disjoint from and homeomorphic to \(  \Delta_{i} \), then there exists \( f \in \Homeo(M) \) with support in \( \Sigma \) such that \( f( \Delta_{1}) =  \Delta_{2} \).
Moreover, if \(  \Delta_{1} \cap  \Delta_{2} = \varnothing \), then \( f \) can be chosen to also satisfy \( f( \Delta_{2}) = \Delta_{1} \).
\end{Lem}

\begin{proof}
If \( \Delta_1 \) and \( \Delta_2 \) are disjoint, then the statement immediately follows from the change of coordinates principle (Corollary~\ref{cor:coordinates}).
Now we suppose \(  \Delta_1 \cap  \Delta_2 \neq \varnothing \). 
Fix \( \mu \in \mathcal S( \Delta_2) \).
If \( \mu \in \mathcal S( \Delta_1) \), then there exists a stable Freudenthal subsurface \( \Delta \) such that \( \Delta \subset \Delta_1 \cap \Delta_2 \) and \( \mu \in \mathcal S(\Delta) \). 
It follows  that \(  \Delta \cap ( \Delta'_1 \cup  \Delta'_2) = \varnothing \). 
Appealing to the disjoint case at the beginning of the proof, we see there exist homeomorphisms \( f_1,f_2,f_3, f_4 \in \Homeo(M) \) supported in \( \Sigma \) such that \( f_1(\Delta_1) = \Delta'_1\), \( f_2(\Delta'_1) = \Delta \),  \( f_3(\Delta) = \Delta'_2 \), and \( f_4(\Delta'_2) = \Delta_2 \), and hence, \( f = f_4 \circ f_3 \circ f_2 \circ f_1 \) is supported in \( \Sigma \) and maps \( \Delta_1 \) onto \( \Delta_2 \).
Otherwise, if \( \mu \notin \mathcal S(\Delta_1) \), then, arguing similarly, we can find a Freudenthal subsurface \( \Delta \) abstractly homeomorphic to  \( \Delta_1 \) such that \(  \Delta \subset \Delta_2 \) and such that  \(  \Delta \cap  \Delta_1 =\varnothing \). 
Hence, we can find \( f_1, f_2, f_3 \in \Homeo(M) \) supported in \( \Sigma \) such that \( f_1(\Delta_1) = \Delta \), \( f_2(\Delta) = \Delta'_2 \), and \( f_3(\Delta'_2) = \Delta_2 \), and hence, \( f = f_3 \circ f_2 \circ f_1 \) is supported in \( \Sigma \) and maps \( \Delta_1 \) onto \( \Delta_2 \).
\end{proof}

\begin{Lem}
\label{lem:translation}
Let \( M \) be a 2-manifold, and let \( D \) be a Freudenthal subsurface of \( M \) contained in the interior of a stable Freudenthal subsurface \( \Delta \).
If either \( D \) is compact or \( \mathcal S( \Delta) \ssm \widehat D \neq \varnothing \), then there exists a convergent \( D \)-translation with support contained in the interior of \( \Delta \). 
\end{Lem}

\begin{proof}
Let us consider the case where \( D \)---and hence \( \Delta \)---is not compact (the case with \( D \) compact is similar).
Apply Lemma~\ref{lem:copies} to \( \widehat \Delta \) and \( \widehat D \) to obtain \( e \in \mathcal S(\Delta) \) and a stable neighborhood basis \( \{\sU_n\}_{n\in\bz} \) of \( e \) such that \( \sU_n \ssm \{e\} \) contains a clopen subset \( \sV_n \) homeomorphic to \( \widehat D \) for each \( n \in \bz \). 
Moreover, we can assume that \( \sU_0 = \widehat \Delta \) and \( \sV_0 = \widehat D \). 

\begin{figure}[t]
\centering
\includegraphics[scale=0.8]{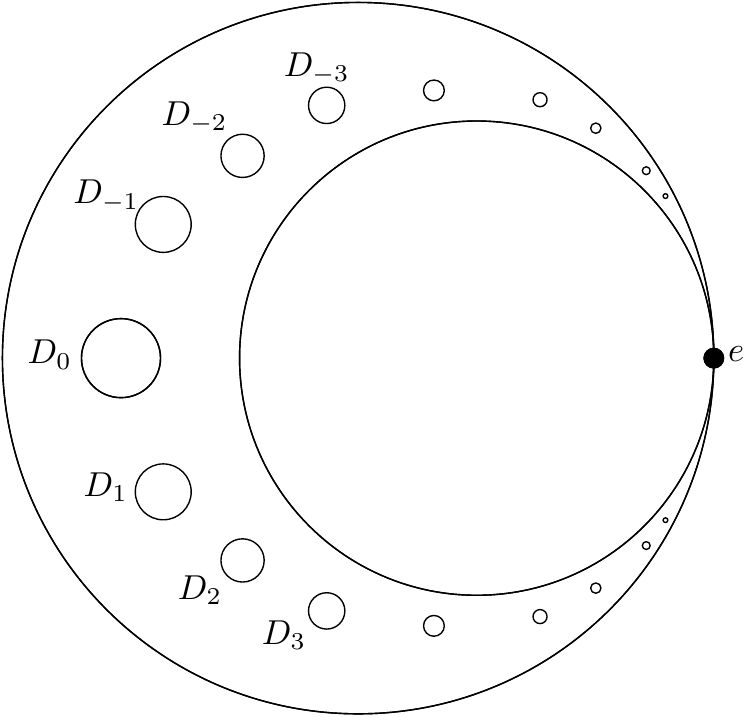}
\caption{The  support of a \( D \)-translation sending \( D_n \) to \( D_{n+1} \) for all \( n \in \bz \), as given by Lemma~\ref{lem:translation}.}
\label{fig:shift}
\end{figure}

For each \( n \in \bz \), fix a homeomorphism \( \psi_n \co \sV_n \to \sV_{n+1} \).
Define \( \Psi \co \sE(M) \to \sE(M) \) by setting \( \Psi \) to be the identity outside of \( \bigcup \sV_n \) and to satsify \( \Psi|_{\sV_n} = \psi_n \); then, \( \Psi \) is a homeomorphism supported on \( \widehat \Delta \). 
Let \( \{D_n\}_{n\in\bz} \) be a sequence of pairwise-disjoint pairwise-homeomorphic Freudenthal subsurfaces such that \(  D_n \subset \Delta \), \( \widehat D_n = \sV_n \), and \( D_0 = D \).
The existence of the \( D_n \) is guaranteed by Lemma~\ref{lem:basis}.
It is now an exercise, using the classification of surfaces (in the manner of an infinite change of coordinates, see Theorem~\ref{thm:classification} and Corollary~\ref{cor:coordinates}) and the fact that the canonical homomorphism \( \Homeo(M) \to \Homeo_M(\sE(M)) \) is surjective (Theorem~\ref{thm:surjection}), to construct a homeomorphism \( \widetilde \Psi \co M \to M \) supported in \( \Delta \) such that \( \widetilde \Psi(D_n) = D_{n+1} \) and such that  \( \widetilde \Psi \) is a lift of \( \Psi \). 
In particular, \( \widetilde \Psi \) is a convergent \( D \)-translation.
(In fact, the \( D \)-translation can be chosen to be a shift map as defined in \cite[Section~2.4]{AbbottInfinite}, which is shown in Figure~\ref{fig:shift}).
\end{proof}

The second part of the following proposition is a version of Anderson's ingenious technique---introduced in \cite[Lemma~1]{AndersonAlgebraic}---suited to our needs.   
The version we give, and its proof, is a direct generalization of Calegari--Chen's version presented in \cite[Lemma~2.10]{CalegariNormal}, which deals with the sphere minus a Cantor set. 

\begin{Prop}[Anderson's Method\index{Anderson's method}]
\label{prop:main2}
Let \( M \) be a 2-manifold, and let \( \Sigma \) and \( \Delta \) be Freudenthal subsurfaces of \( M \) such that \( \Sigma \subset \Delta^\mathrm{o} \), \( \Delta \) is stable, and either \( \Sigma \) is compact or \( \mathcal S(\Delta) \ssm \widehat \Sigma  \neq \varnothing \).
Suppose  \( h \in \Homeo(M) \) such that \( \supp(h) \subset \Sigma \).
Then,
\begin{enumerate}
\item \( h \) can be expressed as a commutator of two elements of \( \Homeo(M) \), each of which is supported in the interior of \( \Delta \), and

\item if \( f \in \Homeo(M) \) such that \( f(\Delta) \cap \Delta = \varnothing \), then \( h \) can be factored as a product of four conjugates of \( f \) and \( f^{-1} \).
\end{enumerate}
\end{Prop}

\begin{proof}
By Lemma~\ref{lem:translation}, we can choose a convergent \( \Sigma \)-translation \( \vp \) supported in \( \Delta \).
For (1),  we simply apply Proposition~\ref{prop:commutator}. 
Now let us consider (2).
Let \( \Delta' = f(\Delta \)).
By assumption, \( \Delta \cap \Delta' = \varnothing \).  
If \( \Delta \) is not orientable, then using the fact that \( \Delta \) is stable and by possibly shrinking \( \Delta \), we may assume that the complement of \( \Delta \cup \Delta' \) is not orientable as well. 
For \( n \in \bz \), let \( \Sigma_n = \vp^n(\Sigma) \) and \( \Sigma_n' = f(\Sigma_n) \). 
By Lemma~\ref{lem:homogeneous}, there exists a homeomorphism \( \psi \in \Homeo(M) \)  such that \( \psi(\Delta) = \Delta' \) and \( \psi(\Delta') = \Delta \).
By possibly pre-composing \( \psi \) with a double slide with support disjoint from \( \Delta' \), we may assume that \( f^{-1} \circ \psi \) fixes \( \partial \Delta \) pointwise. 
Similarly, we may assume that \( f \circ \psi \) fixes \( \partial \Delta' \) pointwise. 
In particular, by further pre-composing \( \psi \) with homeomorphisms supported in \( \Delta \) and \( \Delta' \), we may  assume that \( \psi|_\Delta = f|_\Delta \) and \( \psi|_{\Delta'} = \vp \circ f^{-1}|_{\Delta'} \). 

Let \( \sigma = \prod_{n=0}^\infty \vp^n h \vp^{-n} \) and let \( \tau = [\sigma,f] = \sigma f \sigma^{-1} f^{-1} \).
We claim  that 
\[ 
h = \tau\psi\tau\psi^{-1} = (\sigma f \sigma^{-1}) f^{-1} (\psi \sigma f \sigma^{-1}\psi^{-1})(\psi f^{-1} \psi^{-1}),
\]
and hence \( h \) is a product of four conjugates of \( f \) and \( f^{-1} \). 
We will provide the intuition and leave the details to the reader. 
We have already noted in the proof of Proposition~\ref{prop:commutator} that we should view \( \sigma \) as performing \( h \) on \( \Delta_n \) for \( n \geq 0 \), and so we view \( f \sigma f^{-1} \) as performing \( h^{-1} \) on \( \Delta_n' \) for \( n \geq 0 \), and hence \( \tau = \sigma (f\sigma^{-1}f^{-1}) \) is performing both \( h \) on \( \Delta_n \) and \( h^{-1} \) on \( \Delta_n' \) for \( n \geq 0 \).
Now, \( \psi \tau \psi^{-1} \) is performing \( h^{-1} \) on \( \Delta_n \) for \( n > 0 \) (and in particular is the identity on \( \Delta_0 \)) and is performing \( h \) on \( \Delta_n' \) for \( n \geq 0 \).
And hence,  \( \tau(\psi \tau\psi^{-1}) =h \). 
\end{proof}

Our main application of Anderson's method will involve dividing Freudenthal subsurfaces. 
Note that by Proposition~\ref{prop:dichotomy}, if \( \Delta \) is a dividing Freudenthal subsurface, then either \( \Delta \) is compact or \( \mathcal S(\Delta) \) is  perfect. 
Given a Freudenthal subsurface \( \Delta \) of a 2-manifold \( M \), let \( \Gamma_\Delta \) denote the normal closure of the subgroup of \( \Homeo(M) \) consisting of homeomorphisms with support in \( \Delta \). 

The following corollary will be our main use of Anderson's method. 
To simplify the statement, we abuse notation by letting \( f \in \Homeo(M) \) denote both itself as well as the element of \( \Homeo(\sE(M)) \) induced by \( f \).

\begin{Cor}
\label{cor:main2}
Let \( M \) be a 2-manifold, and let \( \Delta \subset M \) be a dividing Freudenthal subsurface.
Suppose \( f \in \Homeo(M) \) and \( x \in \mathcal S(\Delta) \) such that \(  f(x) \neq x \).
Then, every element of \( \Homeo(M) \) supported in \( \Delta \)  is a product of four conjugates of \( f \) and \( f^{-1} \); in particular, \( \Gamma_\Delta \) is contained in the group normally generated by \( f \).
\end{Cor}

\begin{proof}
By continuity, there exists a stable Freudenthal subsurface \( \Delta' \) contained in \( \Delta \) such that \( x \in \mathcal S(\Delta') \) and \( f(\Delta') \cap \Delta' = \varnothing \).
Moreover, by possibly shrinking \( \Delta' \), we may assume that \( M \ssm (\Delta' \cup f(\Delta')) \) contains a stable Freudenthal subsurface homeomorphic to \( \Delta \). 
Choose any stable Freudehntal subsurface \( \Delta'' \) that is is homeomorphic to \( \Delta' \), contained in the interior of \( \Delta' \), and satisfies \( \mathcal S(\Delta') \ssm \mathcal S(\Delta'') \neq \varnothing \). 
By Lemma~\ref{lem:homogeneous}, any element supported in \( \Delta \) can be conjugated to be supported in \( \Delta'' \).
The result now follows from Anderson's Method. 
\end{proof}

\begin{Cor}
\label{cor:commutator1}
Let \( M \) be a 2-manifold and let \( \Delta \subset M \) be a dividing Freudenthal subsurface. 
Then, every homeomorphism of \( M \) with support in \( \Delta \) can be expressed as the commutator of two elements in \( \Gamma_\Delta \); in particular, \( \Gamma_\Delta \) is uniformly perfect with commutator width one. 
\end{Cor}

\begin{proof}
Fix a  dividing Freudenthal subsurface \( \Delta' \) such that \(  \Delta \subset \Delta' \) and such that \( \mathcal S( \Delta ) \) is a proper subset of \( \mathcal S(\Delta') \). 
Indeed, this is readily seen to be possible if \( \Delta \) is compact, and if \( \Delta \) is not compact, then it can be deduced from Proposition~\ref{prop:disjoint-union} by taking the union of \( \widehat \Delta \) with a homeomorphic copy of \( \widehat \Delta \) in its complement.
Then, by Anderson's Method, for every \( g \in \Homeo(M) \) with support in \( \Delta \) there exists \( h, f \in \Homeo(M)  \) with support in \( \Delta' \) such that \( g = [h,f] \). 
It follows from Lemma~\ref{lem:homogeneous} that \( h, f \in \Gamma_\Delta \), and as the conjugate of a commutator is a commutator, we can conclude that \( \Gamma_\Delta \) is uniformly perfect with commutator width one.  
\end{proof}


\section{Topology of homeomorphism groups}
\label{sec:homeomorphism-groups}

In this section, we recall basic facts about the compact-open topology and establish classical results about homeomorphism groups of 2-manifolds. 
The aim is not to give the state of the art or the strongest statements but to prove what is necessary for the later sections and to do so in a way that highlights the use of Anderson's method. 
Note also that statements will only be made for 2-manifolds; however, much of what is said can be generalized to higher dimensions utilizing the resolution of the annulus conjecture and the stable homeomorphism conjecture (see \cite{EdwardsSolution} for the relevant history), as well as the fragmentation lemma of Edwards--Kirby discussed below. 
We end the section by establishing the equivalence of the two standard definitions of the mapping class group seen in the literature. 

\subsection{The compact-open topology}
Let us begin with considering the basics of the topology of homeomorphism groups before moving to manifolds. 
Let \( X \) be a locally compact Hausdorff topological space. 
We equip \( \mathrm{Homeo}(X) \) with the \emph{compact-open topology}\index{compact-open topology}, that is,  the topology generated by sets of the form \( U(K, W) = \{ f\in \mathrm{Homeo}(X) : f(K) \subset W \} \), where \( K \subset X \) is compact and \( W \subset X \) is open.
One readily observes that if we require \( W \) to be a precompact open subset, then this does not change the topology.

Given a compact Hausdorff topological space \( X \) with a metric \( d \), define \[ \rho_d, \mu_d \co \mathrm{Homeo}(X) \times \mathrm{Homeo(X)} \to \br \] by \( \rho_d(f,g) = \max\{ d(f(x),g(x)) : x \in X\} \) and \( \mu_d(f,g) = \rho_d(f,g) + \rho_d(f^{-1}, g^{-1}) \).
Observe that \( \rho_d \) is invariant with respect to right multiplication. 
We leave the following lemma as an exercise (one may also refer to \cite{ArensTopology}):

\begin{Lem}
Let \( X \) be a compact Hausdorff topological space.
Then, \( \mathrm{Homeo}(X) \), equipped with the compact-open topology, is a topological group.
Moreover, if \( (X,d) \) is a compact metric space, then both \( \rho_d \) and \( \mu_d \) are metrics on \( \mathrm{Homeo}(X) \) inducing the compact-open topology, and \( \mu_d \) is complete. 
\qed
\end{Lem}

The next two propositions are due to Arens \cite{ArensTopologies}, and we leave the interested  reader to either read Arens's proofs or to treat them as exercises.

\begin{Prop}
Let \( X \) be a locally compact locally connected Hausdorff space.
Then, \( \mathrm{Homeo}(X) \), equipped with the compact-open topology, is a topological group.
\qed
\end{Prop}

A \emph{compactification}\index{compactification} of a non-compact locally compact locally connected Hausdorff space \( X \) is a compact Hausdorff space \( C \) such that \( X \) can be topologically embedded in \( C \) as a dense open subset. 
Every such \( X \) admits a compactification (e.g., the one-point compactification and the Freudenthal compactification).
In fact, Arens \cite{ArensTopologies} proves that for a locally compact locally connected Hausdorff space, the compact-open topology agrees with the closed-open topology (i.e., the topology defined analogously to the compact-open topology but allowing \( K \) to be closed instead of compact).
This yields the following:

\begin{Prop}
\label{prop:compactification}
Let \( X \) be a non-compact locally compact locally connected Hausdorff space,  let \( C \) be a compactification of \( X \), and let \( \Homeo(C,X) = \{ f \in \mathrm{Homeo}(C) : f(X) = X \} \).
If every homeomorphism of \( X \) extends continuously to a homeomorphism of \( C \), then \( \mathrm{Homeo}(X) \) and \( \mathrm{Homeo}(C,X) \) are  isomorphic as topological groups, where \( \mathrm{Homeo}(X) \) and \( \mathrm{Homeo}(C) \) are equipped with their respective compact-open topologies, and \( \mathrm{Homeo}(C,X) \) is equipped with the corresponding subspace topology. 
\qed
\end{Prop}

Even though we will not tend to work with metrics on homeomorphism groups, it will arise on several occasions, and so we record an important fact that comes out of the above discussion.

\begin{Cor}
\label{cor:homeo_compactification}
Let \( X \) be a locally compact locally connected Hausdorff space, and let \( C \) be a compactification of \( X \) such that every homeomorphism of \( X \) extends continuously to a homeomorphism of \( C \). 
Suppose \( C \) supports a metric \( d \), and define \( \hat \rho_d, \hat \mu_d \co \mathrm{Homeo}(X)^2 \to \br \) given by \( \hat \rho_d(f,g) = \rho_d(\hat f, \hat g) \) and  \( \hat\mu_d(f,g) = \mu_d(\hat f, \hat g) \), where \( \hat f \) and \( \hat g \) are the extensions of \( f \) and \( g \), respectively, to \( C \).
Then, the metric topologies on \( \mathrm{Homeo(X)} \) given by \( \hat \rho_d \) and \( \hat \mu_d \) agree and are equal to the compact-open topology, and moreover, \( \hat \rho_d \) is right-invariant and \( \hat \mu_d \) is complete. 
\qed
\end{Cor}

The key takeaway of the above corollary is the following:  
\begin{quote} \textit{a complete metric on a homeomorphism group giving rise to the compact-open topology comes from a metric on a compactification of the underlying space, rather than from the space itself.} 
\end{quote}

To finish our quick recap of the basics, recall that a \emph{Polish group} is a completely metrizable  separable topological group. 
It is an exercise to see that if \( X \) is a separable locally compact locally connected Hausdorff space, then \( \mathrm{Homeo}(X) \) equipped with the compact-open topology is also separable, yielding:

\begin{Prop}
The homeomorphism group of a separable locally compact locally connected metrizable space is a Polish group when equipped with the compact-open topology. 
\qed
\end{Prop}

For the remainder of the chapter, we will always consider homeomorphism groups (resp., their subgroups) equipped with the compact-open topology (resp.,~the corresponding subgroup topology).

\subsection{Homeomorphism groups of 2-manifolds}

In this subsection, we establish various topological and algebraic properties of homeomorphism groups.
The majority of the results are classical and highlight the usage of Anderson's method. 
In addition, we will establish several factorization results that are essential for arguments throughout the chapter.

In what follows, we will regularly apply isotopies, and so we record two basic facts.

\begin{Thm}[Alexander's Trick]
Let \( \mathbb D^n = \{ x \in \br^n : |\, x\, |\leq 1 \} \) be the standard \( n \)-disk. 
Any homeomorphism \( f \co \mathbb D^n \to \mathbb D^n \) restricting to the identity on \( \partial \mathbb D^n \) is isotopic to the identity via an isotopy that fixes \( \partial \mathbb D^n \) pointwise at each stage. 
Moreover, if \( f \) fixes the origin, then the isotopy can be chosen to fix the origin at each stage. 
\end{Thm}

\begin{proof}
The function \( H \co \mathbb D^n \times [0,1] \to \mathbb D^n \) given by 
\[
H(x,t) = \left\{
\begin{array}{ll}
t f(x/t) & |\,x\, |< t \\
x & |\,x\, |\geq t
\end{array} \right.
\]
is the desired isotopy. 
\end{proof}

We forgo the proof of the next proposition, but the interested reader should see \cite[Lemmas~2.4~\&~2.5]{EpsteinCurves}.
\begin{Prop}
\label{prop:curves}
Let \( \iota \co \mathbb S^1 \to \mathbb S^1 \times \br \) be the inclusion map \( \iota(x) = (x,0) \).
If \( f \co \mathbb S^1 \to \mathbb S^1 \times \br \) is homotopic to \( \iota \), then \( f \) is isotopic to \( \iota \) via an ambient isotopy of \( \mathbb S^1 \times \br \) that is fixed outside of a compact subset. 
\qed
\end{Prop}

Motivated by Proposition~\ref{prop:curves}, we introduce a class of open subsets of homeomorphism groups of 2-manifolds. 
Let \( M \) be a 2-manifold, let \( \gamma \) be a 2-sided simple closed curve in \( M \), and let \( A \) be an open annular neighborhood of \( \gamma \). 
Fix an embedding \( \iota \co \mathbb S^1 \to A \) such that the image of \( \iota \) is \( \gamma \). 
Let \( U^+(\gamma, A) \) denote the subset of \( H(M) \) consisting of elements \( f \in U(\gamma, A) \)  such that \( f \circ \iota \) is isotopic to \( \iota \) via an isotopy fixed outside of \( A \), and note that \( U^+(\gamma, A) \) does not depend on the choice of \( \iota \).

\begin{Lem}
\label{lem:opengamma}
Let \( M \) be a 2-manifold, let \( \gamma \) be a 2-sided simple closed curve in \( M \), and let \( A \) be an open annular neighborhood of \( \gamma \). 
Then, \( U^+(\gamma, A) \) is open in \( H(M) \). 
\end{Lem}

\begin{proof}[Sketch of proof]
Let \( C \) be a compactification of \( M \), with \( C = M \) if \( M \) is compact. 
Let \( d \) be a metric on \( C \) and let \( \rho_d \) be the corresponding metric on \( H(M) \), as defined in Corollary~\ref{cor:homeo_compactification}. 
Fix an orientation on \( \gamma \), and fix \( f \in U^+(\gamma, A) \).
Given a positive real number \( \ep \), let \( U_\ep \) be the \( \ep \)-ball (with respect to \( \rho_d \)) in \( H(M) \) centered at \( f \). 
We can readily choose \( \ep \) small enough to guarantee that \( g(\gamma) \) does not bound a disk in \( A \) for any \( g \in U_\ep \). 
And, by choosing \( \ep \) much smaller than the \( d \)-diameter of \( \gamma \), we can guarantee that \( g(\gamma) \) is homotopic to \( \gamma \) in \( A \).
Indeed, if \( g(\gamma) \) was in the \( \ep \)-neighborhood of \( \gamma \) and had the opposite orientation, then \( g \) would have to move some point of \( \gamma \) a distance more than \( \ep \) (that is to say, there is not a enough room to turn \( \gamma \) around). 
Hence, \( U_\ep \) is an open neighborhood of \( f \) contained in \( U^+(\gamma, A) \), and therefore, \( U^+(\gamma, A) \) is open. 
\end{proof}

An important tool used by Mann \cite{MannAutomatic1, MannAutomatic2} in her arguments establishing the automatic continuity of various homeomorphism groups is the Fragmentation Lemma\index{Fragmentation Lemma} \cite[Proposition~2.3]{MannAutomatic1}, which is deduced from the  work of Edwards--Kirby (namely the proof of \cite[Corollary~1.3]{EdwardsDeformations}).   
The Fragmentation Lemma holds in all dimensions, and its proof is nontrivial.
Let us state the Fragmentation Lemma, though we only use it once, as we present and sketch a proof of a much weaker result (Proposition~\ref{prop:fragmentation_original}) that is enough for our purposes in most cases.

\begin{Thm}[Fragmentation Lemma]
\label{thm:fragmentation}
Let \( M \) be a compact manifold.
Given any finite open cover \( \{V_1, \ldots, V_m\} \) of \( M \), there exists a neighborhood \( U \) of the identity in \( H_0(M) \) such that each \( g \in U \) can be factored as a composition \( g = g_1g_2\ldots g_m \) with \( \supp(g_i) \subset V_i \) and \( g_i \in H_0(M) \) for \( i \in \{1, \ldots, m\} \). 
\qed
\end{Thm}

Let us now give a weakened version of the fragmentation lemma in dimension two. 

\begin{Prop}
\label{prop:fragmentation_original}
Let \( M \) be a compact 2-manifold, and let \( X \subset M \) be finite. 
Then, there exists a finite collection of closed embedded 2-disks \( \Delta_1, \ldots, \Delta_k \) in \( M \) and an open neighborhood \( U \) of the identity in \( H(M) \) such that 
\begin{itemize}
\item \( \{\Delta_1^\mathrm{o}, \ldots, \Delta_k^\mathrm{o}\} \) is an open cover of \( M \),
\item  \( |X \cap \Delta_i|\leq 1  \) for each \( i \in \{1, \ldots, k\} \), and 
\item every \( g \in U \) can be factored as a composition \( g = g_1g_2\ldots g_k \) with \( \supp(g_i) \subset \Delta^\mathrm{o}_i \) for \( i\in \{1, \ldots, k\} \). 
\end{itemize}
\end{Prop}

\begin{proof}[Sketch of the proof]
Let us assume that \( X \) is empty.
The idea is as follows: cover the 2-manifold with finitely many closed 2-disks \( \Delta_1, \ldots, \Delta_k \) with simple intersection patterns (i.e., the intersection of any two disks is either empty or a bigon). Slightly shrink each disk so that the resulting disks, labelled \( \Delta_1', \ldots, \Delta_k' \), continue to cover the surface.  
Then, letting \( A_i \subset \Delta_i^\mathrm{o} \) be an open annular neighborhood of \( \partial \Delta_i' \), define \( U = \bigcap_{i=1}^k U^+(\partial \Delta_i', A_i) \). 
Given \( g \in U \), there is an isotopy supported in \( A_1 \) sending \( g(\gamma_1) \) back to \( \gamma_1 \). 
We can then apply Alexander's trick to get a homeomorphism \( g_1 \) supported on \( \Delta_1 \) such that \( g_1^{-1} \circ g |_{\Delta_1'} \) is the identity. 
Proceed recursively.
If \( X \) is nonempty, then the same argument holds with the extra care of choosing the \( \Delta_i \) to satisfy the additional hypotheses. 
\end{proof}

We can now prove a theorem of Fisher \cite{FisherGroup} regarding compact 2-manifolds, which generalizes the work of Anderson \cite{AndersonAlgebraic} on the 2-sphere. 
Also, we will require one additional fact: two self-homeomorphisms of a manifold \( M \) are isotopic if and only if there is a path between them in \( \mathrm{Homeo}(M) \) (see Fox \cite{FoxTopologies}).

Given a closed 2-disk \( \Delta \) in a 2-manifold \( M \), recall that \( \Gamma_\Delta \) denotes the subgroup of \( \Homeo(M) \) normally generated by homeomorphisms supported in \( \Delta \). 

\begin{Thm}
\label{thm:fisher}
Let \( M \) be a 2-manifold, and let \( \Delta \subset M \) an embedded copy of the closed 2-disk.
Then,  \( \Gamma_\Delta \) is contained in every nontrivial normal subgroup of \( H(M) \) and is a simple group.
Moreover, if \( M \) is of finite type, then
\begin{enumerate}
\item \( H_0(M) \) is open and path connected, and
\item every element of \( H_0(M) \) can be factored as a composition of homeomorphisms each of which is supported in a stable Freudenthal subsurface.
\end{enumerate}
In particular, if \( M \) is compact, then \( \Gamma_\Delta = H_0(M) \) and \( H_0(M) \) is a simple group.
\end{Thm}

\begin{proof}
Every nontrivial element of \( H(M) \) must move some point of \( M \), and hence, some conjugate must move some point in \( \Delta^\mathrm{o} = \mathcal S(\Delta) \); in particular, by Anderson's Method (Proposition~\ref{prop:main2}), we see that \( \Gamma_\Delta \) is contained in the normal subgroup generated by every nontrivial element of \( H(M) \), and hence \( \Gamma_\Delta \) is contained in every nontrivial normal subgroup of \( H(M) \). 
In particular, every nontrivial element of \( \Gamma_\Delta \) normally generates \( \Gamma_\Delta \), and hence \( \Gamma_\Delta \) is a simple group.

For the remainder of the proof, assume that \( M \) is of finite type. 
Then, there exists a compact 2-manifold \( N \) and a finite subset \( X \subset N \) such that \( M \) is homeomorphic to \( N \ssm X \) (if \( M \) is compact, then \( N = M \) and \( X = \varnothing \)).
Then, by Corollary~\ref{cor:homeo_compactification}, the set \( \Homeo(N,X) = \{ f\in\Homeo(N) : f(X) = X\} \), equipped with the subspace topology, is isomorphic to \( \Homeo(M) \) as a topological group.

Proposition~\ref{prop:fragmentation_original} yields closed 2-disks \( \Delta_1, \ldots, \Delta_k \) embedded in \( N \) and an open neighborhood of the identity \( U \) in \( \Homeo(N) \) such that every \( f \in U \) can be factored as \( f=f_1 \circ \cdots \circ f_k \) with \( \supp(f_j) \subset \Delta_j \). 
Let \( U' = U \cap \Homeo(N,X) \) and let \( f \in U'  \).
Then, \( f= f_1 \circ \cdots \circ f_k \) with \( \supp(f_j) \subset \Delta_j \).
For each \( j \in \{1, \ldots, k\} \),  \( \Delta_j \) contains at most one element of \( X \), and hence it must be that \( f_j \in \Homeo(N,X) \). 
We can apply Alexander's trick to see that each \( f_j \) is isotopic to the identity relative to \( X \), and hence that \( f \) is isotopic to the identity relative to \( X \).  
It follows that \( U' \) is contained the path component of the identity in \( \Homeo(N,X) \); hence, the path component of the identity is open in \( \Homeo(N,X) \) (since it is a union of translates of \( U' \)).
Every open subgroup of a topological group is also closed, so the path component of the identity is a clopen subset of \( \Homeo_0(N,X) \), and hence equal to \( \Homeo_0(N,X) \).
Therefore, \( \Homeo_0(M) \) is open and path connected, establishing (1).

Now, observe that \( \Delta_j \ssm X \) is a stable Freudenthal subsurface of \( N\ssm X \).
Since \( U' \) is an open neighborhood of the identity in the connected topological group \( \Homeo_0(N,X) \), it must generate \( \Homeo_0(N,X) \) (a standard exercise); hence, every element of \( \Homeo_0(M) \) is a product of homeomorphisms each of which is supported in a stable Freudenthal subsurface, establishing (2).
In particular, if \( M \) is compact, then every Freudenthal subsurface is a disk, and hence \( \Homeo_0(M) = \Gamma_\Delta \). 
\end{proof}

Before continuing, we note that, for a compact 2-manifold \( M \), Hamstrom--Dyer \cite{HamstromRegular} appear to be the first to prove that \( \Homeo_0(M) \) is path connected; in fact, they prove the stronger fact that \( \Homeo_0(M) \) is locally contractible. 

For a compact 2-manifold \( M \), it follows from the simplicity of \( \Homeo_0(M) \) (Theorem~\ref{thm:fisher}) that \( \Homeo_0(M) \) is perfect.
For a non-compact finite-type 2-manifold, simplicity of \( \Homeo_0(M) \) does not hold since \( \Gamma_\Delta \) is a proper normal subgroup; however, we will need to know that \( \Homeo_0(M) \) is perfect for any finite-type 2-manifold.
This is a result of McDuff \cite[Corollary~1.3]{McDuffLattice}, and we will prove it below.
We note that McDuff's work relies on a preprint of Ling that appears to never have been published, but the relevant result appears in Ling's thesis \cite{LingAlgebraic}.
The proof of Proposition~\ref{prop:annulus_perfect} below is motivated by an argument in \cite{LingAlgebraic}.

We will use the term \emph{standard open (resp., closed)  annulus} to refer to a subset of \( \mathbb S^1 \times \br \) of the form \( \mathbb S^1 \times I \) for some precompact open (resp., closed) interval \( I \subset \br \). 

\begin{Prop}
\label{prop:annulus_perfect}
If \( f\in \Homeo_0(\mathbb S^1 \times \br) \), then there exist \( g_1, g_2, h_1, h_2 \in \Homeo_0(\mathbb S^1 \times \br) \) such that \( f = [g_1,h_1][g_2,h_2] \).
Moreover, if \( f \) restricts to the identity on a set of the form \( \mathbb S^1 \times [t, \infty) \) for some \( t \in \br 
\), then there exists \( t' > t \) such that the \( g_i \) and \( h_i \) can be chosen to restrict to the identity on \( \mathbb S^1 \times [t',\infty) \). 
\end{Prop}

\begin{proof}

The second statement of the proposition will be a consequence of the constructions in the proof, and we will not mention it explicitly. 
Let us begin by supposing that there exists a locally finite family \( \{A_n\}_{n\in\bz} \) of pairwise-disjoint standard closed annuli such that \( f \) is supported on \( \bigcup_{n\in\bz} A_n \). 
Let \( f_+ \) be the homeomorphism that agrees with \( f \) on \( \bigcup_{n\geq 0} A_n \) and is the identity outside of this union. 
Choose \( \tau_+ \in  H_0(\mathbb S^1 \times \br) \) such that \( \tau_+(A_{n}) = A_{n+1} \) for \( n \geq 0 \) and such that \( \tau_+|_{A_n} \) is the identity for \( n < 0 \) (for instance, \( \tau_+ \) can be chosen to be an appropriate conjugate of the homeomorphism  given by \( (z,t) \mapsto (z, 2^t-1) \) for \( t \geq 0 \) and that restricts to the identity on the negative reals).

Using the fact that, for each \( x \in \mathbb S^1 \times \br \), the set \( \{ k \in \bn : (\tau_+^k \circ f_+\circ \tau_+^{-k})(x) \neq x\} \) is finite, together with the fact that \( \{A_n\}_{n\in\bn} \) is locally finite, we can define the homeomorphism \[ F_+ = \prod_{n=0}^\infty \left( \tau_+^n \circ f \circ \tau_+^{-n} \right). \] 
A similar argument as in the proof of Proposition~\ref{prop:commutator} shows that \( f_+ = [F_+,\tau_+] \). 
Defining \( f_- \), \( \tau_- \) and \( F_- \) analogously, we can write \( f_- = [F_-, \tau_-] \) and  \( f = f_-\circ f_+ \).
Moreover, since the support of \( \tau_\pm \) and the support of \( F_\pm \) are disjoint from the supports of \( F_\mp \) and \( \tau_\mp \),  we can write \[ f = [F_-, \tau_-]\circ[F_+, \tau_+] = [F_-\circ F_+, \tau_-\circ \tau_+]. \]

Now, suppose \( f \) is arbitrary. 
We will decompose \( f \) into a product of two homeomorphisms, each of which is supported in a locally finite union of pairwise-disjoint standard closed annuli (the proof follows the proof of Le Roux--Mann's \cite[Lemma~4.4]{LeRouxStrong}). 
Let \( c_0 = \mathbb S^1 \times \{0\} \) and choose a standard open annulus \( A_0 \) in \( \mathbb S^1 \times \br \) such that \( c_0 \cup f(c_0) \subset A_0 \). 
We can then find \( n_1 \in \bn \) such that \( c_{\pm1} = \mathbb S^1 \times \{\pm n_1\} \) satisfies  \( (c_{\pm1} \cup f(c_{\pm1})) \cap \overline A_0 = \varnothing \) and such that \( n_1 \geq 1 \).
Choose  standard open annuli \( A_{\pm1} \) such that \( \overline A_{\pm1} \cap \overline A_0 = \varnothing \) and \( c_{\pm1} \cup f(c_{\pm1}) \subset A_{\pm1} \). 
Continuing in this fashion, we construct a locally finite sequence of pairwise-disjoint curves \( \{c_n\}_{n\in\bz} \) and a locally finite family of  standard  open annuli \( \{A_n\}_{n\in\bz} \) with pairwise-disjoint closures such that \( c_n \cup f(c_n) \subset A_n \). 

For each \( n \in \bz \), there is an ambient isotopy \( \Phi_n \) of \( \mathbb S^1 \times \br \) fixed on the complement of \( A_n \) such that \( \Phi_n(f(c_n),0)=f(c_n) \), \( \Phi_n(f(c_n),1) = c_n \), and letting \( h_n(x) = \Phi_n(x,1) \), we have \( h_n \circ f \) restricts to the identity on a standard closed annular neighborhood \( C_n \subset A_n \) of \( c_n \).
Let \( f_1 = \prod_{n\in\bz} h_n \).
Then, \( f_1 \) (as well as \( f_1^{-1} \)) is supported on \( \bigcup_{n\in\bz}  A_n \). 
Let \( f_2 = f_1\circ f \). 
Then, \( f_2 \) is supported on the complement of \( \bigcup_{n\in\bz} C_n \), and hence, both  \( f_1 \) and \( f_2 \) are supported on the union of a locally finite family of pairwise-disjoint standard closed annuli.
Therefore, \( f  = f_1^{-1}\circ f_2 \) can be written as the product of two commutators.  
\end{proof}

\begin{Cor}
\label{cor:unif}
\( \Homeo_0(\bS^2) \), \( \Homeo_0( \br^2 ) \), and  \( \Homeo_0(\bS^1 \times \br) \) are uniformly perfect. 
\end{Cor}

\begin{proof}
The case of \( \Homeo_0(\bS^1 \times \br) \) is immediate from Proposition~\ref{prop:annulus_perfect}.
Suppose \( f \in \Homeo_0(\br^2) \).
Then, there exists an embedded closed 2-disk \( \Delta \) containing \( 0 \) and \( f(0) \). 
Fix \( g \in \Homeo_0(\br^2) \) such that \( g(f(0))= 0 \) and \( \supp(g) \subset \Delta^\mathrm{o} \).
By Corollary~\ref{cor:commutator1}, \( g \) can be expressed as a commutator.  
Now we can view \( g\circ f \) as a self-homeomorphism of \( \br^2 \ssm\{0\} \), which is of course homeomorphic to \( \mathbb S^1 \times \br \).
Hence, we may express \( g \circ f \) as a product of two commutators, which allows us to express \( f \) as a product of three commutators.
Therefore, \( \Homeo_0(\br^2) \) is uniformly perfect. 
A similar argument works for \( \Homeo_0(\bS^2) \). 
\end{proof}

In fact, Tsuboi \cite{TsuboiHomeomorphism} has shown that the commutator width of \( H_0(\mathbb S^n) \) is one; building on Tsuboi's work, Bhat and the author \cite{BhatCommutator} recently showed that \( H_0(\br^n) \) also has commutator width one. 
In the next theorem, we show that \( H_0(M) \) is perfect when \( M \) is finite-type.  
We note that Bowden--Hensel--Webb \cite{BowdenQuasi-morphisms} have recently  shown that if \( M \) is a compact manifold of genus at least one, then \( H_0(M) \) is not uniformly perfect, which answered a long-standing question.

\begin{Thm}
\label{thm:finite-perfect}
Let \( M \) be a finite-type 2-manifold.
Then, \( \Homeo_0(M) \) is perfect.
\end{Thm}

\begin{proof}
Let \( f \in \Homeo_0(M) \).
By Theorem~\ref{thm:fisher}, there exists \( f_1,\ldots, f_k \in \Homeo_0(M) \) such that each \( f_i \) is supported in the interior of a stable Freudenthal subsurface \( \Delta_i \) and \( f = f_1 \cdots f_k \).
In particular, \( \Delta_i^\mathrm{o} \) is homeomorphic to either the open 2-disk or \( \mathbb S^1 \times \br \). 
By Anderson's Method in the former case or Proposition~\ref{prop:annulus_perfect} in the latter, we have that \( f_i \) is a product of commutators of homeomorphisms with support in \( \Delta_i^\mathrm{o} \). 
Hence, \( f \) can be expressed as a product of commutators and \( \Homeo_0(M) \) is perfect.
\end{proof}

From the above results regarding finite-type 2-manifolds, we record the following lemma:

\begin{Lem}
\label{lem:finite-commutator}
Let \( M \) be a 2-manifold and let \( N \) be an open finite-type submanifold of \( M \).
Then, there exists an open neighborhood \( U_N \) of the identity in \( \Homeo(M) \) such that every element of \( U_N \) that is supported in \( N^\mathrm{o} \) is contained in \( [\Homeo(M), \Homeo(M)] \). 
\end{Lem}

\begin{proof}
As \( N \) is finite-type, by Theorem~\ref{thm:fisher},  \( \Homeo_0(N) \) is open in \( H(N) \); in particular, there is a finite collection of compact sets \( K_1, \ldots, K_m \) of \( N \) admitting  precompact open neighborhoods \( W_1, \ldots, W_m \) in \( N \) such that \( \bigcap_{i=1}^k U(K_i,W_i) \) is contained in \( \Homeo_0(N) \). 
Now, viewing the sets \( U(K_i, W_i) \) in \( \Homeo(M) \), we let \( U_N = \bigcap_{i=1}^k U(K_i,W_i) \subset \Homeo(M) \). 
If \( f \in U_N \) such that \( \supp(f) \subset N^\mathrm{o} \), then \( f \) restricts to the identity on an open neighborhood of \( \partial N \); therefore, viewing \( f \) as a homeomorphism of \( N \) and applying Theorem~\ref{thm:finite-perfect},  there exists \( g, h \in \Homeo_0(N) \) such that \( [g,h] = f \). 
Now,  Proposition~\ref{prop:annulus_perfect} (and the construction in Theorem~\ref{thm:finite-perfect}) guarantees that \( g \) and \( h \) can be chose such that \( \supp(g), \supp(h) \subset N^\mathrm{o} \), and can therefore be extended by the identity to \( M \), implying \( f \in [\Homeo(M), \Homeo(M)] \). 
\end{proof}

To finish the subsection, we record a factorization result for non-compact 2-manifolds. 

\begin{Prop}
\label{prop:fragmentation2}
Let \( M \) be a non-compact 2-manifold.
Suppose \( N,  \Omega_1, \ldots, \Omega_k \) is an open cover of \( M \) consisting of connected sets satisfying the following conditions:
\begin{enumerate}[(i)]
\item the closure of \( N \) is a finite-type subsurface of \( M \),
\item  \( \overline \Omega_i \) is a Freudenthal subsurface for \( i \in \{1, \ldots, k\} \),
\item \( \overline \Omega_i \cap \overline \Omega_j = \varnothing \) for distinct \( i,j \in \{1, \ldots, k\} \), and
\item for each \( i \in \{1, \ldots, k\} \),  there exists a simple closed curve \( b_i \) contained in \( N \) such that \( \Omega_i \cap N \) is an open annular neighborhood of \( b_i \).
\end{enumerate}
Then, there exists an open neighborhood \( U \) of the identity in \( \Homeo(M) \) such that every element of \( U \) can be factored as \( \prod_{i=0}^k g_i \), where  \( \supp(g_i) \subset \Omega_i \) for \( i \in \{1, \ldots, k\} \), \( \supp(g_0) \subset N \), and \( g_0 \in [\Homeo(M), \Homeo(M)] \).
\end{Prop}

\begin{proof}
Let \( U_N \) be the neighborhood of the identity in \( \Homeo(M) \) given by Lemma~\ref{lem:finite-commutator}.
By shrinking \( U_N \), we may assume that there exist compact subsets \( K_1, \ldots, K_m \) of \( M \) contained in \( N \) admitting relatively compact open neighborhoods \( W_1, \ldots, W_m \), also contained in \( N \), such that \( U_N = \bigcap_{i=1}^m U(K_i,W_i) \). 
For each \( i \in \{1, \ldots, k\} \), let \( R_i \) be an open annular neighborhood of \( b_i \) whose closure is contained in \( \Omega_i \cap N \) and is disjoint from \( \overline W_i \) for \( i \in \{1, \ldots,m\} \).

Let
\[
U = U_N \cap \left[ \bigcap_{i=1}^k U^+(b_i, R_i) \right].
\]
Note that \( U \) is nonempty as it contains the identity. 
Fix \( f \in U \).
By the definition of \( U \), we have  \( f(b_i) \subset R_i \) and \( f(b_i) \) is isotopic to \( b_i \) for each \( i \in \{1, \ldots, k\} \) via an isotopy supported in \( R_i \).
In particular, there exists \( g \in \Homeo(M) \) supported in \( R_1 \cup \cdots \cup R_k \) such that \( g\circ f \) restricts to the identity on an open annular neighborhood of \( b_i \) for all \( i \in \{1, \ldots, k\} \).   
Therefore, we can write \(  g \circ f = g_0'g_1g_2\cdots g_k  \), where  \( g_i \) is supported on \( \Omega_i \) for \( i \in \{1, \ldots, k\} \) and \( g_0' \) is supported on  \( N \).

Since \( g \) is supported in \( R_1 \cup \cdots \cup R_k \), it follows that \( g_0 = g^{-1}\circ g_0' \) is supported in \( N \) and restricts to the identity on an open neighborhood of \( \partial N \).
Moreover, \( g_0|_{K_i} = f|_{K_i} \), which implies \( g_0 \in U_N \) and hence \( g_0 \in [\Homeo(M), \Homeo(M)] \).
Therefore, \( f = \prod_{i=0}^k g_i \) is the desired decomposition.  
\end{proof}

\subsection{Defining the mapping class group}

Given a 2-manifold \( M \), we define the \emph{mappping class group}\index{mapping class group}\footnote{The mapping class group is sometimes referred to as the \emph{homeotopy group} in the literature.} of \( M \), denoted \( \mcg(M) \), to be the group \( \Homeo(M) / \Homeo_0(M) \). 
In the literature, the mapping class group is often defined to be the group of isotopy classes of self-homeomorphisms of \( M \), which agrees with the group \( \pi_0(\Homeo(M)) \). 
We finish this section by proving that these definitions agree.

We have already seen in Theorem~\ref{thm:fisher} that if \( M \) is of finite type, then \( \Homeo_0(M) \) is path connected, which implies that \( \mcg(M) = \pi_0(\Homeo(M)) \). 
It is therefore left to check the infinite-type case, for which we require the following proposition, established by Hern\'andez--Morales--Valdez \cite{HernandezAlexander} in the orientable case and Hern\'andez--Hidber \cite{HernandezFirst} in the non-orientable case (also see the work of Shapiro \cite{ShapiroAlexander} for a more general result).

\begin{Prop}
\label{prop:identity}
If a self-homeomorphism of an infinite-type 2-manifold fixes the isotopy class of every simple closed curve, then it is isotopic to the identity.
\qed
\end{Prop}

The idea of the proof is as follows: fix a collection of simple closed curves that fill\footnote{A set of simple closed curves on a 2-manifold \emph{fill} if the complementary components of their union are either disks or once-punctured disks.}  the 2-manifold and whose intersection pattern has ``simple combinatorics''.  
Then, given some homeomorphism fixing the isotopy class of every simple closed curve, recursively perform isotopies---with increasing support---mapping the image of an increasing finite subset of curves  in the collection back to itself.
One does this in a such a way that the direct limit of the isotopies can be taken to obtain an isotopy returning the image of each curve to its original position simultaneously.  
It is then possible to apply the Alexander Trick simultaneously to each disk in the complement of the collection of curves. 
This is referred to as the Alexander Method (see \cite[Section~2.3]{FarbPrimer} for a discussion in the finite-type case).

It is worth noting that Proposition~\ref{prop:identity} is false for a small number of finite-type 2-manifolds: for example, the hyperelliptic involution of the compact orientable genus two 2-manifold fixes the isotopy class of every simple closed curve. 

\begin{Thm}
\label{thm:mcg_defined}
Let \( M \) be a 2-manifold.
Then, \( \Homeo_0(M) \) is path connected. 
\end{Thm}

\begin{proof}
The finite-type case is handled in Theorem~\ref{thm:fisher}, so we may assume that \( M \) is of infinite-type. 
Let \( c \) be a simple closed curve on \( M \).
If \( c \) is two-sided, let \( A_c \) be an annular neighborhood of \( c \); otherwise, let \( A_c \) be a M\"obius band embedded in \( M \) with \( c \) as its core curve. 
Let \( U_c = U(c, A_c) \cap \Homeo_0(M) \).
Note that \( U_c \) contains the identity and that every element of \( U_c \) fixes the isotopy class of \( c \). 

Fix \( f \in \Homeo_0(M) \).
Then, as every open neighborhood of the identity in a connected topological group generates the group, we can write \( f = f_1f_2\cdots f_k \) with each \( f_j \in U_c \). 
Therefore, \( f \) must fix the isotopy class of \( c \). 
As \( c \) was arbitrary, it follows that every homeomorphism in \( \Homeo_0(M) \) preserves the isotopy class of every simple closed curve in \( M \); hence, by Proposition~\ref{prop:identity}, \( f \) is isotopic to the identity, and the result follows. 
\end{proof}

\begin{Cor}
Let \( M \) be a 2-manifold.
Then, \( \mcg(M) = \pi_0(\Homeo(M)) \). 
\qed
\end{Cor}


\section{Equivalent notions of self-similarity}
\label{sec:self-similar}

As noted in the introduction, the definition of a self-similar 2-manifold given differs from the original definition in \cite{MalesteinSelf}. 
The main purpose of this section is to establish the equivalence of the two definitions for non-compact 2-manifolds.
Along the way, we will introduce the definition of maximally stable ends, half-spaces, and uniquely self-similar 2-manifolds.

Let us recall the definition of self-similarity given previously and show the equivalence with the definition given by Malestein--Tao. 
First recall that a subset \( Y \) of a topological space \( X \) is \emph{displaceable} if there exists a homeomorphism \( f \co X \to X \) such that \( f(Y) \cap Y = \varnothing \).
A 2-manifold \( M \) is \emph{self-similar} if every proper compact subset of \( M \) is displaceable and every separating simple closed curve has a complementary component homeomorphic to \( M \) with a point removed.

\begin{Def}
\label{def:maximally-stable}
An end \( \mu \) of a 2-manifold \( M \) is \emph{maximally stable}\index{end!maximally stable|textbf} if \( \sE(M) \) is a stable neighborhood of \( \mu \). 
\end{Def}

Note that the existence of a maximally stable end is equivalent to \( \sE(M) \) being stable. 
The following proposition establishes the equivalence of the definition of self-similar given here and the one given in \cite{MalesteinSelf} for non-compact 2-manifolds.

\begin{Prop}
\label{prop:noncompactdef}
Let \( M \) be a non-compact 2-manifold. 
Then, \( M \) is self-similar if and only if every compact subset of \( M \) is displaceable and \( M \) has a maximally stable end. 
\end{Prop}

\begin{proof}
First suppose that \( M \) has a maximally stable end \( \mu \) and every compact subset of \( M \) is displaceable. 
Let \( c \) be a separating simple closed curve in \( M \).
Then, there is a component of \( M \ssm c \) whose closure \( \Delta \) satisfies \( \mu \in \widehat\Delta \). 
Then, by Proposition~\ref{prop:well-defined}, \( \widehat \Delta \) is homeomorphic to \( \sE(M) \), and \( \Delta \) is a stable Freudenthal subsurface. 
By the classification of surfaces, \( \Delta \) is homeomorphic to \( M \) with an open disk removed, and hence \( M \) is self-similar. 

Conversely, suppose that \( M \) is self-similar. 
Let \( \sU_1 \) and \( \sU_2 \) be clopen subsets of \( \sE(M) \) such that \( \sE(M) = \sU_1 \cup \sU_2 \). 
We can then find a separating simple closed curve \( c \) such that \( \sU_i = \widehat \Omega_i \), where \( \Omega_1 \) and \( \Omega_2 \) are the complement components of \( M \ssm c \).
The self-similarity of \( M \) guarantees that at least one of \( \sU_1 \) and \( \sU_2 \) is homeomorphic to \( \sE(M) \). 
In particular, by Proposition~\ref{prop:equiv_def}, \( \sE(M) \) is stable, and hence contains a maximally stable end \( \mu \).
\end{proof}

\begin{Def}
\label{def:half-space}
A \emph{half-space}\index{half-space|textbf} of a 2-manifold \( M \) is a dividing Freudenthal subsurface \( H \) such that the closure of \( M \ssm H \) is homeomorphic to \( H \).
\end{Def} 

Recall a self-similar 2-manifold is \emph{perfectly self-similar} if  \( M \# M \) is homeomorphic to \( M \).

\begin{Prop}
\label{prop:half-space}
A 2-manifold is perfectly self-similar if and only if it contains a half-space.
\end{Prop}

\begin{proof}
Let \( M \) be a 2-manifold.
If \( M \) is compact, then \( M \) contains a half-space if and only if \( M \) is homeomorphic to the 2-sphere.
So we may now assume that \( M \) is not compact.

First, let us assume that \( M \) contains a half-space \( H \). 
Let \( H' \) be the closure of \( M \ssm H \); by assumption, \( H' \) is a half-space.
In particular, \( \widehat H \) and \( \widehat H' \) are disjoint and homeomorphic. 
By Proposition~\ref{prop:disjoint-union}, \( \sE(M) = \widehat H \cup \widehat H' \) is homeomorphic to \( \widehat H \).
By classification of surfaces, we have that \( H \) and \( H' \) are both homeomorphic to \( M \) with an open disk removed, and hence, \( M \) is homeomorphic to \( M \# M \).

Conversely, suppose that \( M \) is perfectly self-similar. 
Then, as \( M \) is homeomorphic to \( M \# M \), there exists a subsurface \( H \) in \( M \) that is homeomorphic to \( M \) with an open disk removed and such that the closure of \( M \ssm H \) is homeomorphic to \( H \). 
Note that \( H \) is a trim Freudenthal subsurface. 
Let us argue that \( H \) is stable. 
Let \( \sU_1 \) and \( \sU_2 \) be a clopen subsets of \( \sE(M) \) such that \( \sE(M) = \sU_1 \cup \sU_2 \). 
Let \( \Delta \) be a Freudenthal subsurface such that \( \widehat \Delta = \sU_1 \). 
Then, \( K = \partial \Delta  \cup \partial H \) is a compact subset of \( M \), and hence, by assumption, there exists \( f \in \Homeo(M) \) such that \( f(K) \cap K = \varnothing \). 
It follows that either \( f(H) \subset \Delta \), \( f(H) \subset M \ssm \Delta \), \( \Delta \subset f(H) \), or \( M \ssm \Delta \subset f(H) \); in any case, \( \sU_1 \) or \( \sU_2 \) contains an open set homeomorphic to \( \widehat H \) and hence homeomorphic to \( \sE(M) \). 
By Proposition~\ref{prop:equiv_def}, \( \sE(M) \) is stable, and as \( \widehat H \) is homeomorphic to \( \sE(M) \), we see that \( H \) is a stable Freudenthal subsurface. 
To finish, observe that \( \sE(M) \), and hence \( \widehat H \), is homeomorphic \( \widehat H \times \{1,2\} \).
This shows that \( |\widehat H| > 1 \), and so by Proposition~\ref{prop:dichotomy}, \( \mathcal S(H) \) is infinite and \( H \) is dividing.
In particular, \( H \) is a half-space.
\end{proof}

Let us now exhibit a basic dichotomy of self-similar 2-manifolds. 

\begin{Lem}
\label{lem:uniformlyinfinite}
A self-similar 2-manifold is perfectly self-similar if and only if the set of maximally stable ends is perfect.
\end{Lem}

\begin{proof}
Let \( M \) be a 2-manifold.
If \( M \) is non-compact and perfectly self-similar, then \( M \) has a maximally stable end by Proposition~\ref{prop:noncompactdef}. 
Then, as \( M \) is homeomorphic to \( M \# M \), it must have at least two maximally stable ends, and hence a perfect set worth by Proposition~\ref{prop:dichotomy}. 
Conversely, using Proposition~\ref{prop:well-defined} in the non-compact case, we see that \( M \) admits a half-space and hence is perfectly self-similar by Proposition~\ref{prop:half-space}.
\end{proof}

By Proposition~\ref{prop:dichotomy}, we know that, for a self-similar 2-manifold \( M \), the set of maximally stable ends is either a singleton or a perfect subset of \( \sE(M) \), which motivates the following definition.

\begin{Def}
A 2-manifold is \emph{uniquely self-similar}\index{2-manifold!self-similar!uniquely|textbf} if it is self-similar and has a unique maximally stable end. 
\end{Def}

With this, we end with the following consequence of Lemma~\ref{lem:uniformlyinfinite}, which justifies the definition given in the introduction.

\begin{Prop}
Every self-similar 2-manifold is either uniquely or perfectly self-similar. 
\qed
\end{Prop}


\section{Normal generation and purity}
\label{sec:normal}

The goal of this section is to prove a strengthened version of a theorem of Malestein--Tao \cite[Theorem~A]{MalesteinSelf}.

\begin{Thm}
\label{thm:main2}
Let \( M \) be a perfectly self-similar 2-manifold, and let \( H = H(M) \). 
\begin{enumerate}[(1)]
\item \( H \) is uniformly perfect and its commutator width is at most two. 
\item
If \( g \in H \) displaces a half-space of \( M \), then every element of \( H \)  can be expressed as a product of  at most eight conjugates of \( g \) and \( g^{-1} \). In particular, \( H \) is normally generated by \( g \).
\item
If \( n \in \bn \) such that \( n \geq 2 \), then every element of \( H \) can be expressed as the product of at most eight elements of order \( n \).
\end{enumerate}
\end{Thm}

Before we prove Theorem~\ref{thm:main2}, we need an additional result.
Recall that, for a half-space \( \Delta \) of \( M \), we define \( \Gamma_\Delta \) to be the subgroup of \( \Homeo(M) \) generated by the conjugates of homeomorphisms of \( M \) supported in \( \Delta \). 
By Lemma~\ref{lem:homogeneous}, if \( \Delta' \) is another half space of \( M \), then \( \Gamma_\Delta = \Gamma_{\Delta'} \). 
The following theorem is a version of \cite[Theorem~3.11]{MalesteinSelf}, which allows for non-orientable 2-manifolds.

\begin{Thm}
\label{thm:3}
Let \( M \) be a perfectly self-similar 2-manifold.
Then, for every \( h \in H(M) \) there exist \( h_1, h_2 \in H(M) \) and half-spaces \( \Delta_1 \) and \( \Delta_2 \) such that \( h = h_1\circ h_2 \).
Moreover, given any half-space \( \Delta \), every element of \( H(M) \) can be factored as the product of two homeomorphisms, each of which is conjugate to a homeomorphism with support in \( \Delta \); in particular, \( H(M) = \Gamma_\Delta \).  
\end{Thm}

\begin{proof}
Let \( h \in H(M) \).
Then, by taking a shrinking nested sequence of half-spaces, we are guaranteed to find a half-space \( D_1 \) such that the complement of \( h(D_1) \cup D_1 \) contains a half-space. 
In particular, we can find half-spaces \( D_2 \) and \( D_3 \) such that \(  D_2 \cap ( D_1 \cup h( D_1)) = \varnothing \) and  \(  D_1, h( D_1),  D_2 \subset D_3^\mathrm{o} \) . 
Therefore, by Lemma~\ref{lem:homogeneous}, there exists \( f_1 \in \Homeo(M) \) with support in \( D_3 \) such that \( f_1(h(D_1))=D_1 \).

If \( M \) is orientable, then so is \( f_1 \), and hence we can assume that \( f_1 \circ h \) restricts to the identity on \( \partial D_1 \). 
However, if \( M \) is non-orientable, this is not guaranteed by Lemma~\ref{lem:homogeneous} as \( f_1 \) may be orientation-reversing in an annular neighborhood of \( \partial D_1 \). 
However, if this is case, then as both \( D_1 \) and \( D_2 \) are not orientable, we may find a simple closed curve \( c \) such that the double slide \( s_{\partial D_1, c} \) is supported in \( D_3 \).
Then, replacing \( f_1 \) with \( s_{\partial D_1,c} \circ f_1 \), we may assume that \( f_1 \circ h \) restricts to the identity on \( \partial D_1 \) in both the orientable and non-orientable cases. 

Let \( f_2 \in \Homeo(M) \) such that \( f_2 \) agrees with \( (f_1\circ h)^{-1} \) on \( D_1 \) and is the identity on the complement of \( D_1 \). 
Then,  \( f_2 \) is supported in \(  D_1 \subset D_3 \) and \( f_2\circ f_1 \circ h \) restricts to the identity on \(  D_1 \).
Let \( \Delta_1 \) be the half-space obtained by taking the closure of \( M \ssm D_1 \).
Then, \( h_1 = f_2 \circ f_1 \circ h \) is supported in \( \Delta_1 \).  
Now, note that \( h_2 = (f_2\circ f_1)^{-1} \) is supported in \( D_3 \).
Set \( \Delta_2 = D_3 \).
Then, by construction, \( \partial \Delta_1 \subset \Delta_2^\mathrm{o} \), \( \Delta_1 \cap \Delta_2 \) contains a half-space (namely \( D_2 \)), \( h_i \) is supported in \( \Delta_i \), and \( h = h_1h_2 \). 

To finish, let \( \Delta \) be a half-space. Then, as \( h_i \) is supported in a half-space, it is conjugate to a homeomorphism supported in \( \Delta \) by Lemma~\ref{lem:homogeneous}, and hence \( h \in \Gamma_\Delta \).
\end{proof}

\begin{proof}[Proof of Theorem~\ref{thm:main2}]
Let \( f \in H \) and let \( D \) be a half-space of \( M \) such that \( g(D) \cap D = \varnothing \). 
Then, by Theorem~\ref{thm:3}, \( f  = h_1h_2 \), where \( h_1, h_2 \in \Gamma_D \). 
By Corollary~\ref{cor:commutator1}, each of \( h_1 \) and \( h_2 \) can be expressed as a commutator of two elements in \( \Gamma_D \).
By Proposition~\ref{prop:main2}, each of \( h_1 \) and \( h_2 \) can be expressed as the product of four conjugates of \( g \). 

For the third statement, let \( \widehat\bc \) denote the Riemann sphere and let \( \Sigma_0 \subset \widehat\bc \) be obtained by removing \( n \) pairwise-disjoint open disks from \( \widehat\bc \) such that each of the disks has the same radii and such that the center of each disk is an \( n^{\text{th}} \) root of unity.  
Then the rotation \( z \mapsto e^{2\pi i/n} z \) restricts to an order \( n \) element \( \vp \) of \( \Homeo(\Sigma_0) \). 
For \( i \in \{1, \ldots, n\} \), let \( \Sigma_i \) be a copy of a half-space in \( M \).
We can then construct the surface \( M'  = (\Sigma_0 \sqcup \Sigma_1 \sqcup \cdots \sqcup \Sigma_n) / \sim \), where \( \sim \) identifies  \( \partial \Sigma_i \) with a component of \( \partial \Sigma_0 \) via a homeomorphism of the circle  in such a way that \( \vp \) extends to an order \( n \) homeomorphism of \( M' \).
Observe that \( \sE(M') \) is a disjoint union of \( n \) copies of \( \sE(M) \), and so by Proposition~\ref{prop:disjoint-union}, \( \sE(M') \) is homeomorphic to \( \sE(M) \), and moreover, \( \Sigma_i \), viewed as a Freudenthanl subsurface of \( M' \), is a half-space of \( M' \). 
Therefore, by the classification of surfaces, there exists a homeomorphism \( h \co M \to M' \).
To finish, \( h^{-1} \circ \vp \circ h \) is an order \( n \) element of \( \Homeo(M) \) displacing a half-space of \( M \), and the result follows from (2).
\end{proof}

As a corollary, we obtain a classical result for which the author is not certain how to attribute.
We refer the interested reader to the history discussion in the introduction of \cite{FisherGroup}.
We note that the corollary could have just as easily appeared in Section~\ref{sec:homeomorphism-groups} as a consequence of Theorem~\ref{thm:fisher} and Alexander's trick. 

\begin{Cor}
\( \Homeo(\mathbb S^2) \) is path connected.
\end{Cor}

\begin{proof}
By Theorem~\ref{thm:main2}, \( \Homeo(\bS^2) \) is simple.
But, by Theorem~\ref{thm:fisher}, \( \Homeo_0(\bS^2) \) is a normal  subgroup of \( \Homeo(\bS^2) \), and hence, they must be equal.
Also, by Theorem~\ref{thm:main2}, \( \Homeo_0(\bS^2) \) is path connected, and hence so is \( \Homeo(\bS^2) \).
\end{proof}

We finish the section by proving the corollaries mentioned in Section~\ref{sec:overview}:

\begin{proof}[Proof of the Purity Theorem, Corollary~\ref{cor:countable}]
Let \( N \) be a proper normal subgroup of \( G \) and let \( \Gamma \) be the stabilizer of the maximally stable ends of \( M \). 
By continuity, if \( g \in G \ssm \Gamma \), then there exists a half-plane \( D \) for which \( g(D) \cap D = \varnothing \) (this follows from the fact that a maximally stable end has a neighborhood basis consisting of half-planes), and hence, by Theorem~\ref{thm:main2}, every element of \( G \) is a product of conjugates of \( g \). 
In particular, as \( N \) is a proper normal subgroup, we must have that \( N < \Gamma \). 
Moreover, the set of maximally stable ends of \( M \) is homeomorphic to \( 2^\bn \) (Proposition~\ref{prop:dichotomy}), and the action of \( G \) on this set is transitive (Proposition~\ref{prop:well-defined}). 
This action induces an isomorphism between \( G / \Gamma \) and \( \mathrm{Homeo}(2^\bn) \). 
It follows that \( G / \Gamma \), and hence \( G/N \), is uncountable. 
\end{proof}

\begin{proof}[Proof of Corollary~\ref{cor:torsion}]
The proof is identical to the one given in \cite[Lemma~2.5]{CalegariNormal}.
Every torsion element of \( \mcg(M) \) has a finite-order representative in \( H(M) \) by \cite[Theorem~2]{AftonNielsen}, so it is enough to consider \( H(M) \). 
Every homeomorphism of \( M \) extends to a homeomorphism of the 2-sphere.
This follows from basic properties of end spaces and the associated Freudenthal compactification (we have not discussed the Freudenthal compactification in detail:~for a self-contained and succinct treatment, see \cite[Section~2.1]{FernandezEnds}). 
It is well known that every finite-order orientation-preserving homeomorphism of the 2-sphere is conjugate to a rotation (for instance, this can be deduced from \cite[Theorem~2.8]{EpsteinPointwise}). 
Since a nontrivial rotation has exactly two fixed points, we can conclude that any finite-order element of \( H(M) \) does not act trivially on the (infinite) set of maximally stable ends of \( M \), and hence, by Theorem~\ref{thm:main2}, must normally generate all of \( H(M) \).
\end{proof}


\section{Strong distortion}
\label{sec:distortion}

Recall that a group \( G \) is \emph{strongly distorted} if there exists \( m \in \bn \) and a sequence \( \{w_n\}_{n\in\bn} \subset \bn \)  such that, given any sequence \( \{g_n\}_{n\in\bn} \) in \( G \), there exists a set \( S \subset G \) of cardinality at most \( m \) with \( g_n \in S^{w_n} \) for every \( n \in \bn \). 
Here, we establish the strong distortion property for homeomorphism groups of perfectly self-similar 2-manifolds. 
We use the proof of Calegari--Freedman \cite{CalegariDistortion} for spheres as the outline and adapt where necessary.
The strategy can be adapted to a fairly general setting (for instance, see \cite[Construction~2.3]{LeRouxStrong}) and goes back to the work of Fisher \cite{FisherGroup}.

\begin{Lem}
\label{lem:factor2}
Let \( M \) be a perfectly self-similar 2-manifold. 
If \( D_1 \) and \( D_2 \) are half-spaces such that  \( D_1 \cap D_2 \) contains a half-space and \( \partial D_i \subset D_{3-i}^\mathrm{o} \) for \( i \in \{1,2\} \), then every element of \( H(M) \) can be factored as a product of six homeomorphisms, each of which is supported in either \( D_1 \) or \( D_2 \). 
\end{Lem}

\begin{proof}
Let \( h \in H(M) \). 
By Theorem~\ref{thm:3}, there exists \( h_1, h_2 \in H(M) \) supported in half-planes \( \Delta_1, \Delta_2 \subset M \), respectively, such that \( h = h_1\circ h_2 \).  
Observe that, by  assumption, \( M = D_1^\mathrm{o} \cup D_2^\mathrm{o} \).
Therefore, up to relabelling, we may assume that \( D_1^\mathrm{o} \ssm \Delta_1 \) contains a half-space \( D' \); note that \( h_1 \) restricts to the identity on \( D' \). 
If \( D' \ssm D_2 \) does not contain a half-space, then there exists a half-space \( D \) contained in \( D_1^\mathrm{o} \ssm (D' \cup D_2) \); otherwise, we can choose a half-space \( D \) contained in \( D' \ssm D_2 \). 
In the latter case, set \( f_1 \) to be the identity, and in the former case, using Lemma~\ref{lem:homogeneous}, choose \( f_1 \in H(M) \) supported in \( D_1 \) such that \( f_1(D') = D \).
Then, in either case, \( f_1 \circ h_1 \circ f_1^{-1} \) restricts to the identity on \( D \). 

\begin{figure}
\includegraphics{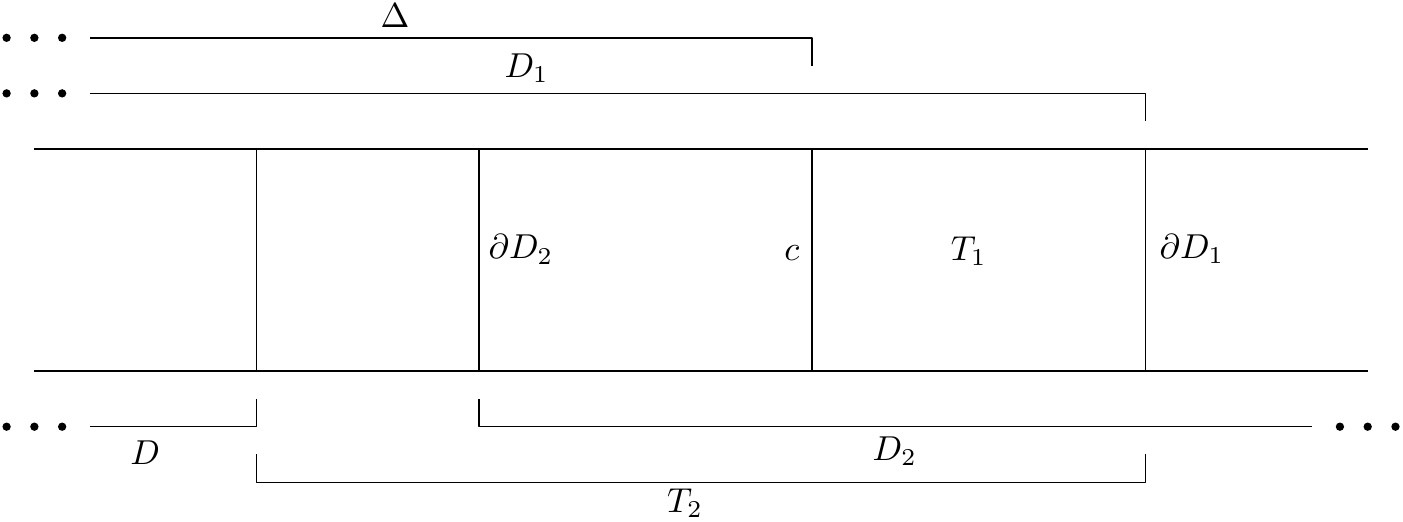}
\caption{A schematic of the subsurfaces in the proof of Lemma~\ref{lem:factor2}.}
\label{fig:distortion}
\end{figure}

Choose a separating simple closed curve \( c \) such that \( c \) is contained in the interior of \( D_1 \cap D_2 \), separates \( \partial D_1 \) from \( \partial D_2 \), and the component of \( (D_1 \cap D_2) \ssm c \) containing \( \partial D_1 \) contains a half-space; let the closure of this component be labelled \( T_1 \).
Let \( T_2 \) denote the closure of \( (M \ssm D) \cap  D_1 \), and let \( \Delta \) denote the closure of component of \( M \ssm c \) containing \( D \).
Then, by Proposition~\ref{prop:same} and the classification of surfaces, we have that \( D \) is homeomorphic to \( \Delta \) and \( T_1 \) is homeomorphic to \( T_2 \) (see Figure~\ref{fig:distortion} for a schematic).

Therefore, by the change of coordinates principle, there exists an ambient homeomorphism \( f_2 \) of \( M \) supported in \( D_1 \) that maps \( D \) onto \( \Delta \) and \( T_2 \) onto \( T_1 \); in particular, \( f_2( M \ssm D ) \subset D_2 \).
Let \( g_1 = f_2 \circ f_1 \), and note that \( \supp(g_1) \subset D_1 \).
Then, as \( \supp(g_1 \circ h_1 \circ g_1^{-1}) = f_2(\supp(f_1\circ h_1 \circ f_1^{-1}))\), \( \supp(f_1 \circ h_1 \circ f_1^{-1}) \subset M \ssm D^\mathrm{o} \), and \( f_2(M\ssm D ) \subset D_2 \), we have \( \supp(g_1 \circ h_1 \circ g_1^{-1})  \subset D_2 \).
Writing
\[
h_1 = g_1^{-1} \circ (g_1 \circ h_1 \circ g_1^{-1}) \circ g_1,
\]
factors \( h_1 \) as a product of three homeomorphisms, each of which is supported in \( D_1 \) or \( D_2 \). 
Proceeding in a similar fashion for \( h_2 \), we can factor \( h \) as desired. 
\end{proof}

\begin{Thm}
\label{thm:distortion}
Let \( M \) be a perfectly self-similar 2-manifold. 
Then, \( H(M) \) is strongly distorted. 
\end{Thm}

\begin{proof}
Let \( \{ h_n\}_{n\in\bn} \) be a sequence in \( H(M) \). 
Fix half-spaces \( D_1 \) and \( D_2 \) in \( M \) satisfying the hypotheses of Lemma~\ref{lem:factor2}.
Then, applying Lemma~\ref{lem:factor2}, for each \( n \in \bn \), we can write
\[
h_n = h_{n,1} \circ h_{n,2} \circ \cdots \circ h_{n,6}
\]
where each \( h_{n,\ell} \) is supported in  \( D_1 \) or \( D_2 \). 

Let \( \Delta \) be a half-space in \( M \).
Consider the following 2-manifold \( M' \):
Begin with \( \br^2 \).
Then, for each \( i,j \in \bz \), remove the open ball of radius \( 1/4 \) in \( \br^2 \) centered at \( (i,j) \). 
Now for each boundary component of the resulting surface, glue in a copy of \( \Delta \) in such a way that the standard \( \bz^2 \) action on \( \br^2 \) extends to an action on the resulting manifold, which we call \( M ' \).
Label the copy of \( \Delta \) glued in place of the disk centered at \( (i,j) \) by \( \Delta_{i_j} \). 
The end space of \( M ' \) is homeomorphic to the one-point compactification of \( \Ends(M) \times \bz^2 \), where \( \bz^2 \) is given the discrete topology.
Therefore, by Proposition~\ref{prop:disjoint-union} and the classification of surfaces, \( M ' \) is homeomorphic to \( M \).
Let us identify \( M \) with \( M ' \).

We will now proceed with a modified version of Anderson's method. 
Let \( \vp, \psi \in H(M) \) be homeomorphisms such that
\begin{itemize}
\item \( \vp(\Delta_{i,j}) = \Delta_{i+1,j} \) for all \( i,j \in \bz \),
\item  \( \psi(\Delta_{0,j}) = \Delta_{0,j+1} \) for all \( j \in \bz \), and
\item  \( \psi \) restricts to the identity on \( \Delta_{i,j} \) for all \( j \in \bz \) and for all \( i \in \bz \ssm \{0\} \). 
\end{itemize}

For \( k \in \{1,2\} \), choose \( \eta_k \in H(M) \) such that \( \eta_k(D_k) = \Delta_{0,0} \). 
For \( n \in \bz \) and \( \ell \in \{1, \ldots, 6\} \), set \( \eta_{n,\ell} = \eta_1 \) if \( h_{n,\ell} \) is supported in \( D_1 \); otherwise set \( \eta_{n,\ell} = \eta_2 \). 
For \( \ell \in \{1, \ldots, 6\} \), define \( f_\ell \in H(M) \) by
\[
f_\ell = \prod_{i=0}^\infty \prod_{j=0}^\infty  \left[ (\vp^i \circ \psi^j \circ \eta_{i,\ell}) \circ h_{i,\ell} \circ (\vp^i \circ \psi^j \circ \eta_{i,\ell})^{-1} \right].
\]
Observe that for \( i,j \in \bn\cup\{0\} \), the product \( \vp^i \circ \psi^j \circ \eta_{i,\ell} \) maps one of \( D_1 \) or \( D_2 \) (depending on \( \eta_{i, \ell} \)) onto \( \Delta_{i,j} \). 
Therefore, intuitively, for  \( i \in \bn\cup\{0\} \),  \( f_\ell \) is performing \( h_{i,\ell} \) on  \( \Delta_{i,j} \) for all \( j \in \bn \cup \{0\} \). 
In particular, the support of \( f_\ell \) is \( \bigcup_{i,j \in \bn\cup\{0\}} \Delta_{i,j} \). 
Now, for \( n \in \bn \), consider 
\[
g_{n,\ell} 	=  \vp^{-n} \circ f_\ell \circ \vp^n,
\]
which intuitively translates the support of  \( f_\ell \) so that \( g_{n,\ell} \) is performing \( h_{n,\ell} \) on \( \Delta_{0,j} \) for all \( j \in \bn \cup\{0\} \). 
Observe that \( \psi \circ g_{n,\ell} \circ \psi^{-1} \) is the identity on \( D_{0,0} \) but otherwise agrees with \( g_{n,\ell} \).
Therefore, the same computation as in Propositions~\ref{prop:commutator} yields
\[
h_{n,\ell} = \eta_{n,\ell} \circ [g_{n,\ell}, \psi ] \circ \eta_{n,\ell}^{-1}.
\]
Let \( S \subset H(M) \) be the set consisting of the elements \( f_1, \ldots, f_6, \vp, \psi, \eta_1, \eta_2 \), and their inverses.
Then, we have shown that for each \( n \in \bn \) and \( \ell \in \{1, \ldots, 6 \} \), the element \( h_{n,\ell} \) can be written as a word in \( S \) of length \( 4n + 6 \), and hence \( h_n \) can be written as a word of length \( 24n+36 \). 
In particular, as the sequence \( \{h_n\}_{n\in\bn} \) was arbitrary, we have that \( H(M) \) is strongly distorted, with \( m = 20 \) and \( w_n = 24n+36 \).
\end{proof}


\section{Coarse boundedness}

Up to this point, we have mainly considered perfectly self-similar 2-manifolds, but here (and in the next section) we will consider the larger class of self-similar 2-manifolds. 

Recall that a topological group is \emph{coarsely bounded} if, whenever the group acts continuously by isometries on a metric space, every orbit is bounded. 
Mann--Rafi \cite{MannLarge} proved that the mapping class group of a self-similar 2-manifold is coarsely bounded.  
With only slight modifications to Mann--Rafi's proof, we prove this theorem in the setting of homeomorphism groups.

\begin{Thm}
\label{thm:cb}
If \( M \) is a self-similar 2-manifold, then  \( \Homeo(M) \) and \( \mcg(M) \) are coarsely bounded. 
\end{Thm}

The coarse boundedness of \( H(M) \) implies the coarse boundedness of \( \mcg(M) \), so we will only work with \( H(M) \).
Before getting to the proof, we require two lemmas and a definition. 

\begin{Lem}
\label{lem:cb}
Let \( G \) be a topological group.
If for any open neighborhood \( V \) of the identity in \( G \) there exists a finite set \( F \subset G \) and \( k \in \bn \) such that \( G = (FV)^k \), then \( G \) is coarsely bounded. \qed
\end{Lem}

\begin{proof}
Let \( X \) be a metric space and suppose \( G \) acts on \( X \) continuously by isometries. 
Fix \( x \in X \) and let \( \vp \co G \to X \) be the orbit map \( g \mapsto g\cdot x \). 
Let \( B \) be the ball of radius 1 centered at \( x \).
Then, \( V = \vp^{-1}(B) \) is an open neighborhood of the identity in \( X \). 
By assumption, there exists a finite set \( F \subset G \) and \( k \in \bn \) such that \( G = (FV)^k \), and hence \( G\cdot x \) is bounded. 
\end{proof}

In the setting of Polish groups, the converse of Lemma~\ref{lem:cb} is true as well (see \cite[Proposition~2.7]{RosendalBook}), but we will not require this fact.

\begin{Def}
A Freudenthal subsurface \( \Delta \) of a 2-manifold \( M \) is \emph{maximal}\index{subsurface!Freudenthal!maximal} if \( \Delta \) is homeomorphic to \( M \) with an open disk removed. 
\end{Def}

Note that there are no half-spaces in a uniquely self-similar 2-manifold.
The above definition is meant to be a weaker condition on a Freudenthal subsurface that is meant to capture some key features of half-spaces. 
In particular, every half-space is maximal. 
Also note that every maximal Freudenthal subsurface in a self-similar 2-manifold is stable.

\begin{Lem}
\label{lem:displace}
Let \( M \) be a self-similar 2-manifold.
Suppose \( \Delta \subset M \) is a maximal Freudenthal subsurface and \( D \subset M \ssm \Delta^\mathrm{o} \) is a Freudenthal subsurface. 
Then,  there exists \( f \in H(M) \) and a maximal Freudenthal subsurface \( \Delta' \) contained in the interior of \( \Delta \) satisfying:
\begin{enumerate}[(i)]
\item \( f(D) \subset \Delta \),
\item \( D \subset f(\Delta) \),
\item \( \supp(f) \subset M \ssm \Delta' \), and
\end{enumerate}
Moreover, if \( Z \subset M \ssm (D \cup \Delta) \) is a Freudenthal subsurface, then \( f \) can be chosen to restrict to the identity on \( Z \).
\end{Lem}

\begin{proof}
As \( \Delta \) is homeomorphic to \( M \) with an open disk removed, there exists \( D' \subset \Delta \) homeomorphic to \( D \). 
Let \( \Delta' \) be a maximal Freudenthal subsurface contained in \( \Delta \ssm D' \).
Then we can find a simple path \( \gamma \) in \( M \ssm \Delta' \) connecting \( \partial D' \) and \( \partial D \) (if \( Z \subset M \ssm (D \cup \Delta) \) is a Freudenthal subsurface, then \( \gamma \) can be chosen to be disjoint from \( Z \)). 
Taking the closure of a small neighborhood of \( D \cup \gamma \cup D' \) yields a Freudenthal subsurface \( N \).
By the change of coordinates principle, there exists a homeomorphism \( f \co M \to M \) supported in \( N \) such that \( f(D) = D' \) and \( f(D') = D \).  
It is readily verified that \( f \) is the desired homeomorphism.  
\end{proof}

\begin{proof}[Proof of Theorem~\ref{thm:cb}]
Let \( H = H(M) \). 
Let \( V \) be an open neighborhood of the identity in \( H \). 
By possibly shrinking \( V \), we are guaranteed to find a maximal Freudenthal set \( \Delta \) such that \( M \ssm \Delta \) is non-orientable if \( M \) is non-orientable and such that every homeomorphism supported in \( \Delta \) is contained in \( V \). 
Let \( \Sigma = M \ssm \Delta^\mathrm{o} \). 
Let \( f \in H \) be the homeomorphism obtained by inputting \( \Delta \) into Lemma~\ref{lem:displace} with \( D = \Sigma \), and set \( F = \{f,f^{-1}\} \). 
We claim that \( H = (FV)^5 \), which implies that \( H \) is coarsely bounded by Lemma~\ref{lem:cb}.

Fix \( g \in H \). 
Note \( M = \Delta \cup  f (\Delta)\), and so \(  g ( \Delta) \cap  \Delta \) or \(  g( \Delta) \cap  f (\Delta) \) must contain a maximal Freudenthal subsurface \( \Delta_1 \). 
By shrinking \( \Delta_1 \), we may assume that if \( \Delta_1 \) intersects \( \Sigma \) (resp., \( f(\Sigma) \)), then \( \Delta_1 \) is contained in \( \Sigma \) (resp., \( f(\Sigma) \)). 
Let \( \vp \) be a \( f(\Sigma) \)-translation supported in \( \Delta \), as guaranteed by Lemma~\ref{lem:translation}, and note that \( \vp \in V \).
Then, up to post-composing \( g \) with \( \vp \circ f \) if \( \Delta_1 \subset \Sigma \), or post-composing \( g \) with \( \vp \) if \( \Delta_1 \subset f(\Sigma) \), we may assume that \( \Delta_1 \) is disjoint from \( \Sigma \cup f(\Sigma) \); in particular,  
\begin{equation}
\label{eq1}
\Delta_1 \subset g ( \Delta) \cap  \Delta \cap  f ( \Delta).
\end{equation}
Under this additional assumption, it is enough to show that \( g \in (FV)^4 \) as \( (\vp \circ f)^{-1} \in FV \). 

Apply Lemma~\ref{lem:displace} to \( \Delta_1 \) and \( f(\Sigma) \subset M \ssm \Delta_1^\mathrm{o} \) with \( Z = \Sigma \) to obtain a homeomorphism \( h \in H \) that restricts to the identity on \( \Sigma \) and satisfies \( h(f(\Sigma)) \subset \Delta_1 \subset g(\Delta) \). 
Note that \( h \in V \) as \( \supp(h) \subset \Delta \).

Apply the complement to the containment \( h(f(\Sigma)) \subset g(\Delta) \) to see that \( g(\Sigma) \subset h(f(\Delta)) \).
Moreover, \( \Sigma \subset f(\Delta) \) and  \( \Sigma \subset h(f(\Delta)) \), as \( \Sigma = h(\Sigma) \).
Hence,
\begin{equation} 
\label{eq2}
\Sigma \cup g(\Sigma) \subset h(f(\Delta)).
\end{equation}
As a consequence of \( h \) being given by Lemma~\ref{lem:displace}, there is a maximal Freudenthal subsurface \( \Delta_2 \) contained in \( \Delta_1 \) and fixed by \( h \).

\begin{figure}[t]
\centering
\includegraphics{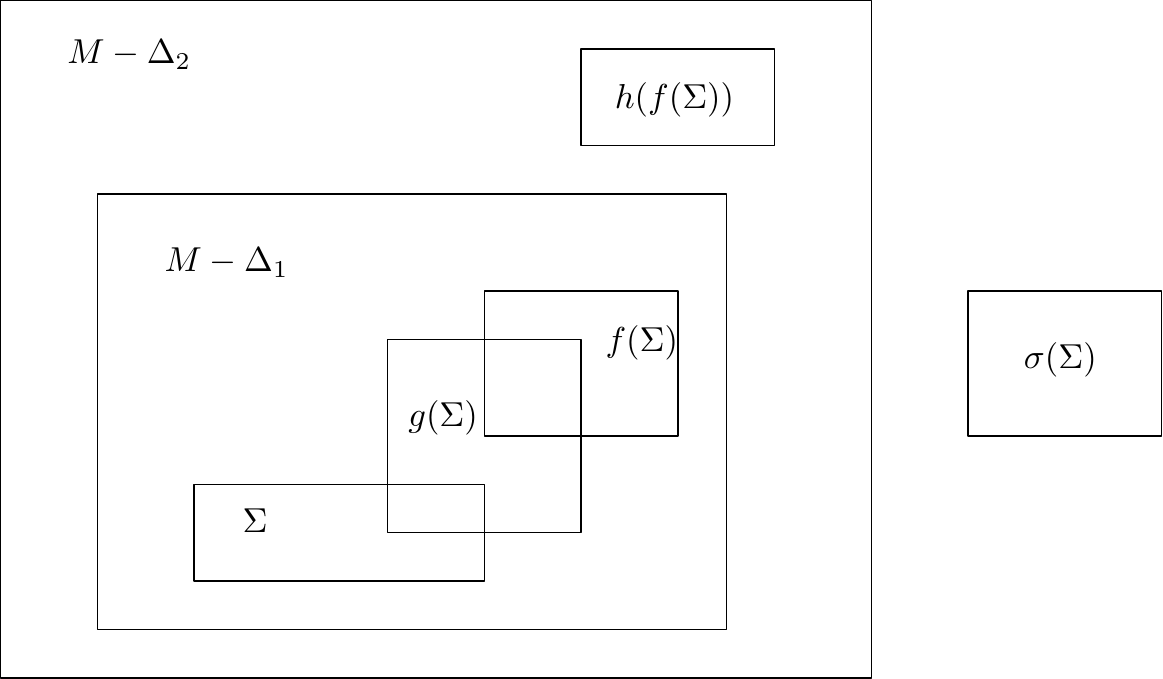}
\caption{A schematic of the sets in the Proof of Theorem~\ref{thm:cb} with \( M = \br^2 \).}
\label{fig:coarse}
\end{figure}

Using that \( h(f(\Sigma)) \) is disjoint from \( \Sigma \), we again apply Lemma~\ref{lem:displace}, this time to \( \Delta_2 \) and \( \Sigma \subset M \ssm \Delta_2^\mathrm{o} \) with \( Z = h(f(\Sigma)) \), to obtain a homeomorphism \( \sigma \co M \to M \) such that \( \sigma(\Sigma) \subset \Delta_2 \) and such that \( \sigma \) restricts to the identity on \( h (f(\Sigma)) \). 
From this, we know that \( h(f(\Sigma)) \) is disjoint from \( \Sigma \), so together with \eqref{eq1}, we have
\begin{equation}
\label{eq3}
 \sigma(\Sigma) \cap \left( h(f(\Sigma)) \cup \Sigma \cup g(\Sigma) \right)=\varnothing .
 \end{equation}
A schematic of all these sets and their intersection relations can be seen in Figure~\ref{fig:coarse}.

Applying the change of coordinates principle twice (first sending \( g(\Sigma) \) to \( \sigma(\Sigma) \) and then \( \sigma(\Sigma) \) to \( \Sigma \)), we find \( v \in H \) such that \( v(g(\Sigma)) = \Sigma \).
Moreover, from \eqref{eq2} and \eqref{eq3}, we have  that \( h(f(\Sigma)) \) is disjoint from \( \Sigma \cup g(\Sigma)\cup \sigma(\Sigma) \), and hence we may assume that  \( v \)  restricts to the identity on \( h(f(\Sigma)) \).
Additionally, after possibly post composing \( v \) with a double slide (to reverse the orientation of a neighborhood of \( \partial \Sigma \)) and  then post composing with a homeomorphism supported in \( \Sigma \), we may assume that \( v\circ g \) restricts to the identity on \( \Sigma \), and hence \( v\circ g \in V \). 
Observe that as \( v \) restricts to the identity on \( h(f(\Sigma)) \), we have  \( (hf)^{-1}v (hf) \) restricts to the identity \(  \Sigma \), and hence \( (hf)^{-1}v(hf) \in V \).
It follows that \( v \in (FV)^3 \) and \( g \in (FV)^4 \), as desired.
\end{proof}

\section{Rokhlin property}
A topological group has the \emph{Rokhlin property} if it contains an element whose conjugacy class is a dense set in the group. 
In this section, we will prove the following, which generalizes \cite[Theorem~4.1]{LanierMapping} and the reverse direction of \cite[Theorem~0.2]{HernandezConjugacy} (these are identical statements proved independently)  from mapping class groups to homeomorphism groups, and also accounts for non-orientable 2-manifolds:

\begin{Thm}
\label{thm:rokhlin}
Let \( M \) be a  self-similar 2-manifold.
If either \( M \) is the 2-sphere or \( M \) is uniquely self-similar, then \( H(M) \) has the Rokhlin property. 
\end{Thm}

We note again that the converse of Theorem~\ref{thm:rokhlin} is true, that is, the 2-sphere is the only perfectly self-similar 2-manifold whose (orientation-preserving) homeomorphism group has the Rokhlin property; we direct the interested reader to \cite{LanierMapping} and \cite{HernandezConjugacy}.

Before beginning to prove Theorem~\ref{thm:rokhlin}, we require the following characterization of the Rokhlin property. 
A topological group \( G \) has the \emph{joint embedding property}\index{joint embedding property}, or \emph{JEP}\index{JEP|see {joint embedding property}}, if given any two nonempty open sets \( U \) and \( V \) in \( G \), there exists \( g \in G \) such that \( U \cap V^g \neq \varnothing \), where \( V^g = \{ gvg^{-1} : v \in V \} \). 
In other words, the action of \( G \) on itself by conjugation is topologically transitive.
We will also use the notation \( v^g = gvg^{-1} \) for \( v,g \in G \). 
The following lemma is a standard result; for completeness, we include a short proof, which is given in \cite[Theorem~2.1]{KechrisTurbulence} and its following remark. 

\begin{Lem}
\label{lem:jep}
A Polish group has the JEP if and only if it has the Rokhlin property.
\end{Lem}

\begin{proof}
Let \( G \) be a Polish group and let \( \mathcal B \) be a countable basis for its topology. 
Then, for each \( U \in \mathcal B \), define \( D_U = \bigcup_{g\in G} U^g \).
If \( G \) has the JEP, then \( D_U \) is necessarily a dense open subset of \( G \).
As an application of the Baire Category Theorem, \( D= \bigcap_{U\in \mathcal B} D_U \) is dense, as it is the intersection of countably many dense open subsets. 
Then, the conjugacy class of an any element of \( D \) must be dense in \( G \):
Indeed, let \( h \in D \).  
If \( V \) is an open subset of \( G \), then there exists \( U \in \mathcal B \) such that \( U \subset V \), and, in turn, there exists \( g \in G \) such \( h \in U^g \); hence, there is a conjugate of \( h \) in \( V \). 

Conversely, if \( h \in G \) such that the conjugacy class of \( h \) is dense, then, for any open subsets \( U \) and \( V \) of \( G \), there exists \( g_1, g_2 \in G \) such that \( h^{g_1} \in U \) and \( h^{g_2} \in V \). 
In particular, \( h^{g_1} = (h^{g_2})^{g_3} \in U \cap V^{g_3} \), where \( g_3 = g_1g_2^{-1} \), and \( G \) has the JEP.
\end{proof}

We now proceed towards the proof of Theorem~\ref{thm:rokhlin}. 
The proofs that the (orientation-preserving) homeomorphism groups of the 2-sphere and of uniquely self-similar 2-manifolds have the Rokhlin property are similar in nature, but despite the general theme of giving unified proofs in this note, it is more natural in this case to separate out the 2-sphere.

\begin{Lem}
\label{lem:2sphere}
If \( U \) is an open subset of \( \Homeo(\bS^2) \), then there exists a closed disk \( \Delta \) in \( \bS^2 \) and a homeomorphism \( g \in U \) such that such that \( \supp(g) \subset \bS^2 \ssm \Delta^\mathrm{o} \) and such that  \( h \circ g \in U \) for all \( h \in \Homeo(\bS^2) \) supported in \( \Delta \).  
\end{Lem}

\begin{proof}
Fix \( f \in U \).
Let \( d \) be the standard metric on \( \bS^2 \), and let \( \rho_d \) be the metric on \( H(\bS^2) \) defined by \( \rho_d(f_1, f_2) = \max\{d(f_1(x), f_2(x)): x \in \bS^2 \} \) (see Corollary~\ref{cor:homeo_compactification}).
Fix \( \ep \in \br \) such that the \( (2\ep) \)-ball (with respect to \( \rho_d \)) in \( \Homeo(\bS^2) \)  centered at \( f \)  is contained in \( U \). 
As a corollary to the Hairy Ball Theorem, there exists \( z \in \bS^2 \) such that \( f(z) = z \). 
For \( r \in \br \), let \( B_r \) denote the closed ball of radius \( r \) in \( \bS^2 \) centered at \( z \). 
Fix \( \delta \in \br \) such that \( \delta < \ep \) and \( f(B_\delta) \subset B_{\ep}^\mathrm{o} \). 
Then, there exists \( f_1 \in \Homeo(\bS^2) \) supported in \( B_\ep \) such that \( f_1 \circ f \) restricts to the identity on \( B_\delta \). 
Set \( g = f_1 \circ f \) so \( \supp(g) \subset \bS^2 \ssm B_\delta^\mathrm{o} \). 

In particular, as \( f \circ g^{-1} = f_1^{-1} \) is supported in \( B_\ep \), we have that \( \rho_d(f,g)< 2\ep \) and \( g \in U  \). 
Observe that if \( h \in \Homeo(\bS^2) \) is supported in \( B_\delta \), then \( f \circ (h\circ g)^{-1} = f_1^{-1}  \circ h^{-1} \) is supported in \( B_\ep \); hence, \( \rho_d(f, h\circ g) < 2\ep \) and \( h\circ g \in U \). 
To finish, set \( \Delta = B_\delta \). 
\end{proof}

\begin{Thm}
\label{thm:2sphere}
\( \Homeo(\bS^2) \) has the Rokhlin property. 
\end{Thm}

\begin{proof}
Let \( U_1 \) and \( U_2 \) be open subsets of \( \Homeo(\bS^2) \). 
For \( i \in \{1,2\} \),  let  \( g_i \in U_i \) and \( \Delta_i  \subset \bS^2 \) be obtained by applying  Lemma~\ref{lem:2sphere} to \( U_i \). 
Fix  \( \sigma \in \Homeo(\bS^2) \) such that \( \sigma(\bS^2 \ssm \Delta_2^\mathrm{o}) \subset \Delta_1^\mathrm{o} \). 
Let \( g =   g_1 \circ g_2^\sigma \). 
Note that \( g_1 \) and \( g_2^\sigma \) commute as their  supports have disjoint interiors, so \( g = g_2^\sigma \circ g_1 \). 
It follows that \( g = g_2^\sigma \circ g_1 \in U_1 \) as \( \supp(g_2^\sigma) \subset \Delta_1 \).
Now, if \( h \in \Homeo(\bS^2) \) is supported in \( \sigma(\Delta_2) \), then \( \supp(\sigma^{-1}\circ h \circ \sigma ) \subset \Delta_2 \) and \( (\sigma^{-1}\circ h \circ \sigma)\circ g_2 \in U_2 \); in particular, \( h \circ g_2^\sigma \in U_2^\sigma \). 
Hence, as \( \supp(g_1) \subset \bS^2 \ssm \Delta_1^\mathrm{o} \subset \sigma(\Delta_2) \), we have \( g = g_1 \circ g_2^\sigma \in U_2^\sigma \). 
Therefore, \( U_1 \cap U_2^\sigma \neq \varnothing \) (as it contains \( g \)), implying \( \Homeo(\bS^2) \) has the JEP, and hence the Rokhlin property by Lemma~\ref{lem:jep}.
\end{proof}

Theorem~\ref{thm:2sphere} was first proved by Glasner--Weiss \cite[Theorem~3.4]{GlasnerTopological}. 
Their proof works for all even-dimensional spheres, and we simply note that the proof given above also works in all even dimensions (since the Harry Ball Theorem is true for even-dimensional spheres).
The issue in odd dimensions is that the antipodal map is orientation preserving and admits an open neighborhood in which every element is a fixed point free homeomorphism.  
In particular, the group of orientation-preserving homeomorphisms of an odd-dimensional sphere contains an open set consisting of fixed-point free homeomorphisms and an open set in which every element fixes a point, and since both of these properties are conjugacy invariants, the group cannot have a dense conjugacy class. 

Let us move to the uniquely self-similar case. 
We first show that the (normal) subgroup of \( H(M) \) consisting of homeomorphisms supported on the complement of a maximal Freudenthal subsurface in \( M \) is dense.
This will allow us to readily establish the JEP property.

\begin{Lem}
\label{lem:dense2}
Let \( M \) be a uniquely self-similar 2-manifold.
Let \( \Gamma \) be the subgroup of \( H(M) \) consisting of homeomorphisms that restrict to the identity on a maximal Freudenthal subsurface. 
Then, \( \Gamma \) is dense in \( H(M) \). 
\end{Lem}

\begin{proof}
Let \( H = H(M) \).  
It is enough to show that each set in a basis for \( H \) intersects \( \Gamma \) nontrivially. 
So, fix
\[
U = \bigcap_{i=1}^n U(K_i, W_i),
\]
where \( K_1, \ldots, K_n \subset M \) are compact  and \( W_1, \ldots, W_n \subset M \) are precompact and open. 

Choose a maximal Freudenthal subsurface \( \Delta_1 \) disjoint from \( K_i \) and the closure of \( W_i \) for each \( i \in\{1, \ldots, n\} \). 
Let \( \Sigma = M \ssm \Delta_1^\mathrm{o} \), and let \( f \in U \). 
If \( M \) is non-orientable, then by possibly enlarging \( \Sigma \) (and shrinking \( \Delta_1 \)), we may assume that \( \Sigma \) is not orientable. 
Let \( \Delta_2 \) be a maximal Freudenthal subsurface disjoint from \( \Sigma \cup f(\Sigma) \).
By shrinking \( \Delta_2 \), we may assume that the surface co-bounded by \( \partial \Sigma \) and \( \partial \Delta_2 \) is not orientable if \( M \) is not orientable. 
By Lemma~\ref{lem:displace}, there exists \( f_1 \in \Gamma \) supported in the complement of a maximal Freudenthal subsurface \( \Delta_3 \) such that \( f_1(\Sigma) \subset \Delta_2 \). 
Then, using the fact that \( f_1(\Sigma) \cap f(\Sigma) = \varnothing \), the change of coordinates principle yields a homeomorphism \( f_2 \in \Gamma \) supported in the complement of \( \Delta_3 \) such that \( f_2(f(\Sigma)) = f_1(\Sigma) \).
In particular, \( f^{-1}\circ f_2^{-1}\circ f_1(\Sigma) = \Sigma \).

\begin{figure}[t]
\centering
\includegraphics{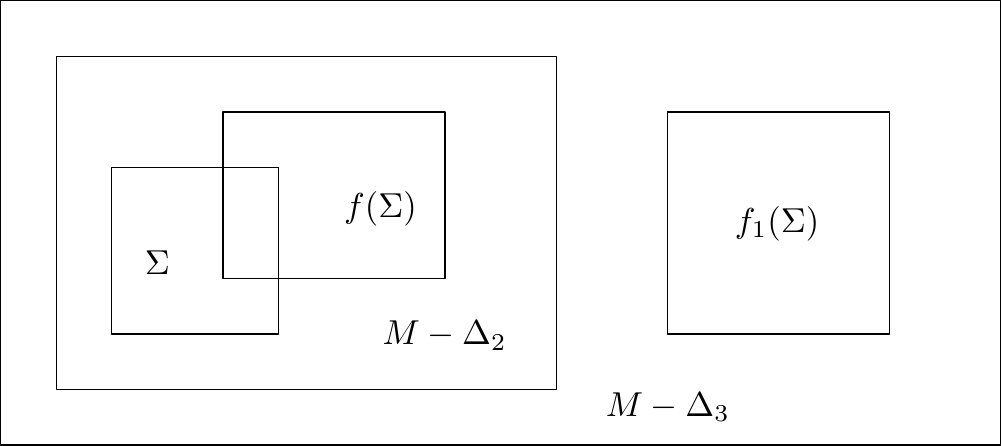}
\caption{A schematic of the subsurfaces in the proof of Lemma~\ref{lem:dense2} with \( M = \br^2 \).}
\end{figure}

If \( M \) is orientable, we may assume that \( f^{-1}\circ f_2^{-1}\circ f_1 \) restricts to the identity on \( \partial \Sigma \) since it preserves orientation; however, if \( M \) is non-orientable, it is possible that \( f^{-1}\circ f_2^{-1}\circ f_1 \) reverses the orientation of \( \partial \Sigma \).
If this is the case, then by construction, there exists a double slide \( s \) supported in \( M \ssm \Delta_3 \) preserving \( \Sigma \) and reversing the orientation of \( \partial \Sigma \) so that \(  f^{-1}\circ f_2^{-1}\circ f_1\circ s(\Sigma) \) restricts to the identity on \( \partial \Sigma \).
Up to replacing \( f_1 \) with \( f_1\circ s \), we can assume that  \( f^{-1}\circ f_2^{-1}\circ f_1 \) restricts to the identity on \( \partial \Sigma \). 
Therefore, there exists \( f_3, f_4 \in H \) such that \( f_3 \) is supported in \( \Sigma \),  \( f_4 \) is supported in \( M \ssm \Sigma^\mathrm{o} \), and \( f^{-1}\circ f_2^{-1}\circ f_1 = f_4 \circ f_3 \). 
It follows that \(    f_2^{-1}\circ f_1 \circ f_3^{-1} = f\circ f_4  \), and as \( f_4 \) restricts to the identity on \( \Sigma \), \( f \circ f_4(K_i) = f(K_i) \subset W_i \) for all \( i \in \{1, \ldots, n\} \); in particular, \( f_1^{-1}\circ f_2 \circ f_3^{-1} = f\circ f_4 \in U \). 
The result follows as \( f_1, f_2, f_3 \in \Gamma \). 
\end{proof}

Observe that in the case \( M \) is one-ended (i.e., \( M \) is homeomorphic to either the plane, the Loch Ness monster, or the non-orientable Loch Ness monster), then Lemma~\ref{lem:dense2} says that the subgroup of compactly supported homeomorphisms is dense in \( H(M) \). 
The case of the Loch Ness monster was established by Patel and the author in \cite{PatelAlgebraic} and was a key motivating example  for the results in \cite{AramayonaFirst} and \cite{LanierMapping}.

\begin{Thm}
If \( M \) is a uniquely self-similar 2-manifold, then \( H(M) \) has the Rokhlin property.
\end{Thm}

\begin{proof}
Let \( H = H(M) \).
As before, let \( \Gamma \) be the subgroup of \( H \) consisting of homeomorphisms that restrict to the identity on a maximal Freudenthal subsurface.
Let \( U_1 \) and \( U_2 \) be nonempty open subsets of \( H \). 
We must find \( g \in H \) such that \( U_1 \cap U_2^g \neq \varnothing \). 
Since both \( U_1 \) and \( U_2 \) must contain basic open sets, we may assume that for \( i \in \{1, 2\} \)
\[
U_i = \bigcap_{j=1}^{n_i} U(K^i_j, W^i_j),
\]
where each \( K_j^i \) is a compact subset of \( M \) and each \( W_j^i \) is a precompact open subset of \( M \). 
Moreover, by Lemma~\ref{lem:dense2}, there exists \( f_i \in \Gamma \) such that \( f_i \in U_i \). 

Choose a maximal Freudenthal subsurface \( \Delta \) disjoint from the  set
\[ \bigcup_{i=1}^2 \bigcup_{ j=1}^{n_i} \left(\supp(f_i) \cup K_j^i \cup \overline W_j^i\right).
\]  
Let \( \Sigma  = M \ssm \Delta^\mathrm{o} \).
By Lemma~\ref{lem:displace}, there exists \( g \in \Gamma \) such that \( g(\Sigma) \subset \Delta \). 
Now, \( f_1 \) is supported in \( \Sigma \) and \( f_2^g \) is supported in \( g(\Sigma) \); in particular, \( f_1 \) restricts to the identity on \( g(W_k^2) \) and \( f_2^g \) restricts to the identity on \( K_j^1 \) for all \( k \in \{1, \ldots, n_2\} \) and \( j \in \{1, \ldots, n_1\} \). 

Let \( f = f_1\circ f_2^g \).
We claim that \( f \in U_1 \cap U_2^g \). 
First observe that 
\[
f(K_j^1) = (f_1\circ f^g_2)(K_j^1) = f_1(K_j^1) \subset W_j^1
\]
for all \( j \in \{1, \ldots, n_1\} \), and hence \( f \in U_1 \). 
Second, observe that 
\[
f(g(K^2_j)) = (f_1\circ f_2^g)(g(K^2_j)) = f_1( g(f_2(K^2_j))) \subset f_1(g(W_j^2)) = g(W_j^2)
\]
for \( j \in \{1, \ldots, n_2\} \), and hence \( f \in U_2^g \).  
Thus, \( f \in U_1 \cap U_2^g \) and \( H \) has the JEP.
The result follows from Lemma~\ref{lem:jep}.
\end{proof}

\section{Automatic continuity}

Recall, from the introduction (Definition~\ref{def:perfectly_stable}), that a 2-manifold \( M \) is  \emph{perfectly tame} if  it is homeomorphic to the connected sum  of a finite-type 2-manifold and finitely many perfectly self-similar 2-manifolds, each of whose space of ends can be written as \( \mathscr P \cup \mathscr D \), where \( \mathscr P \) is perfect, \( \mathscr D \) is a discrete set consisting entirely of planar ends, and \( \partial \mathscr D = \mathscr P \).

Now that we have the notions of a stable set and a Freudenthal subsurface at hand, we note that the definition of a perfectly tame 2-manifold is made exactly so that the following lemma holds:

\begin{Lem}
\label{lem:def-tame}
If \( M \) is a perfectly tame 2-manifold, then
\begin{enumerate}[(i)]
\item
\( \sE(M) \) has a basis consisting of stable sets, and
\item
if \( \sU \) is a stable clopen subset of \( \sE(M) \), then either \( \mathcal S(\sU) \)  is perfect or  \( \sU \) is a singleton consisting of a single planar end.
\end{enumerate}
\end{Lem}

The converse of Lemma~\ref{lem:def-tame} is true, but we will not require this fact. 
Recall that (1) a subset of a group is \emph{countably syndetic} if the group is the union of countably many left translates of the set, and (2) a topological group \( G \) is \emph{Steinhaus} if there exists \( m \in \bn \) such that \( W^m \) contains an open neighborhood of the identity in \( G \) for any symmetric countably syndetic set \( W \) in \( G \).
As noted in the introduction, the definition of a perfectly tame 2-manifold is designed so that Mann's proof in \cite{MannAutomatic2} can be  extended to the class of perfectly tame 2-manifolds; in particular, we will establish the following theorem:

\begin{Thm}
\label{thm:main1}
The homeomorphism group of a perfectly tame 2-manifold is Steinhaus. 
\end{Thm}

Before beginning we require several lemmas from \cite{MannAutomatic2}.
The first is a standard result about Polish groups:

\begin{Lem}[{\cite[Lemma~3.1]{MannAutomatic2}}]
\label{lem:dense}
Let \( G \) be a Polish group and \( W \subset G \) a symmetric set.
If \( W \) is  countably syndetic, then there exists a neighborhood \( U \) of the identity in \( G \) such that \( W^2 \) is dense in \( U \). 
\end{Lem}

\begin{proof}
A left translate of \( W \) is meager if and only if \( W \) is meager, and hence, as \( G \) cannot be a countable union of meager sets, we must have that \( W \) is not meager.
Therefore, \( W \) is dense is some open subset of \( G \).
It follows that \( W^2 \) is dense in an open subset \( U \) of \( G \), and as \( W \) is symmetric, \( U \) is a neighborhood of the identity. 
\end{proof}

We will split Mann's \cite[Lemma~3.2]{MannAutomatic2} into two separate lemmas.
We say a family of subsets \( \{X_n\}_{n\in\bn} \) in a topological space is \emph{piece-wise convergent}\index{convergent family!piece-wise} if there exists \( k \in \bn \) such that for each \( n \in \bn \) there exist pairwise-disjoint subsets \( Y^n_1, Y^n_2, \ldots, Y^n_k \) such that \( X_n = Y^n_1 \cup Y^n_2 \cup \cdots \cup Y^n_k \) and \( \{Y^n_i\}_{n\in\bn} \) is convergent for every \( i \in \{1, \ldots, k \} \). 
Observe that given  a sequence \( \{f_n\}_{n\in\bn} \) of self-homeomorphisms of a topological space such that \( \{\supp(f_n)\}_{n\in\bn} \) is piece-wise convergent, the infinite product \( \prod_{n\in\bn} f_n \) exists and is a homeomorphism. 

\begin{Lem}
\label{lem:step1}
Let \( M \) be a manifold, and let \( \{\Delta_i\}_{i\in\bn} \) be a piece-wise convergent family of pairwise-disjoint closed subsets of \( M \). 
Suppose \( W \) is a symmetric countably syndetic subset of \( \Homeo(M) \) and  \( \{g_i\}_{i\in\bn} \subset \Homeo(M) \) such that \( \Homeo(M) = \bigcup_{i\in\bn} g_i W \). 
Then, there exists \( k \in \bn \) such that for each \( f \in \Homeo(M) \) with \( \supp(f) \subset \Delta_k \) there exists an element of \( W^2 \) agreeing with \( f \) on \( \Delta_k \) and supported in the closure of \( \bigcup_{i\in\bn} \Delta_i \). 
\end{Lem}

\begin{proof}
Let \( \Delta \) denote the closure of \(  \bigcup_{i\in\bn} \Delta_i \).
We claim that there exists \( k \in \bn \) such that  for each \( f \in \Homeo(M) \) with \( \supp(f) \subset \Delta_k \) there exists \( w_f \in g_kW \) supported in \( \Delta \) such that \( w_f|_{\Delta_k} = f|_{\Delta_k} \). 
If not, then, for each \( i\in \bn \), there exists \( f_i \in \Homeo(M) \) with \( \supp(f_i) \subset \Delta_i \) such that no element of \( g_iW \) supported in \( \Delta \) agrees with \( f \) on \( \Delta_i \). 
Let \( w = \prod_{i\in\bn} f_i \), and note that \( w \) is supported in \( \Delta \).
Then, by the assumption on \( W \), there exists \( k \in \bn \) such that \( w \in g_kW \).
But, \( w \) agrees with \( f_k \) on \( \Delta_k \), a contradiction. 

Now, let \( k \) be the natural number given by the previous paragraph, let  \( f \in \Homeo(M) \) with \( \supp(f) \subset \Delta_k \), and let \( w_\mathrm{id} \) and \( w_f \) be the homeomorphisms obtained above. 
Then, letting \( w = w_\mathrm{id}^{-1}\circ w_f \), we see that \( w \) agrees with \( f \) on \( \Delta_k \) and \( w \in (Wg_k^{-1})(g_kW) = W^2 \). 
\end{proof}

\begin{Lem}
\label{lem:8p}
Let \( M \) be a manifold, and let \( W \subset \Homeo(M) \) be a symmetric countably syndetic set.
Suppose \( \mathcal A \) is a family of closed subsets of \( M \)  with the following property: there exists a piecewise-convergent collection \( \{\Delta_i\}_{i\in\bn} \) of pairwise-disjoint closed sets such that, for each \( i \in \bn \),  there is a piece-wise convergent collection of pairwise-disjoint sets \( \{A_j\}_{j\in\bn} \) such that \( A_j \in \mathcal A \) and \( A_j \subset \Delta_i \) for each \( j \in \bn \). 
Then, there exists \( A \in \mathcal A \) such that \( [a,b] \in W^8\) whenever \( a,b \in \Homeo(M) \) with \( \supp(a), \supp(b) \subset A \).
\end{Lem}

\begin{proof}
Let \( k \in \bn \) be obtained by applying Lemma~\ref{lem:step1} to the sequence \( \{\Delta_i\}_{i\in\bn} \), and let \( \Delta = \Delta_k \).
Let \( \{A_j\}_{j\in\bn} \) be a convergent collection of pairwise-disjoint sets in \( \mathcal A \), each of which is contained in \( \Delta \). 
Apply Lemma~\ref{lem:step1} to \( \{ A_j\}_{j\in\bn} \) to obtain \( m \in \bn \)  such that for each \( a \in \Homeo(M) \) with \( \supp(a) \subset A_m \) there exists an element of \( W^2 \) agreeing with \( a \) on \(  A_m \).
Let \( A = A_m \). 

Let \( a, b \in \Homeo(M) \) be supported in \( A \). 
As \( A \subset \Delta \), \( a \) is supported in \( \Delta \) and hence there exists \( w_a \in W^2 \) supported in the closure of \( \bigcup_{i\in\bn} \Delta_i \) and whose restriction to \( \Delta \) (and hence to \( A \)) agrees with \( a \). 
Similarly, there exists \( w_b \in W^2 \) supported in the closure of \( \bigcup_{j\in\bn}  A_j \) and whose restriction to \( A \) agrees with \( b \).
As \( \supp(w_a) \cap \supp(w_b) \subset A \), one readily verifies that \( [a,b]=[w_a,w_b] \) and hence \( [a,b]\in W^8 \).  
\end{proof}

The final lemma we require is a tailored version of \cite[Proposition~3.3]{MannAutomatic2}, as \cite[Proposition~3.3]{MannAutomatic2} does not hold in the infinite-genus setting.
We will only offer a sketch of the proof of this lemma as the full details would take us too far afield.
Moreover, after a combinatorial setup, the proof will be identical to a part of the proof of Theorem~\ref{thm:main1}, and hence we postpone the presentation of the sketch of the proof.

\begin{Lem}
\label{lem:compact_support}
Let \( M \) be a 2-manifold, and let \( W \) be a symmetric countably syndetic subset of \( \Homeo(M) \). 
If \( \Sigma \) is a finite-type subsurface of \( M \), then there exists an open subset \( U \) of the identity in \( \Homeo(M) \) such that any element of \( U \) supported in \( \Sigma \) is an element of \( W^{36} \). 
\end{Lem}

For the proof of Theorem~\ref{thm:main1}, we will follow the proof presented by Mann \cite{MannAutomatic2} to the extent possible.
The main tool we lose is the Fragmentation Lemma as presented in \cite[Proposition~2.3]{MannAutomatic1} (see Theorem~\ref{thm:fragmentation}); however, Lemma~\ref{prop:fragmentation2} will suffice.
Another superficial difference to note is that we will dispense with using an auxillary metric on \( M \) and \( \Homeo(M) \) and instead work directly with the compact-open topology.

\begin{proof}[Proof of Theorem~\ref{thm:main1}]
If \( M \) is of finite type, then the theorem follows directly from Lemma~\ref{lem:compact_support} by setting \( \Sigma = M \), so we may assume that \( M \) is of infinite type. 
Let \( W \) be a countably syndetic subset of \( \Homeo(M) \). 
By Lemma~\ref{lem:dense}, \( W^2 \) is dense in some open neighborhood \( U_0 \) of the identity in \( \Homeo(M) \).
As a neighborhood of the identity, \( U_0 \) contains a set of the form \( U(K_1, V_1) \cap \cdots \cap U(K_r, V_r) \), where \( K_i \subset V_i \),  \( K_i \) is compact, and \( V_i \) is open and precompact for \( i \in \{1, \ldots, r\} \).
Fix a finite-type subsurface \( F \) of \( M \) such that the closure of \( V_1 \cup \cdots \cup V_r \) is contained in the interior of \( F \).
Using Lemma~\ref{lem:def-tame}, we can enlarge \( F \) so that each component of \( M \ssm F^\mathrm{o} \) is a non-compact dividing Freudenthal subsurface.
Note that every homeomorphism with support in the complement of \( F \) fixes each of the \( K_i \) and hence is in \( U_0 \). 

\begin{figure}
\centering
\includegraphics{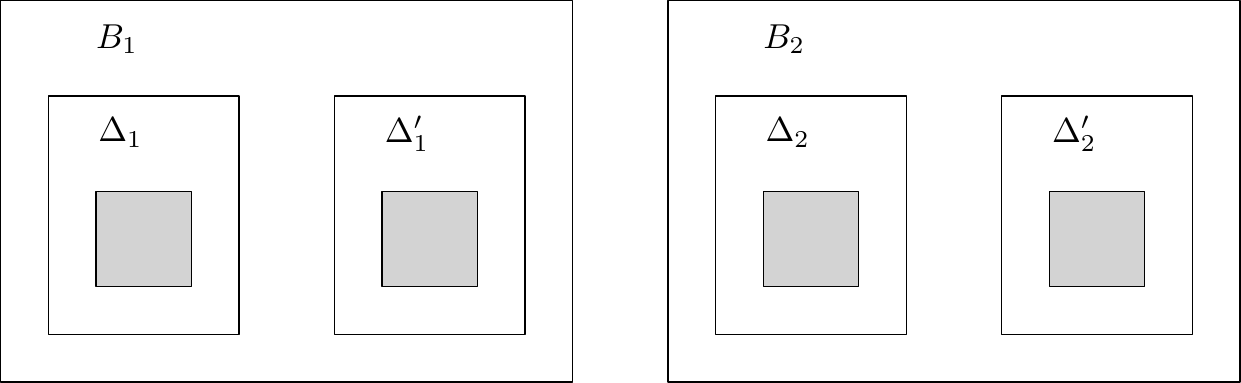}
\caption{A schematic of the subsurface \( N \), in the proof of Theorem~\ref{thm:main1}, realized as the complement of the shaded squares in the manifold \( M = \br^2 \). }
\label{fig:N}
\end{figure}

Let \( B_1, \ldots, B_k \) denote the components of \( M \ssm F^\mathrm{o} \).
For each \( i \in \{1, \ldots, k\} \), choose two disjoint stable Freudenthal subsurfaces \( \Delta_i , \Delta_i' \subset B_i^\mathrm{o} \) each of which is homeomorphic to \( B_i \)  and such that 
\( \widehat \Delta_i \cup \widehat \Delta_i' = \widehat B_i \). 
For each \( i \in \{1, \ldots, k\} \), let \( R_i \) and \( R_i' \) be open annular neighborhoods of \( \partial \Delta_i \) and \( \partial \Delta_i' \), respectively, such that \( \overline R_i \) and \( \overline R_i' \) are disjoint subsurfaces of \( M \), each of which is also disjoint from \( \partial B_i \). 
Let 
\[
N = F \cup \left[\bigcup_{i=1}^k(B_i \ssm (\Delta_i \cup \Delta_i'))\right] \cup \left[\bigcup_{j=1}^k (R_j \cup R_j')\right]
\]
(see Figure~\ref{fig:N}).

Now, let \( U_1 \) be the open neighborhood of the identity in \( \Homeo(M) \) obtained by setting \( \Sigma = \overline N \) in Lemma~\ref{lem:compact_support}, and let  \( U_2 \) be the open neighborhood of the identity obtained by applying Lemma~\ref{prop:fragmentation2} with the open cover \( \{N, \Delta_1^\mathrm{o}, \ldots, \Delta_k^\mathrm{o}, (\Delta_1')^\mathrm{o}, \ldots, (\Delta_k')^\mathrm{o}\} \).
Set
\(
U = U_1 \cap  U_2 .
\)
We now claim that \( U \subset W^{50} \). 

Fix \( f \in U \). 
Then, since \( f \in U_2 \), we have
\[
f = g_0\prod_{i=1}^k (g_ig_i'),
\] 
where \( \supp(g_0) \subset N, \supp(g_i) \subset \Delta_i^\mathrm{o} \), and \( \supp(g_i') \subset (\Delta_i')^\mathrm{o} \). 
It follows that \( (g_0)|_{K_j} = f|K_j \) for all \( j \in \{1, \ldots, r\} \), and hence \( g_0 \in U_1 \). 
Therefore, \( g_0 \in W^{36} \). 

Let \( \Delta = \bigcup \Delta_i \) and \( \Delta' = \bigcup \Delta_i' \), and let \( g = \prod_{i=1}^k g_i \) and \( g' = \prod_{i=1}^k g_i' \), so that \( \supp(g) \subset \Delta^\mathrm{o} \) and \( \supp(g') \subset (\Delta')^\mathrm{o} \). 
Then, \( f = g_0gg' \), and we  claim that \( g, g' \in W^{12} \), and hence \( U \subset W^{50} \). 
We will prove the claim for \( g \) (the case for \( g' \) is identical). 
 
Let \( \mathcal A \) be defined as follows: \( A \in \mathcal A \) if and only if \( A = A_1 \cup \cdots \cup A_k \) with \( A_i \subset \Delta_i^\mathrm{o} \) a stable Freudenthal subsurface satisfying \( \mathcal S(A_i) \cap \mathcal S(\Delta_i) \neq \varnothing \).
Next, choose a stable Freudenthal subsurface  \( D_i \) contained in \( \Delta_i^\mathrm{o} \) such that \( \mathcal S(D_i) \) is a proper subset of \( \mathcal S(\Delta_i) \). 
By Lemma~\ref{lem:translation}, there exists a convergent \( D_i \)-translation \( \vp_i \) supported in \( \Delta_i \). 
Again using Lemma~\ref{lem:translation} and the fact that \( \mathcal S(D_i) \) is perfect, we have that, for each \( n \in \bn \), the set \( \bigcup_{i=1}^k \vp_i^n(D_i) \) contains a piece-wise convergent collection of pairwise-disjoint sets \( \{A_j\}_{j\in\bn} \) contained in \( \mathcal A \).
Therefore, \( \mathcal A \) satisfies the conditions of Lemma~\ref{lem:8p}, and hence there exists \( A \in \mathcal A \) such that the commutator of any two homeomorphisms with support in \( A \) is contained in \( W^8 \).

\begin{figure}
\centering
\includegraphics[scale=0.75]{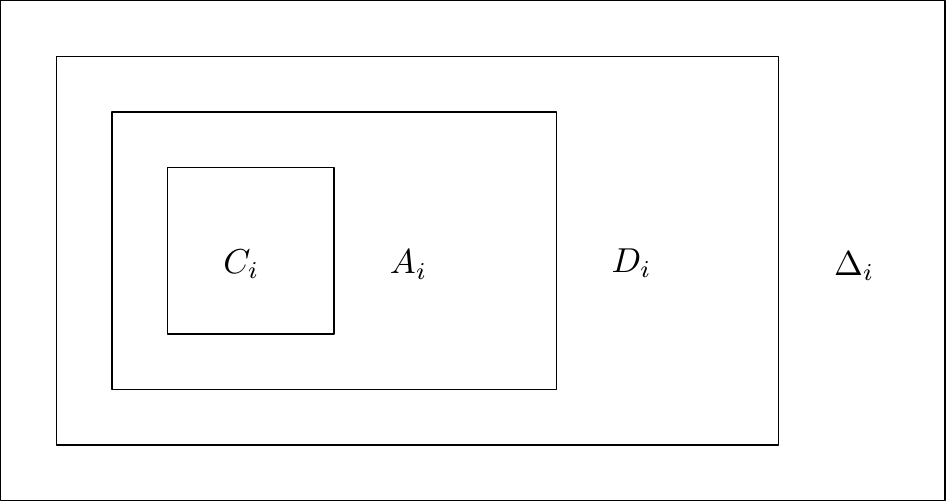}
\caption{The containment relations of the nested collection of stable Freudenthal subsurfaces.}
\end{figure}

For each \( i \in \{1,\ldots, k\} \), let \( A_i = A \cap \Delta_i \), so that \( A_i \subset D_i \). 
Now, let \( C_i \subset A_i^\mathrm{o} \) be a stable Freudenthal subsurface such that \( \mathcal S(C_i) \) is a proper subset of \( \mathcal S (A_i) \). 
Using that \( \mathcal S(C_i) \) is perfect and Lemma~\ref{lem:homogeneous}, there exists \( \sigma_i \in H(M) \) supported in \( B_i \) such that \( \sigma_i(\Delta_i) \subset C_i \) and \( \mathcal S(C_i) \ssm \sigma_i(\Delta_i) \neq \varnothing \). 
Let \( \sigma = \sigma_1 \circ \cdots \circ \sigma_k \). 
Observe that \( \sigma \in U_0 \), and as \( W^2 \) is dense in \( U_0 \), we may assume \( \sigma \in W^2 \).

Since \( \sigma \circ g_i \circ \sigma^{-1} \) is supported in \( C_i \), we can apply Anderson's Method (Proposition~\ref{prop:main2}), to express \( g_i \) as \( [\alpha_i, \beta_i] \) with \( \supp(\alpha_i), \supp(\beta_i) \subset A_i \). 
Now since the \( A_i \) are pairwise disjoint, we have that
\[
\sigma g \sigma^{-1} =  \prod_{i=1}^k \left(\sigma g_i \sigma^{-1} \right)  = \prod_{i=1}^k [\alpha_i,\beta_i] = \left[\prod_{i=1}^k \alpha_i, \prod_{i=1}^k \beta_i\right] = [\alpha, \beta],
\]
where \( \alpha = \prod_{i=1}^k \alpha_i \) and \( \beta = \prod_{i=1}^k \beta_i \) are supported in \( A \).
Hence,  \( \sigma\circ g \circ \sigma^{-1} \in W^8 \), and therefore \( g \in W^{12} \).
\end{proof} 

\begin{proof}[Sketch of proof of Lemma~\ref{lem:compact_support}]
The sketch presented here is an adaptation of \cite[Section~3]{MannAutomatic1} to work in an infinite-type setting. 
By Lemma~\ref{lem:dense}, there exists an open neighborhood \( U_0 \) of the identity in \( \Homeo(M) \) such that \( W^2 \) is dense in \( U_0 \). 
By shrinking \( U_0 \), we may assume that \( U_0 = U(K_1^0, V_1^0) \cap \cdots U(K_q^0, V_q^0) \) with \( K_i^0 \subset V_i^0 \), \( K_i^0 \) compact, and \( V_i^0 \) open and precompact.
Choose an open subset \( F \) of \( M \) such that \( \overline F \) is a finite-type subsurface of \( M \), each component of \( \partial F \) is separating, \( \Sigma \subset F \),  and \( \overline V_i^0 \subset F \) for each \( i \in\{1, \ldots, q\} \).

We can realize \( F \) as  \( N \ssm X \) for some compact 2-manifold \( N \) and some finite subset \( X \) of \( N \). 
Moreover, by Proposition~\ref{prop:compactification}, \( \Homeo(N,X) = \{ f \in \Homeo(N) : f(X)=X \} \) is topologically isomorphic to \( \Homeo(F) \). 
Now, triangulate \( N \) such that each element of \( X \) is a vertex of some triangle and such that if a triangle intersects \( K_i^0 \) nontrivially then it is contained in \( V_i^0 \).
Cover \( N \) by a finite collection of open disks \( D_1, \ldots, D_n \) such that if \( D_j \) intersects \( K_i^0 \) nontrivially then \( D_j \subset V_i^0 \) and such that the dual graph of the cover (i.e., the graph whose vertices correspond to each open set in the cover and where an edge denotes nontrivial intersection) is 3-colorable.  
Such a cover can be found by first covering the vertices of the triangulation by disjoint disks,  then covering the edges in an alternating fashion, and finishing with a single disk for each 2-cell, see Figure~\ref{fig:triangle}. 
Moreover, we may assume that if \( x \in X \) is disjoint from the closure of \( \Sigma \) in \( N \), then the disk containing \( x \) is disjoint from \( \Sigma \). 

\begin{figure}
\centering
\includegraphics[scale=2]{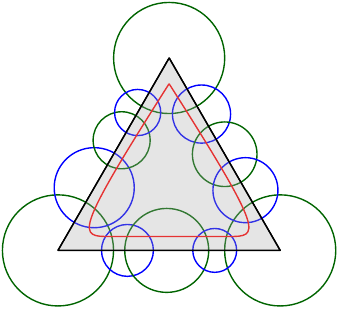}
\caption{A 3-colored open cover of a triangle by disks.}
\label{fig:triangle}
\end{figure}

Now, choose open disks \( B_1, \ldots, B_n \) such that \( \overline B_i \subset D_i \) and \( \{B_1, \ldots, B_n\} \) remains an open cover of \( M \).
Note that this new cover is also 3-colorable. 
By the Fragmentation Lemma, Theorem~\ref{thm:fragmentation}, or arguing as in Proposition~\ref{prop:fragmentation_original}, there exists an open neighborhood \( U_1' \) of the identity in \( \Homeo(N) \) such that every element of \( U_1' \) can be factored as \( g_1\circ g_2 \circ\cdots\circ g_n \) with \( \supp(g_i) \subset B_i \). 
Since \( |g_i \cap X| \leq 1 \), we can conclude that if \( g \in U_1' \cap \Homeo(N,X) \), then \( g= g_1 \circ \cdots g_n \) with \( \supp(g_i) \subset B_i \) and \( g_i \in \Homeo(N,X) \). 

Let \( U_1'' \) be the open neighborhood of the identity in \( \Homeo(F) \) corresponding to the open set \( U_1' \cap \Homeo(N,X) \) in \( \Homeo(N,X) \).
By shrinking \( U_1'' \), we may write \( U_1'' = U(K_1, V_1) \cap \cdots U(K_r, V_r) \) with \( K_i \subset V_i \), \( K_i \) compact, and \( V_i \) open and precompact. 
Viewing \( F \subset M \), each \( K_i \)  remains compact in \( M \) and each \( V_i \) remains open and precompact in \( M \).
We can then view \( U_1 = U(K_1, V_1) \cap \cdots U(K_r, V_r) \) as an open neighborhood of the identity in \( \Homeo(M) \). 

Set \( U = U_0 \cap U_1 \). 
Let \( g \in U \) such that \( \supp(g) \subset \Sigma \subset F \). 
Replacing \( D_i \) and \( B_i \) with \( D_i \ssm X \) and \( B_i \ssm X \), respectively, and viewing these sets in \( M \), we may write \( g = g_1 \circ \cdots \circ g_n \) with \( \supp(g_i) \subset B_i \). 
Note that by the choice of cover, we can ignore that fact that the closure of \( F \) has boundary, as \( g \) restricts to the identity near \( \partial F \).

There are two cases to consider: (1) \( B_i \) is an open disk and (2) \( B_i \) is an annulus. 
For the remainder of the argument, we will assume that only the first case occurs and direct the interested reader to \cite[Section~4]{MannAutomatic2} for the details of the second case. 
Using the 3-colorability of the cover, we can write \( g = f_1 \circ f_2 \circ f_3 \) where each \( f_k \) is a product of homeomorphisms supported on a disjoint union of the \( B_i \)'s.
Proceeding identically as in the last three paragraphs of the proof of Theorem~\ref{thm:main1} and using the fact that every homeomorphism of \( M \) supported in \( D_i \) is an element of \( U_0 \), we see that \( f_k \in W^{12} \).
Therefore, \( g \in W^{36} \) as desired. 
\end{proof}


\section{Commutator subgroups}

In this section we prove a slight generalization of a result in \cite{FieldStable}.
We note that the class that appears in the statement of the theorem appears differently than in \cite{FieldStable}, but they are in fact equivalent as shown in \cite[Lemma~A.1]{FieldStable}.
The proof uses similar constructions that appears in the proof of Theorem~\ref{thm:main1}.

\begin{Thm}
\label{thm:commutator}
Let \( M \) be a 2-manifold obtained by taking the connected sum of a finite-type 2-manifold and finitely many perfectly self-similar 2-manifolds.
Then, \( [\Homeo(M),\Homeo(M)] \) is open in \( \Homeo(M) \).
\end{Thm}

\begin{proof}
Let us write \( M = F \# M_1 \# \cdots \# M_k \), where \( F \) is a finite-type 2-manifold and \( M_j \) is a perfectly self-similar 2-manifold for \( j \in \{1, \ldots, k\} \). 
We can then find a finite-type subsurface \( \Sigma \) of \( M \) such that \( \Sigma \) is homeomorphic to \( F \) with \( k \) pairwise-disjoint open disks removed and such that \( M \ssm \Sigma \) has \( k \) components whose closures, labelled \( B_1, \ldots, B_k \), are pairwise disjoint and  such that \( B_j \) is homeomorphic to \( M_j \) with an open disk removed. 

We can then find disjoint dividing Freudenthal subsurfaces \( \Delta_j \) and \( \Delta_j' \) contained in \( B_j^\mathrm{o} \)  such that the complement of \( (\Delta_j \cup \Delta_j')^\mathrm{o} \) in \( B_j \) is compact. 
For each \( j \in \{1, \ldots, k\} \), fix an open annular neighborhood \( R_j \) of \( \partial \Delta_j \) and \( R_j' \) of \( \partial \Delta_j' \), each of which is disjoint from \( \partial B_j \), and such that \( \overline R_j \) and \( \overline R_j' \) are subsurfaces.
Let
\[
N = \Sigma \cup \left[\bigcup_{i=1}^k(B_i \ssm (\Delta_i \cup \Delta_i'))\right] \cup \left[\bigcup_{j=1}^k (R_j \cup R_j')\right].
\]

Let \( U \) be the open neighborhood of the identity in \( \Homeo(M) \) obtained by applying Lemma~\ref{prop:fragmentation2} with the open cover \( \{N, \Delta_1^\mathrm{o}, \ldots, \Delta_k^\mathrm{o}, (\Delta_1')^\mathrm{o}, \ldots, (\Delta_k')^\mathrm{o}\} \). 
Given \( g \in U \), we can write \( g = g_0 \left(\prod_{i=1}^k g_ig_i'\right) \) with \( g_0 \in [\Homeo(M), \Homeo(M)] \), \( \supp(g_i) \subset \Delta_i^\mathrm{o} \), and \( \supp(g_i') \subset (\Delta_i')^\mathrm{o} \) for all \( i \in \{1, \ldots, k\} \). 
By Proposition~\ref{prop:main2}, \( g_i \) and \( g_i' \) can be expressed as commutators of homeomorphisms supported in \( B_i \) for each \( i \in \{1, \ldots, k\} \).
Therefore, \( U \subset [\Homeo(M), \Homeo(M)] \), and hence \( [\Homeo(M), \Homeo(M)] \) is open as it can be written as union of translates of \( U \). 
\end{proof}

\clearpage
\bibliographystyle{amsplain}
\bibliography{references}

\end{document}